\documentclass[a4paper,10pt,leqno]{amsart}
        \title[Topological $K$-theory]
        {Topological $K$-theory of the group $C^*$-algebra of a semi-direct product $\IZ^n \rtimes \IZ/m$ for a free conjugation action}
       \author{Martin Langer}
       \author{Wolfgang L\"uck}
        \address{Rheinische Wilhelms-Universit\"at Bonn\\
               Mathematisches Institut\\
               Endenicher Allee 60, 53115 Bonn, Germany}
          \email{martinlanger@yahoo.com}
          \email{wolfgang.lueck@him.uni-bonn.de}
      \urladdr{http://www.math.uni-bonn.de/people/mlanger}
      \urladdr{http://www.him.uni-bonn.de/lueck}
         \date{September, 5th, 2011}

\usepackage{pdfsync}
\usepackage{calc}
\usepackage{enumerate,amssymb}
\usepackage{color}
\usepackage[arrow,curve,matrix,tips,2cell]{xy}
  \SelectTips{eu}{10} \UseTips
  \UseAllTwocells
\usepackage{tikz}
\usepackage{pdfcolmk}

\DeclareMathAlphabet{\matheurm}{U}{eur}{m}{n}

\newcommand{\eub}[1]{\underline{E}#1}

\newcommand{\bub}[1]{\underline{B}#1}

\DeclareMathOperator{\all}{all}
\DeclareMathOperator{\alt}{alt}
\DeclareMathOperator{\asmb}{asmb}

\DeclareMathOperator{\ch}{ch}

\DeclareMathOperator{\cok}{cok}

\DeclareMathOperator{\ev}{ev}

\DeclareMathOperator{\id}{id}
\DeclareMathOperator{\im}{im}
\DeclareMathOperator{\ind}{ind}

\DeclareMathOperator{\kg}{kg}
\DeclareMathOperator{\map}{map}

\DeclareMathOperator{\odd}{odd}

\DeclareMathOperator{\perm}{perm}
\DeclareMathOperator{\pt}{pt}
\DeclareMathOperator{\pr}{pr}
\DeclareMathOperator{\quot}{quot}
\DeclareMathOperator{\res}{res}
\DeclareMathOperator{\rk}{rk}

\DeclareMathOperator{\tors}{tors}

\newcommand{\calfin}{{\mathcal{F}\text{in}}}

\newcommand{\caltr}{\mathcal{T}\!\text{r}}

  \newcommand{\IC}{\mathbb{C}}

  \newcommand{\II}{\mathbb{I}}

  \newcommand{\IQ}{\mathbb{Q}}
  \newcommand{\IR}{\mathbb{R}}

  \newcommand{\IZ}{\mathbb{Z}}

  \newcommand{\calc}{\mathcal{C}}

  \newcommand{\calm}{\mathcal{M}}

\newcounter{commentcounter}

\theoremstyle{plain}
\newtheorem{theorem}{Theorem}[section]

\newtheorem{lemma}[theorem]{Lemma}

\newtheorem{corollary}[theorem]{Corollary}

\newtheorem*{theorem*}{Theorem}

\theoremstyle{definition}

\newtheorem{example}[theorem]{Example}

\newtheorem{remark}[theorem]{Remark}
\newtheorem{notation}[theorem]{Notation}

\theoremstyle{remark}

\makeatletter\let\c@equation=\c@theorem\makeatother

\newcommand{\version}[1]                       
{\begin{center} last edited on #1\\
last compiled on \today \\
name of texfile: \jobname
\end{center}
}

\newcommand{\xycomsquare}[8]                      
{\xymatrix{#1 \ar[r]^{#2} \ar[d]^{#4} &
#3 \ar[d]^{#5}  \\
#6\ar[r]^{#7} &
#8
}
}

\newcommand{\xycomsquareminus}[8]                      
{\xymatrix{#1 \ar[r]^-{#2} \ar[d]^-{#4} &
#3 \ar[d]^-{#5}  \\
#6\ar[r]^-{#7} &
#8
}
}

\newcommand{\triv}{\caltr}   

\begin{document}

\begin{abstract}
We compute the topological $K$-theory of the group $C^*$-algebra $C^*_r(\Gamma)$ for a group extension
$1 \to \IZ^n \to   \Gamma \to  \IZ/m \to 1$ provided that the conjugation action of $\IZ/m$ on $\IZ^n$ is free outside the origin.
\end{abstract}

\maketitle


\typeout{--------------------   Introduction --------------------------}

\section*{Introduction}

Throughout this paper let $1 \to \IZ^n \to \Gamma \to \IZ/m \to 1$ be a group extension such that
conjugation action of $\IZ/m$ on $\IZ^n$ is free outside the origin.
Our main goal is to compute the topological $K$-theory of the group $C^*$-algebra $C^*_r(\Gamma)$. This
generalizes results of Davis-L\"uck~\cite{Davis-Lueck(2010)}, where $m$ was
assumed to be a prime.  Except ideas from that paper, the proof of a Conjecture
due to Adem-Ge-Pan-Petrosyan in Langer-L\"uck~\cite{Langer-Lueck(2011_homology)} is a key
ingredient. The calculation and its result are surprisingly complicated. It will
play an important role in a forthcoming paper by Li-L\"uck~\cite{Li-Lueck(2011)}.
There the computation of the topological $K$-theory of a $C^*$-algebra
associated to the ring of integers in an algebraic number field will be carried
out in general, thus generalizing the work of Cuntz and
Li~\cite{Cuntz-Li(2009integers)} who had to assume that $+1$ and $-1$ are the
only roots of unity.


\subsection{Main Result}
\label{subsec:main_result}

Our main result is the following theorem. In the sequel $C^*_r(G)$ is the
\emph{reduced group $C^*$-algebra} of a group $G$.  We denote by $\eub{G}$ the
\emph{classifying space of proper actions} of a group $G$ and by $\bub{G}$ its
quotient space $G\backslash \eub{G}$. Let $\widehat{H}^*(G;M)$ be the \emph{Tate
  cohomology} of a group $G$ with coefficients in a $\IZ G$-module $M$. Denote by $\Lambda^i\IZ^n$
the $i$-th exterior power.

\begin{theorem}[Computation of the topological $K$-theory]
  \label{the:Topological_K-theory_for_Zr_rtimes_Z/n}
  Consider the  extension of groups
  $1 \to \IZ^n \to   \Gamma \to  \IZ/m \to 1$ such  that the 
  conjugation action  of $\IZ/m$ on
  $\IZ^n$ is free outside the origin  $0 \in \IZ^n$.
  Let $\calm$ be the set of conjugacy  classes  of   maximal finite subgroups of $\Gamma$.

\begin{enumerate}

\item \label{the:Topological_K-theory_for_Zr_rtimes_Z/n:K1_abstract}
  We obtain an isomorphism
\[
\omega_1\colon K_1(C^*_r(\Gamma)) \xrightarrow{\cong} K_1(\bub{\Gamma}).
\]
Restriction with the inclusion $k \colon \IZ^n \to \Gamma$ induces an
isomorphism
\[
k^* \colon K_1(C^*_r(\Gamma)) \xrightarrow{\cong} K_1(C^*_r(\IZ^n))^{\IZ/m}.
\]
Induction with the inclusion $k$ yields a homomorphism
\[
\overline{k_*} \colon \IZ \otimes_{\IZ[\IZ/m]} K_1(C^*_r(\IZ^n))
\to K_1(C^*_r(\Gamma)).
\]
It fits into an exact sequence
\[
0 \to \widehat{H}^{-1}(\IZ/m,K_1(C^*_r(\IZ^n))) \to
 \IZ \otimes_{\IZ[\IZ/m]}  K_1(C^*_r(\IZ^n))  
\xrightarrow{\overline{k_*}}  K_1(C^*_r(\Gamma)) \to 0.
\]
In particular $\overline{k_*}$
is surjective and its  kernel is annihilated by multiplication
with $m$;

\item \label{the:Topological_K-theory_for_Zr_rtimes_Z/n:K0_abstract}
  There is an exact sequence
\[
0 \to \bigoplus_{(M) \in \calm} \widetilde{R}_{\IC}(M) 
\xrightarrow{\bigoplus_{(M) \in \calm} i_M} K_0(C^*_r(\Gamma))
\xrightarrow{\omega_0} K_0(\bub{\Gamma}) \to 0,
\]
 where
$\widetilde{R}_{\IC}(M)$ is the kernel of the map $R_{\IC}(M) \to \IZ$
sending the class $[V]$ of a complex $M$-representation $V$ to
$\dim_{\IC}(\IC \otimes_{\IC M} V)$ and the map $i_M$ comes from the
inclusion $M \to \Gamma$ and the identification $R_{\IC}(M) = K_0(C^*_r(M))$.

We obtain a homomorphism
\[
\overline{k_*} \oplus \bigoplus_{(M) \in \calm} i_M \colon
 \IZ \otimes_{\IZ[\IZ/m]}  K_0(C^*_r(\IZ^n)) \oplus \bigoplus_{(M) \in \calm}
\widetilde{R}_{\IC}(M) \to K_0(C^*_r(\Gamma)).
\]
It is injective. It is bijective after inverting $m$;

\item \label{the:Topological_K-theory_for_Zr_rtimes_Z/n:explicite}
We have
\[
K_i(C^*_r(\Gamma)) \cong \IZ^{s_i}
\]
where
\[
s_i = 
\begin{cases}
\bigl(\sum_{(M) \in \calm} (|M|-1)\bigr) 
+ \sum_{l \in \IZ} \rk_{\IZ}\bigl((\Lambda^{2l} \IZ^n)^{\IZ/m}\bigr)
& \text{if}\; i \; \text{even};
\\
\sum_{l \in \IZ} \rk_{\IZ}\bigl((\Lambda^{2l+1} \IZ^n)^{\IZ/m}\bigr)
& \text{if}\; i \; \text{odd};
\end{cases}
\]

\item \label{the:Topological_K-theory_for_Zr_rtimes_Z/n:non-torsion:n_even}
If $m$ is even, then $s_1 = 0$ and
\[
K_1(C^*_r(\Gamma)) \cong \{0\}.
\]

\end{enumerate}
\end{theorem}

Another interesting result is Theorem~\ref{the:Cohomology_of_BGamma_and_bub(Gamma)}
where we will address the cohomology of $\Gamma$ and  of the associated toroidal quotient $\bub{\Gamma} = \Gamma \backslash\IR^n$.


\subsection{Organization of the paper}
\label{subsec:Organization_of_the_paper}

We will compute $K_*(C^*_r(G))$ in a more general setting in
Section~\ref{sec:Groups_satisfying_condition_(M)}, where we consider groups $G$
for which each non-trivial finite subgroup is contained in a unique maximal
finite subgroup.  In Section~\ref{sec:Groups_satisfying_condition_(NM)} we
consider the special case where we additionally assume that the normalizer of
every maximal finite subgroup is the maximal finite subgroup itself.  The groups
$\Gamma$ appearing in Theorem~\ref{the:Topological_K-theory_for_Zr_rtimes_Z/n} 
will satisfy this assumption.

In Sections~\ref{sec:The_cohomology_of_classifying_spaces} and~\ref{sec:Topological_K-theory_of_classifying_spaces} we deal with the
cohomology and the topological $K$-theory of the spaces $B\Gamma$ and $\bub{\Gamma}$, and, finally
complete the proof of Theorem~\ref{the:Topological_K-theory_for_Zr_rtimes_Z/n}.

The rest of the paper is devoted to the computation of the numbers $s_i$ 
appearing in Theorem~\ref{the:Topological_K-theory_for_Zr_rtimes_Z/n}. 
Recall that $\IZ^n$ becomes a $\IZ[\IZ/m]$-module by the conjugation action. Sometimes we write
$\IZ^n_{\rho}$ instead of $\IZ^n$ to emphasize the $\IZ[\IZ/m]$-module structure. Notice that the numbers $s_i$
are determined by the set $\calm$ of conjugacy classes of maximal finite subgroups of $\Gamma$
and the two numbers  $\sum_{l \ge 0} \, \rk_{\IZ}((\Lambda^l \IZ^n_{\rho})^{\IZ/m})$ 
and $\sum_{l \ge 0} (-1)^l  \rk_{\IZ}((\Lambda^l \IZ^n_{\rho})^{\IZ/m})$.

Section~\ref{sec:Conjugacy_classes_of_finite_subgroups_and_cohomology} is 
devoted to compute the partial ordered set  of conjugacy classes of finite subgroups
of $\Gamma$, directed by subconjugation, in terms of group cohomology. This yields also a computation
of the set $\calm$.

In Section~\ref{sec:The_prime_power_case} we will compute $\calm$ and 
the numbers $\sum_{l \ge 0} \, \rk_{\IZ}((\Lambda^l \IZ^n_{\rho})^{\IZ/m})$ in the case, where $m$ is a prime power. 

The general case is
treated in Section~\ref{sec:The_general_product_case}.  The calculation of
$\calm$ is explicit. The numbers $\sum_{l \ge 0} \, \rk_{\IZ}((\Lambda^l  \IZ^n_{\rho})^{\IZ/m})$ will be computed explicitly provided
that $m$ is even. For $m$ odd our methods yield  at least a recipe for a case by case computation, the problem is to determine
$\sum_{l \ge 0} \, \rk_{\IZ}((\Lambda^l \IZ^n_{\rho})^{\IZ/m})$. 
 Since the roots of unity in an algebraic number field
is always a finite cyclic group of even order, the case, where $m$ is even, is the most interesting for us.

In Section~\ref{sec:Equivariant_Euler_characteristics} we compute the equivariant
$\IZ/m$-Euler characteristic of $\IZ^n \backslash \eub{\Gamma}$ which takes
values in the Burnside ring of $\IZ/m$. This will determine explicitly the numbers
$\sum_{l \ge 0} (-1)^l \cdot  \rk_{\IZ}((\Lambda^l  \IZ^n_{\rho})^{\IZ/m})$.

The actual answers to the computations of the set $\calm$ and the numbers $s_i$ are interesting but also very complicated. 
As an illustration we state and explain some examples already here in the introduction. 
The group $\Gamma$ and the numbers $m$ and $n$ are the ones appearing
in Theorem~\ref{the:Topological_K-theory_for_Zr_rtimes_Z/n}. Here and in the sequel we will use the convention
that a sum of real numbers indexed by the empty set is understood to be zero, e.g.,
$\sum_{i=2}^1 a_i$ is defined to be zero.


\subsection{$m$ is a prime}
\label{subsec:m_is_prime}

Suppose that $m = p$ for a prime number $p$.
We have $\IZ^n_{\rho} \otimes_{\IZ} \IQ \cong_{\IQ[\IZ/p]} \IQ[\zeta_p]^k$
for $\zeta_p = \exp(2\pi i/p)$ and some natural number $k$.
The natural number $k$ is determined by the property $n = (p-1) \cdot k$.
All  non-trivial finite subgroups of $\Gamma$ are cyclic of order $p$ and are maximal finite. 
We obtain
from Theorem~\ref{the:prime_power_case}, Lemma~\ref{lem:sum_rk_Z(Lambda_j_L_G)} and
Theorem~\ref{the:a_priori_estimates}~\ref{the:a_priori_estimates:G/backslashunderlineEG}
\begin{eqnarray*}
|\calm| & = & p^k;
\\
\sum_{(M) \in \calm} (|M|-1)
& = & 
p^k \cdot (p-1);
\\
\sum_{l \in \IZ} \rk_{\IZ}\bigl((\Lambda^{l} L)^{\IZ/m}\bigr) 
& = & 
\begin{cases}
1 + \frac{2^{n}-1}{p} & p \not= 2;
\\
2^{k-1} & p = 2;
\end{cases}
\\
\sum_{l \in \IZ} (-1)^l \cdot \rk_{\IZ}\bigl((\Lambda^{l} \IZ^n)^{\IZ/m}\bigr)
& = & 
p^{k-1} \cdot(p-1).
\end{eqnarray*}
This implies (and is consistent with~\cite{Davis-Lueck(2010)})
\begin{eqnarray*}
s_i & = & 
\begin{cases}
 p^k \cdot (p-1) + \frac{2^{n}+ p -1}{2p} + \frac{p^{k-1} \cdot (p-1)}{2} & p \not= 2 \; \text{and} \; i \; \text{even};
\\
\frac{2^{n}+ p -1}{2p} - \frac{p^{k-1} \cdot (p-1)}{2}  & p \not= 2 \; \text{and} \: i \; \text{odd};
\\
3 \cdot 2^{k-1} & p = 2 \; \text{and} \; i \; \text{even};
\\
0 & p = 2 \; \text{and} \; i \; \text{odd}.
\end{cases}
\end{eqnarray*}


\subsection{$m$ is a prime power $p^r$ for $r \ge 2$}
\label{subsec:m_prime_power}
Next we consider the case, where $m$ is a prime power, let us say $m = p^r$.
Since we have treated the case
$r = 1$ already in Example~\ref{subsec:m_is_prime}, we will assume in the sequel $r \ge2$.
There exists precisely one natural number $k$ satisfying $n= (p-1) \cdot p^{r-1} \cdot k$.
We obtain from  Theorem~\ref{the:prime_power_case}
and Remark~\ref{rem:expressing_sums_involving_call(M)}
\begin{eqnarray*}
\sum_{(M) \in \calm} (|M|-1)
& = & 
\sum_{j = 1}^{r}  \sum_{(M) \in \calm(G_j)} (p^j -1)
\\
& = & 
p^k \cdot (p^r -1) + \sum_{j = 1}^{r-1} \bigl(p^{kp^{r-j}-r +j} - p^{kp^{r-j -1}-r + j}\bigr) \cdot (p^j -1). 
\end{eqnarray*}

We get from Lemma~\ref{lem:sum_rk_Z(Lambda_j_L_G)}
\begin{eqnarray*}
\sum_{l \ge 0} \rk_{\IZ}(\Lambda^l L)^{\IZ/m})
& = & 
1 + \frac{2^{n}-1}{p^r} \quad  \text{if}\; p \not= 2.
\end{eqnarray*}
If $p = 2$,   Lemma~\ref{lem:sum_rk_Z(Lambda_j_L_G)} yields
\begin{eqnarray*}
\sum_{l \ge 0} \rk_{\IZ}(\Lambda^l L)^{\IZ/m}
& = & 
\begin{cases}
 2^{2k -2} + 2^{k-1}  & \text{if}\; r = 2;
\\
2^{k-1}   +  2^{k2^{r-2} -r+1} + 2^{k \cdot 2^{r-1}-r}  +  \sum_{i=3}^{r-1}  2^{k2^{r-i}-r+i-1}  & \text{if}\; r \ge 3.
\end{cases}
\end{eqnarray*}
We get from Theorem~\ref{the:prime_power_case},
Theorem~\ref{the:a_priori_estimates}~\ref{the:a_priori_estimates:G/backslashunderlineEG}
and Remark~\ref{rem:expressing_sums_involving_call(M)}
\begin{eqnarray*}
\sum_{l \ge 0} (-1)^l \cdot \rk_{\IZ}(\Lambda^l L)^{\IZ/m})
& = &
p^k \cdot \frac{(p^r -1)}{p^r} + \sum_{j = 1}^{r-1} \bigl(p^{kp^{r-j}-r+j} - p^{kp^{r-j -1}-r+j}\bigr)\cdot \frac{p^j -1}{p^j} 
\\
& = &
p^k - p^{k-r} + \sum_{j = 1}^{r-1} \bigl(p^{kp^{r-j}-r} - p^{kp^{r-j -1}-r}\bigr) \cdot (p^j -1).
\end{eqnarray*} 
We leave it to the reader to combine these results to determine the numbers $s_i$.

\begin{example} [$m = 4$] \label{exa:m_is_4}
We get for $m = 4$, i.e., $p = 2$ and $r = 2$
\[
s_i = 
\begin{cases}
3 \cdot 2^{2k-2} + 3 \cdot 2^k & i \; \text{even};
\\
0 &  i \; \text{odd}.
\end{cases}
\]
\end{example} 
\begin{example} [$m = 9$] \label{exa:m_is_9}
We get for $m = 9$, i.e., $p = 3$ and $r = 2$
\[
s_i = 
\begin{cases}
\frac{64^k +8}{18} +  23 \cdot 3^{k-1}  + 7 \cdot 3^{3k-2}  & i \; \text{even};
\\
\frac{64^k +8}{18} - 3^{k-1} - 3^{3k-2} &  i \; \text{odd}.
\end{cases}
\]
\end{example}


\subsection{$m$ is square-free and  even}
\label{subsec:m_is_square-free_and-even}
Next  we consider the case, where $m$ is square-free and even.
The case $m = 2$ has already been treated in Example~\ref{subsec:m_is_prime}.
Hence we will assume in the sequel that $m = p_1 \cdot p_2 \cdot \cdots \cdot p_s$ for pairwise distinct prime numbers
$p_1,p_2, \ldots, p_s$ with $p_1 = 2$ and $s \ge 2$.
Then we get from Example~\ref{exa:M_square-free} and 
 Remark~\ref{rem:expressing_sums_involving_call(M)}
\begin{eqnarray*}
\sum_{(M) \in \calm} (|M|-1)
& = & 
 m-1 +\sum_{j = 1}^s \frac{p_j\cdot (p_j-1) \cdot (p_j^{n/(p_j-1)} -1)}{m}.
\end{eqnarray*}
From Lemma~\ref{lem:m_even_and_Lambda_odd} we get 
$\bigl(\Lambda^{2l+1} L\bigr)^{\IZ/m}  =  \{0\}$
for every $l \ge 0$. 
We conclude from  Example~\ref{exa:M_square-free},
Theorem~\ref{the:a_priori_estimates}~\ref{the:a_priori_estimates:G/backslashunderlineEG}
and Remark~\ref{rem:expressing_sums_involving_call(M)}
\[
\sum_{l \in \IZ} (-1)^l \cdot \rk_{\IZ}\bigl((\Lambda^{l} L)^{\IZ/m}\bigr) 
=  \frac{m-1}{m}  +\sum_{j = 1}^s \frac{(p_j-1) \cdot (p_j^{n/(p_j-1)} -1)}{m}.
\]
Hence we get for $i$ even
\begin{eqnarray*}
s_i  
& = & 
 m-1 +\sum_{j = 1}^s \frac{p_j \cdot (p_j-1) \cdot (p_j^{n/(p_j-1)} -1)}{m}
+ \frac{m-1}{m}  +\sum_{j = 1}^s \frac{(p_j-1) \cdot (p_j^{n/(p_j-1)} -1)}{m}
\\
& = & 
m + \frac{-m + (m-1) + \sum_{j = 1}^s \left(p_j\cdot (p_j-1) \cdot (p_j^{n/(p_j-1)} -1) + (p_j-1) \cdot (p_j^{n/(p_j-1)} -1)\right)}{m}
\\
& = & 
m + \frac{-1 + \sum_{j = 1}^s (p_j^2-1) \cdot (p_j^{n/(p_j-1)} -1)}{m}.
\end{eqnarray*}
Thus we have computed
\[
s_i = 
\begin{cases}
m + \frac{-1 + \sum_{j = 1}^s (p_j^2-1) \cdot (p_j^{n/(p_j-1)} -1)}{m} & i \; \text{even};
\\
0 &  i \; \text{odd}.
\end{cases}
\]


\subsection*{Acknowledgements}
\label{subsec:Acknowledgements}

The work was financially supported by the HCM (Hausdorff Center
for Mathematics) in Bonn, and the Leibniz-Award of the second author.


\typeout{----------   Section 1: Groups satisfying condition (M) --------}

\section{Groups satisfying condition $(M_{\triv \subseteq \calfin})$}
\label{sec:Groups_satisfying_condition_(M)}
 
Let $G$ be a group.  A group $H \subseteq G$ is called \emph{maximal
finite} if $H$ is finite and for every finite subgroup $K \subseteq
G$ with $H \subseteq K$ we have $H = K$. In this section we will make
the assumption $(M_{\triv\subseteq \calfin})$ appearing
in~\cite[Notation~2.7]{Lueck-Weiermann(2007)}, i.e., every non-trivial
finite subgroup of $G$ is contained in a unique maximal finite subgroup.
 
Let $\calm$ be the set of conjugacy
classes of maximal finite subgroups of $G$.  Let $N_GM$ be the
normalizer of $M \subseteq G$.  Put $W_GM = N_GM/M$.  Denote by 
$p_M \colon N_GM \to W_GM$ the canonical projection.  Notice that the group
$W_GM$ contains no torsion since $M \subseteq G$ is maximal finite.
Let $p_M^*EW_GM$ the $N_GM$-space obtained from the $W_GM$-space
$EW_GM$ by restriction with $p_M$. Denote by $\eub{G}$ the classifying
space for proper $G$-action.
For more information about these spaces we refer for instance 
to~\cite{Baum-Connes-Higson(1994)} and~\cite{Lueck(2005s)}.
Put $\bub{G} = G \backslash \eub{G}$.

We will suppose that $G$ satisfies the Baum-Connes Conjecture
(see for example~\cite{Baum-Connes-Higson(1994)} and~\cite{Lueck-Reich(2005)}),
i.e., the assembly map
\begin{eqnarray}
\asmb \colon K_i^G(\eub{G}) & \xrightarrow{\cong} & K_i(C^*_r(G))
\label{assembly_map}
\end{eqnarray}
is bijective for all $i \in \IZ$. Induction with the projection 
$G \to \{1\}$ yields a map (see~\cite[Chapter~6 on pages 732ff]{Lueck-Reich(2005)})
\begin{eqnarray*}
\ind_{G \to \{1\}} \colon K_i^G(\eub{G}) & \to & K_i(\bub{G}).
\end{eqnarray*}
Its composite with the inverse of the assembly map $\asmb$ of~\eqref{assembly_map}
is denoted by
\begin{eqnarray*}
\omega_i \colon K_i(C^*_r(G)) & \to & K_i(\bub{G}).
\end{eqnarray*}
Define
\begin{eqnarray*}
\eta_i \colon K_i(BG) & \to & K_i(C^*_r(G))
\end{eqnarray*}
to be the composite 
\[
K_i(BG) \xrightarrow{(\ind_{G \to \{1\}})^{-1}} K_i^G(EG) \xrightarrow{K_i^G(f)}
K_i^G(\eub{G}) \xrightarrow{\asmb} K_i(C^*_r(G)),
\]
where $f \colon EG \to \eub{G}$ is the up to $G$-homotopy unique $G$-map.
Let 
\[
i_M \colon \ker\bigl((p_M)_*\colon 
  K_i(C^*_r(N_{G}M)) \to K_i(C^*_r(W_{G}M))\bigr)
  \; \to \; 
   K_i(C^*_r(G))
\]
be the map induced by the inclusion $N_GM \to G$.
The main result of this section is

\begin{theorem}
\label{the:long_exact_sequence_for_K_ast(Cast(G))}
Suppose that $G$ satisfies condition $(M_{\triv \subseteq \calfin})$ appearing
in~\cite[Notation~2.7]{Lueck-Weiermann(2007)}, i.e., every non-trivial
finite subgroup of $G$ is contained in a unique maximal finite
subgroup.   Assume that the Baum-Connes Conjecture holds for $G$ and for
$W_GM$ and $N_GM$ for all $(M) \in \calm$. 

\begin{enumerate}
\item \label{the:long_exact_sequence_for_K_ast(Cast(G)):exact_sequence}
Then there is a long exact sequence
\begin{multline*}
  \cdots \xrightarrow{\partial_{i+1}}
  \bigoplus_{(M) \in \calm} \ker\bigl((p_M)_*\colon 
  K_i(C^*_r(N_{G}M)) \to K_i(C^*_r(W_{G}M))\bigr)
  \\
  \xrightarrow{\bigoplus_{(M) \in \calm}  i_M}
   K_i(C^*_r(G))
  \xrightarrow{\omega_i}
  K_i(\bub{G})
  \\
  \xrightarrow{\partial_i} 
  \bigoplus_{(M) \in \calm} \ker\bigl((p_M)_*\colon 
  K_{i-1}(C^*_r(N_{G}M)) \to K_{i-1}(C^*_r(W_{G}M))\bigr)
  \\
  \xrightarrow{\bigoplus_{(M) \in \calm}  i_{i-1}}
   K_{i-1}(C^*_r(G))
  \xrightarrow{\omega_{i-1}} \cdots;
\end{multline*}

\item \label{the:long_exact_sequence_for_K_ast(Cast(G)):split} 
  For every $i \in \IZ$ the map $
  \omega_i \colon K_i(C^*_r(G)) \to K_i(\bub{G})$ 
  is split surjective after inverting the orders of
  all finite subgroups of $G$.

  For every $i \in \IZ$ the homomorphism
  \begin{multline*}
  \eta_i \oplus \bigoplus_{(M) \in \calm}  i_{M} \colon
  K_i(BG) \oplus \bigoplus_{(M) \in \calm} \ker\bigl((p_M)_*\colon
  K_i(C^*_r(N_{G}M)) \to K_i(C^*_r(W_{G}M))\bigr) 
  \\ \to K_i(C^*_r(G))
  \end{multline*}
  is bijective after inverting the orders of all finite subgroups of $G$.
\end{enumerate}
\end{theorem}

The rest of this section is devoted to the proof of 
Theorem~\ref{the:long_exact_sequence_for_K_ast(Cast(G))}

\begin{proof}[Proof of Theorem~\ref{the:long_exact_sequence_for_K_ast(Cast(G))}]
We have a $G$-pushout of $G$-$CW$-complexes
(see~\cite[Corollary~2.10]{Lueck-Weiermann(2007)})
\begin{eqnarray}
\xymatrix{
\coprod_{(M) \in \calm} G \times_{N_{G} M}  EN_{G}M 
\ar[r]^-{i} 
\ar[d]_{\coprod_{(M) \in \calm} \id_G \times_{N_{G}M} f_M}
& EG \ar[d]^f
\\
\coprod_{(M) \in \calm} G \times_{N_{G}M}  p_M^* EW_{G}M 
\ar[r]_-{j} 
& \eub{G}
}
\label{Delta-pushout_for_eub(G)}
\end{eqnarray}
where $f_M \colon EN_{G}M \to p_M^*EW_{G}M$ is some cellular $N_{G}M$-map,
$f \colon EG \to \eub{G}$ is some $G$-map, and $i$ and $j$ are
inclusions of $G$-$CW$-complexes.

Dividing out the $G$-action yields
a pushout of $CW$-complexes
\begin{eqnarray}
\xymatrix{
\coprod_{(M) \in \calm} BN_{G}M 
\ar[r]
\ar[d]
& BG \ar[d]^f
\\
\coprod_{(M) \in \calm} BW_{G}M 
\ar[r]
& \bub{G}
}
\label{pushout_for_bub(G)}
\end{eqnarray}
The associated Mayer-Vietoris sequences yield a commutative diagram
whose columns are exact and whose horizontal arrows are given by
induction with the projection~$G \to \{1\}$
\begin{eqnarray}
\xymatrix{
\vdots \ar[d] 
& \vdots \ar[d] 
\\
\bigoplus_{(M) \in \calm} K_i^{N_GM}(EN_{G}M)  \ar[d] \ar[r]
&
\bigoplus_{(M) \in \calm} K_i(BN_{G}M)  \ar[d] 
\\
K_i^{G}(EG) \oplus 
\bigoplus_{(M) \in \calm} K_i^{N_{G}M}(p_M^* EW_{G}M)  \ar[d] \ar[r]  
&
K_i(BG) \oplus  \bigoplus_{(M) \in \calm} K_i(BW_{G}M) \ar[d] 
\\
K_i^G(\eub{G})  \ar[d] \ar[r]
&
K_i(\bub{G})  \ar[d]
\\
\bigoplus_{(M) \in \calm} K_{i-1}^{N_GM}(EN_{G}M)  \ar[d] \ar[r]
&
\bigoplus_{(M) \in \calm} K_{i-1}(BN_{G}M)  \ar[d] 
\\
K_{i-1}^{G}(EG) \oplus 
\bigoplus_{(M) \in \calm} K_{i-1}^{N_{G}M}(p_M^* EW_{G}M)  \ar[d] \ar[r]  
&
K_{i-1}(BG) \oplus  \bigoplus_{(M) \in \calm} K_{i-1}(BW_{G}M) \ar[d] 
\\
\vdots & \vdots
}
\label{commutative_diagram_coming_from_MV-sequences}
\end{eqnarray}

Since $G$ acts freely on $EG$ and $N_{G}M$ acts freely
on $EN_{G}M$, the maps given by induction with the projection to the trivial group
\begin{eqnarray*}
  K_i^G(EG) & \xrightarrow{\cong} & K_i(BG);
  \\
  K_i^{N_{G}M}(EN_{G}M) & \xrightarrow{\cong} & K_i(BN_{G}M),
\end{eqnarray*}
are bijective for all $i \in \IZ$ and $(M) \in \calm$.

Define the map
\[
k_i'  \colon \bigoplus_{(M) \in \calm} K_i^{N_{G}M}(p_M^* EW_{G}M)
\to 
K_i^{G}(\eub{G}) \oplus \bigoplus_{(M) \in \calm} K_i(BW_{G}M)
\]
to be the product of the map
\[
\bigoplus_{(M) \in \calm} K_i^{G}(u_M) \circ \ind_{N_GM \to G} \colon
K_i^{N_GM}(p_M^* EW_{G}M) \to K_i^G(\eub{G})
\]
for the up to $G$-homotopy unique $G$-map 
$u_M \colon G\times_{N_GM} p_M^* EW_GM \to \eub{G}$
and the map
\[
\bigoplus_{(M) \in \calm} \ind_{N_GM \to \{1\}} \colon 
\bigoplus_{(M) \in \calm} K_i^{N_{G}M}(p_M^* EW_{G}M) 
\to \bigoplus_{(M) \in \calm} K_i(BW_{G}M).
\]

Define the map
\[
j_i' \colon K_i^{G}(\eub{G}) \oplus \bigoplus_{(M) \in \calm} K_i(BW_{G}M)
  \to
  K_i(\bub{G})
\]
to be the direct sum of the map
\[
\ind_{G \to \{1\}} \colon K_i^{G}(\eub{G}) \to K_i(\bub{G})
\]
and $(-1)$-times the map
\[
\bigoplus_{(M) \in \calm} 
K_i(G \backslash u_M) \colon \bigoplus_{(M) \in \calm} K_i(BW_{G}M) 
\to K_i(\bub{G}).
\]
Define the map
\[
\partial_i' \colon K_i(\bub{G}) 
\to \bigoplus_{(M) \in \calm} K_{i-1}^{N_{G}M}(p_M^* EW_{G}M)
\]
to be the composite of the three maps
\begin{multline*}
K_i(\bub{G}) \xrightarrow{\delta_i} \bigoplus_{(M) \in \calm} K_{i-1}(BN_{G}M) 
\xrightarrow{\bigoplus_{(M) \in \calm} \bigl(\ind_{N_GM \to \{1\}}\bigr)^{-1}}
\bigoplus_{(M) \in \calm} K_{i-1}^{N_GM}(EN_{G}M) 
\\
\xrightarrow{\bigoplus_{(M) \in \calm} K_{i-1}(f_M)}
\bigoplus_{(M) \in \calm} K_{i-1}^{N_{G}M}(p_M^* EW_{G}M),
\end{multline*}
where $\delta_i$ is the boundary map appearing in the right column of
the diagram~\eqref{commutative_diagram_coming_from_MV-sequences} and
$f_M \colon EN_GM \to p_M^*EW_GM$ is the up to $N_GM$-homotopy
unique $N_GM$-map.

The exact columns in the diagram~\eqref{commutative_diagram_coming_from_MV-sequences} 
can be spliced together to the following long exact sequence.
\begin{multline}
  \cdots \xrightarrow{\partial_{i+1}'}
  \bigoplus_{(M) \in \calm} K_i^{N_{G}M}(p_M^* EW_{G}M)
  \xrightarrow{k_i'}
  K_i^{G}(\eub{G}) \oplus \bigoplus_{(M) \in \calm}
  K_i(BW_{G}M)
  \\
  \xrightarrow{j_i'}
  K_i(\bub{G})
    \xrightarrow{\partial_i'} 
  \bigoplus_{(M) \in \calm} K_{i-1}^{N_{G}M}(p_M^* EW_{G}M)
  \\
  \xrightarrow{k_{i-1}'}
  K_{i-1}^{G}(\eub{G}) \oplus \bigoplus_{(M) \in \calm} K_{i-1}(BW_{G}M) 
  \xrightarrow{j_{i-1}'} \cdots.
  \label{version_of_long_exact_sequence}
\end{multline}

We will need the following lemmas.

\begin{lemma} \label{lem:K_i(eubG)_to_K_i(bub(G))_is_rationally_split_injective}
Consider any group $G$ and any $i \in \IZ$. Let
$f \colon EG \to \eub{G}$ be the up to $G$-homotopy unique $G$-map.
Denote by $\overline{f} \colon BG \to \bub{G}$ the induced map on the
$G$-quotients. Then  the composite 
\[
K_i(BG) \xrightarrow{(\ind_{G \to \{1\}})^{-1}}
K_i^G(EG) \xrightarrow{f_*} K_i^G(\eub{G}) 
\xrightarrow{\ind_{G \to \{1\}}} K_i(\bub{G})
\]  
agrees with the map
\[
\overline{f}_* \colon K_i(BG) \to K_i(\bub{G}).
\]
This map is bijective after inverting the order of all finite subgroups of $G$.
\end{lemma}
\begin{proof}
Because induction is natural, we obtain the commutative diagram
\[
\xymatrix@C=22mm{K_i^G(EG) \ar[r]^{\ind_{G \to \{1\}}} \ar[d]^{f_*} 
& 
K_i(BG) \ar[d]^{(G\backslash f)_* = \overline{f}_*} 
\\
K_i^G(\eub{G}) \ar[r]^{\ind_{G \to \{1\}}} 
&
K_i(\bub{G})}
\]

The upper horizontal arrow is bijective since $G$ acts freely on $EG$.
It remains to show that the right vertical arrow is bijective after
inverting the orders of all finite subgroups of $G$.

Let $\Lambda$ be the ring $\IZ \subseteq \Lambda \subseteq \IQ$
obtained from $\IZ$ by inverting the orders of all finite subgroups of
$G$.  By a spectral sequence argument it suffices to show that
$\overline{f}_* \colon H_l(BG,\Lambda) \to H_l(\bub{G};\Lambda)$ is
bijective for all $l \in \IZ$. This follows from the fact that the
$\Lambda G$-chain map 
$C_*(f) \colon C_*(EG) \otimes_{\IZ} \Lambda \to C_*(\eub{G}) \otimes_{\IZ} \Lambda$ 
is a homology equivalence of
projective $\Lambda G$-chain complexes and hence a $\Lambda G$-chain
homotopy equivalence.
\end{proof}

Let $p \colon G \to H$ be an epimorphism with finite kernel.  Then
restriction defines for every $H$-$CW$-complex $X$ and every $i\in \IZ$ 
a natural map
\[
\res_p(X)  \colon K_i^H(X) \to K_i^G(\res_p(X))
\]
where $\res_p(X)$ is the $G$-$CW$-complex whose underlying space is
$X$ and for which $g \in G$ acts on $X$ by multiplication with
$p(g)$. This follows from the methods developed in~\cite[Chapter~6 on pages 732ff]{Lueck-Reich(2005)}. 
The maps $\res_p(X)$ above define a transformation
of $H$-homology theories. Notice that we obtain for every
$H$-$CW$-complex a natural $H$-homeomorphism
\[
\alpha(X) \colon \ind_p \res_p(X) \xrightarrow{\cong} X
\]
that is the adjoint of the identity $\id \colon \res_p(X) \to \res_p(X)$.  
Explicitly it sends $(h,x) \in H \times_G \res_p(X)$ to
$h \cdot x$.  The inverse sends $x$ to $(1,x)$.
We leave it to the reader to check the proof of the next lemma which is just a direct inspection of the definitions.

\begin{lemma} \label{lem:ind_split_surjective}
Let $p \colon G \to H$ be an epimorphism with finite kernel.
Then for every $H$-$CW$-complex $X$ the composite
\[
K_i^H(X) \xrightarrow{\res_p} K_i^G(\res_p(X)) \xrightarrow{\ind_p} 
K_i^H(\ind_p \circ \res_p(X)) \xrightarrow{K_i^H(\alpha(X))} K_i^H(X)
\]
is the identity.
\end{lemma}

Lemma~\ref{lem:ind_split_surjective} applied to the projection
$N_GM \to W_GM$ and the $W_GM$-space $EW_GM$ implies that 
$\ind_{N_GM \to W_GM} \colon K_i^{N_GM}(p_M^*EW_GM) \to  K_i^{W_GM}(EW_GM)$
is split surjective. Since 
$\ind_{W_GM \to \{1\}} \colon  K_i^{W_GM}(EW_GM) \to K_i(BW_GM)$
is bijective, the map
\[
\ind_{N_GM \to \{1\}} \colon K_i^{N_GM}(p_M^*EW_GM) \to K_i(BW_GM)
\]
is split surjective. Now 
assertion~\ref{the:long_exact_sequence_for_K_ast(Cast(G)):exact_sequence}
of Theorem~\ref{the:long_exact_sequence_for_K_ast(Cast(G))}
follows from the long exact sequence~\eqref{version_of_long_exact_sequence}
and the fact that $EW_GM$ is a model for $\eub W_GM$ 
and $p_M^*EW_GM$ is a model for $\eub N_GM$ 
because $W_GM$ is the torsion-free quotient of $N_G M$ by a finite group.
We obtain assertion~\ref{the:long_exact_sequence_for_K_ast(Cast(G)):split} 
from assertion~\ref{the:long_exact_sequence_for_K_ast(Cast(G)):exact_sequence} 
and Lemma~\ref{lem:K_i(eubG)_to_K_i(bub(G))_is_rationally_split_injective}.
This finishes the proof of Theorem~\ref{the:long_exact_sequence_for_K_ast(Cast(G))}.
\end{proof}

\begin{example}[Groups satisfying property $(M_{\triv \subseteq
    \calfin})$]
  \label{exa:Groups_satisfying_property_(M)}
  Consider an extension of groups 
  $1 \to \IZ^r \xrightarrow{j} G \xrightarrow{q} A \to 1$
  such that $A$ is finitely generated
  abelian. Suppose that $\tors(A)$ acts freely on $\IZ^r$ outside the
  origin $0 \in \IZ^r$.  

  Then $G$ satisfies condition 
  $(M_{\triv\subseteq \calfin})$ by the following argument. Put 
  $H = q^{-1}(\tors(A))$. We obtain an exact sequence 
  $1 \to \IZ^r  \to H \to \tors(A) \to 1$ such that the conjugation action
  of the finite group $\tors(A)$ on $\IZ^r$ is free outside the
  origin. Then $H$ satisfies by~\cite[Lemma~6.3]{Lueck-Stamm(2000)}
  the condition $(N\!M_{\triv \subseteq \calfin})$ appearing
  in~\cite[Notation~2.7]{Lueck-Weiermann(2007)}, i.e., every
  non-trivial finite subgroup of $H$ is contained in a unique maximal
  finite subgroup and $N_HM = M$ holds for any maximal finite subgroup
  $M \subseteq H$.  We also obtain an exact sequence  
  $1 \to H \to G \xrightarrow{\overline{q}} A/\tors(A) \to 1$. Hence any 
  finite subgroup of $G$ belongs to $H$. This implies that $G$ satisfies
  $(M_{\triv \subseteq \calfin})$.

  Consider any maximal finite subgroup
  $M \subseteq G$. Then $M \subseteq H$ and  have $N_GM \cap \IZ^r = \{1\}$ and 
  $N_GM \cap H = M$. Hence $q$ induces an isomorphism $N_GM \to q(N_GM)$.
  Since $A$ is abelian, $N_GM$ is abelian. We get an exact sequence 
  $1 \to M \to N_GM \to \overline{q}(N_GM) \to 1$, 
  where $\overline{q}(N_GM)$ is a finitely
  generated free abelian group. This implies that $N_GM \cong M \times W_GM$
  and $W_GM$ is a finitely generated free abelian group. Let $\widetilde{R}_{\IC}(M)$ be the kernel of the split surjection
$R_{\IC}(M) \to R_{\IC}(\{1\})$ sending the class of an $M$-representation $V$ to $\dim_{\IC}(\IC \otimes_{\IC M} V)$.
  An easy calculation shows
  \begin{eqnarray*} 
  \ker\bigl((p_M)_*\colon 
  K_i(C^*_r(N_{G}M)) \to K_i(C^*_r(W_{G}M))\bigr)
  & \cong &
  K_i(W_GM) \otimes_{\IZ} \widetilde{R}_{\IC}(M).
\end{eqnarray*}
  In particular $\ker\bigl((p_M)_*\colon 
  K_i(C^*_r(N_{G}M)) \to K_i(C^*_r(W_{G}M))\bigr)$ is torsionfree.
  Now Theorem~\ref{the:long_exact_sequence_for_K_ast(Cast(G))}
  yields for every $m \in \IZ$ a short exact sequence
  \[
  0 \to 
  \bigoplus_{(M) \in \calm} K_i(BW_{G}M) \otimes_{\IZ} \widetilde{R}_{\IC}(M) 
  \to
  K_i(C^*_r(G)) \to K_i(\bub{G}) \to 0
  \]
 which splits after inverting $|\tors(A)|$.

 A prototype for this example is $G = R \rtimes R^{\times}$ for an
 integral domain $R$ such that the underlying abelian group of $R$ is
 finitely generated free and the abelian group $R^{\times}$ is
 finitely generated, where $R^{\times}$ acts on $R$ by multiplication.
 The ring of integers in an algebraic number field is an example by Dirichlet's Unit Theorem
(see for instance~\cite[Theorem~7.4 in Chapter I on page~42]{Neukirch(1999)}).
\end{example}


\typeout{----------   Section 2: Groups satisfying condition (NM) --------}

\section{Groups satisfying condition $(N\!M_{\triv \subseteq \calfin})$}
\label{sec:Groups_satisfying_condition_(NM)}

Let $G$ be a group. In this section we will make the assumption
$(N\!M_{\triv \subseteq \calfin})$ appearing
in~\cite[Notation~2.7]{Lueck-Weiermann(2007)}, i.e., every non-trivial
finite subgroup of $G$ is contained in a unique maximal finite
subgroup and $N_GM = M$ holds for any maximal finite subgroup $M \subseteq G$.


\subsection{On the topological $K$-theory of the group $C^*$-algebra}
\label{subsec:On_the_topological_K-theory_of_the_group_Cast-algebra}
\ \\
 
Let $\widetilde{R}_{\IC}(M)$ be the kernel of the split surjection
$R_{\IC}(M) \to R_{\IC}(\{1\})$ sending the class of an $M$-representation $V$ to $\dim_{\IC}(\IC \otimes_{\IC M} V)$.
It corresponds under the identifications $R_{\IC}(M) = K_0^M(\pt)$ and $R_{\IC}(\{1\}) = K_0^{\{1\}}(\pt)$
to the homomorphism $\ind_{M \to \{1\}} \colon K_0^M(\pt)  \to K_0^{\{1\}}(\pt)$. 
Notice that $K_0(C^*_r(M)) = R_{\IC}(M)$ and $\widetilde{R}_{\IC}(M)$ are
finitely generated free abelian groups and $K_1(C^*_r(M)) = 0$ for every
finite group $M$. 

The next result is a direct consequence of 
Theorem~\ref{the:long_exact_sequence_for_K_ast(Cast(G))}.

\begin{corollary}
\label{cor:long_exact_sequence_for_K_ast(Cast(G))_for_G_satisfying_NM}
Suppose that $G$ satisfies condition $(N\!M_{\triv \subseteq \calfin})$ appearing
in~\cite[Notation~2.7]{Lueck-Weiermann(2007)}, i.e., every non-trivial
finite subgroup of $G$ is contained in a unique maximal finite
subgroup and for every maximal finite subgroup $M$ we have $N_GM = M$.
Suppose that $G$ satisfies the Baum-Connes Conjecture.

\begin{enumerate}

\item \label{cor:long_exact_sequence_for_K_ast(Cast(G))_for_G_satisfying_NM:K1}
We obtain an isomorphism
\[
 \omega_1 \colon K_{1}(C^*_r(G)) \xrightarrow{\cong} K_1(\bub{G});
\]
\item \label{cor:long_exact_sequence_for_K_ast(Cast(G))_for_G_satisfying_NM:K0}
We obtain a short exact sequence
\[
0 \to \bigoplus_{(M) \in \calm} \widetilde{R}_{\IC}(M) 
\xrightarrow{\bigoplus_{(M) \in \calm} i_M }   K_{0}(C^*_r(G))   
\xrightarrow{\omega_0}  K_{0}(\bub{G}) \to 0,
\]
where the map $i_M$ comes from the
inclusion $M \to G$ and the identification $R_{\IC}(M) = K_0(C^*_r(M))$.
It splits if one inverts
the orders of all finite subgroups of $G$.

The homomorphisms
\begin{align*}
\eta_0 \oplus \bigoplus_{(M) \in \calm} i_M  \colon 
K_0(BG) \oplus \bigoplus_{(M) \in \calm} \widetilde{R}_{\IC}(M)  
& \to 
K_0(C^*_r(G))
\\
\eta_1  \colon 
K_1(BG)
& \to 
K_1(C^*_r(G))
\end{align*}
are bijective after inverting the orders of all finite subgroups of $G$.
\end{enumerate}
\end{corollary}

\typeout{----------   Section 3: The group cohomology of classifying spaces  ------------------------}

\section{The cohomology of $\Gamma$ and of the associated toroidal orbifold quotient}
\label{sec:The_cohomology_of_classifying_spaces}

In this section we compute the cohomology of $B\Gamma$  and $\bub{\Gamma}$
which is also called the associated toroidal orbifold quotient. We need some preliminaries.

\begin{lemma}
\label{lem:extension_splits}
Let $1 \to \IZ^n \to   \Gamma \to  \IZ/m \to 1$ be an extension such that the conjugation action of
$\IZ/m$ on $\IZ^n$ is free outside the origin. Then the extensions splits, the group $\Gamma$ is a crystallographic
group of rank $n$
and possesses a finite $n$-dimensional $\Gamma$-$CW$-model for $\eub{\Gamma}$.
\end{lemma}
\begin{proof}
  Let $\gamma \in \Gamma$ be an element in $\Gamma$ which is mapped under
  $\Gamma \to \IZ/m$ to a generator of $\IZ/m$.  Then $\gamma^m$ belongs to
  $\IZ^n$ and $\gamma \gamma^m \gamma^{-1} = \gamma^m$. Since $\gamma$ is
  non-trivial and the conjugation action of $\IZ/m$ on $\IZ^n$ is free outside
  the origin, $\gamma^m$ is the origin in $\IZ^n$.  This implies $\gamma^m = 1$
  in $\Gamma$.

  The subgroup $\IZ^n$ of $\Gamma$ is normal and its own centralizer in
  $\Gamma$.  Hence $\Gamma$ is a crystallographic group of rank $n$ and has a
  finite $n$-dimensional $G$-$CW$-model for $\eub{\Gamma}$, namely $\IR^n$ with
  the associated $\Gamma$-action, by~\cite[Propositions~1.12]{Connolly-Kozniewski(1990)}.
\end{proof}

One key ingredient for the sequel is the following result from Langer-L\"uck~\cite[Theorem~0.1 and Theorem~0.5]{Langer-Lueck(2011_homology)}.

\begin{theorem}[Tate cohomology]
\label{the:Tate_cohomology}
Suppose that  the $\IZ/m$-action on $\IZ^n$ is free outside the origin.
Then:

\begin{enumerate}

\item \label{the:Tate_cohomology:Tate} 
We get for the Tate cohomology 
\[
\widehat{H}^i(\IZ/m;\Lambda^j(\IZ^n_{\rho})) = 0
\]
for all $i,j$ for which $i + j$ is odd;

\item \label{the:Tate_cohomology:LSS} 
The Lyndon-Serre sequence associated to the extension
 $1 \to \IZ^n \to   \Gamma \to  \IZ/m \to 1$ collapses in the strongest sense, i.e., all
differentials in the $E_r$-term are trivial for all $r \ge 2$, and all extension
problems at the $E_{\infty}$-level are trivial. 
\end{enumerate}
\end{theorem}

The main result of this section is

\begin{theorem}[Cohomology of $B\Gamma$ and $\bub{\Gamma}$]
\label{the:Cohomology_of_BGamma_and_bub(Gamma)}
  Consider the  extension of groups
  $1 \to \IZ^n \to   \Gamma \xrightarrow{\pi}  \IZ/m \to 1$. Assume  that the 
  conjugation action  of $\IZ/m$ on   $\IZ^n$ is free outside the origin  $0 \in \IZ^n$. 
 Put for $i \ge 0$
 \[
 r_i := \rk_{\IZ}((\Lambda^{i} \IZ^n_{\rho})^{\IZ/m}) = \rk_{\IZ}(H^i(\IZ^n_{\rho})^{\IZ/m}).
 \]

Then
  
  \begin{enumerate}

  \item \label{the:Cohomology_of_BGamma_and_bub(Gamma):Hm(BGamma)} 
 For $i \ge 0$
\[
H^i(\Gamma) \cong 
\begin{cases}
  \IZ^{r_i} \oplus \bigoplus_{l = 0}^{i-1} \widehat{H}^{i-l}(\IZ/m,\Lambda^l\IZ^n_{\rho}) & i \;\text{even;}
  \\
  \IZ^{r_i} & i \;\text{odd}, i \ge 3;
  \\ 0 & i = 1.\end{cases}
\]

\item \label{the:Cohomology_of_BGamma_and_bub(Gamma):restriction_to-p-subgr} The map induced
by the various inclusions
\[
 \varphi^{i} \colon {H}^{i}(\Gamma)  \to \bigoplus_{(M) \in \calm} H^{i}(M)
\]
is bijective for $i > n$;

\item
  \label{the:Cohomology_of_BGamma_and_bub(Gamma):bub(Gamma)}
  For $i \ge 0$
\[
H^i(\bub{\Gamma}) \cong 
\begin{cases}
  \IZ^{r_i} & i \;\text{even;}
  \\
  \IZ^{r_i} \oplus  \bigoplus_{l = i}^{n} \widehat{H}^{l}(\IZ/m,\Lambda^{l}\IZ^n_{\rho})  & i \;\text{odd}, i \ge 3;
  \\
  0 & i = 1.
\end{cases}
\]

\end{enumerate}

\end{theorem}
\begin{proof}~\ref{the:Cohomology_of_BGamma_and_bub(Gamma):Hm(BGamma)} 
follows directly from Theorem~\ref{the:Tate_cohomology} since $r_1 = 0$.
\\[1mm]~\ref{the:Cohomology_of_BGamma_and_bub(Gamma):restriction_to-p-subgr}
The pushout~\eqref{pushout_for_bub(G)} reduces to the following
pushout of $CW$-complexes since $\Gamma$ satisfies the 
condition $(N\!M_{\triv \subseteq \calfin})$ by~\cite[Lemma~6.3]{Lueck-Stamm(2000)}.
\begin{eqnarray}
\xymatrix{
\coprod_{(M) \in \calm} BM 
\ar[r]
\ar[d]
& B\Gamma \ar[d]^f
\\
\coprod_{(M) \in \calm} \pt 
\ar[r]
& \bub{\Gamma}
}
\label{pushout_for_bub(G)_NM}
\end{eqnarray}
 Since $H^{2i+1}(M) = 0$ for all $i$ and all $(M) \in \calm$,
the associated Mayer-Vietoris sequence yields the long exact sequence
\begin{multline}
  0 \to {H}^{2i}(\bub{\Gamma}) \xrightarrow{\overline{f}^*} {H}^{2i}(\Gamma)
  \xrightarrow{\varphi^{2i}} \bigoplus_{(M) \in \calm} \widetilde{H}^{2i}(M)
  \\
  \xrightarrow{\delta^{2i}} {H}^{2i+1}(\bub{\Gamma})
  \xrightarrow{\overline{f}^*} {H}^{2i+1}(\Gamma) \to 0
  \label{long_exact_cohomology_sequences_for_bub(Gamma)_BGamma)}
\end{multline}
where $\varphi^{2i}$ is the map induced by the various inclusions $M \to \Gamma$.
Since there exists a $n$-dimensional model for $\eub{\Gamma}$ by Lemma~\ref{lem:extension_splits}, 
assertion~\ref{the:Cohomology_of_BGamma_and_bub(Gamma):restriction_to-p-subgr}
follows.
\\[1mm]~\ref{the:Cohomology_of_BGamma_and_bub(Gamma):bub(Gamma)}
In the sequel let $i$ be an integer with $i \ge 1$.
Recall that the Lyndon-Serre spectral sequence  associated to the extension
 $1 \to \IZ^n \to   \Gamma \to  \IZ/m \to 1$ yields a descending filtration
$$H^{l}(\Gamma) = F^{0,l} \supset F^{1,l-1} \supset  \cdots \supset F^{l,0} \supset
F^{l+1,-1} =0$$ such that $F^{r,l-r}/F^{r+1,l-r-1} \cong
E^{r,l-r}_{\infty}$.  Let $\beta \in H^2(\IZ/m)$ be a fixed generator. 
Recall that $E^{2,0}_2 = H^2(\IZ/p;H^0(\IZ^n_{\rho})) =
H^2(\IZ/p)$ so that we can think of $\beta$ as an element in $E^{2,0}_2$.
We conclude $E^{i,j}_2 = E^{i,j}_{\infty}$ from
Theorem~\ref{the:Tate_cohomology}~\ref{the:Tate_cohomology:LSS}.
From the multiplicative structure of the spectral sequence we see that the image of the map
$$- \cup \pi^*(\beta)^{n} \colon H^{2i}(\Gamma)  \to H^{2i+2n}(\Gamma)$$
lies in $F^{2n,2i}$ and the following diagram commutes
\begin{equation}
  \label{filtration_commute}
  \xymatrix@!C=8em{0 \ar[d] 
    &
    0 \ar[d] 
    \\
    F^{1,2i-1} \ar[r]_{-\cup \pi^{*}(\beta)^{n}}^{\cong} \ar[d]
    & 
    F^{2n+1,2i-1} \ar[d]
    \\        
    H^{2i}(\Gamma) \ar[r]_{-\cup \pi^{*}(\beta)^{n}} \ar[d]
    & 
    F^{2n,2i} \ar[d]
    \\
    E_{\infty}^{0,2i} \ar[r]_{-\cup \beta^{n}} \ar[d]
    & 
    E_{\infty}^{2n,2i} \ar[d]
    \\
    0 
    &
    0
  }
\end{equation}
where the columns are exact. 

Next we show that the upper horizontal arrow is bijective.  Namely,
we prove by induction over $r = -1, 0, 1, \ldots, 2i-1$ that the map
\[
-\cup \pi^{*}(\beta)^{n} \colon F^{2i-r,r} \to F^{2i-r+2n,r}
\]
is bijective. The induction beginning $r = -1$ is trivial since then both the
source and the target are trivial, and the induction step from $r-1$ to $r$
follows from the Five-Lemma and the fact that the map
\[
- \cup \beta^n \colon E^{2i-r,r}_{\infty} = H^{2i-r}(\IZ/m;H^r(\IZ^n_{\rho})) \to 
E^{2i-r+2n,r}_{\infty} = H^{2i-r+2n}(\IZ/m;H^r(\IZ^n_{\rho}))
\]
 is bijective for $1 \le 2i -r$

The bottom horizontal map in diagram~\eqref{filtration_commute} can be
identified with the composition of the canonical quotient map
\[
(\Lambda^{2i} \IZ^n_{\rho})^{\IZ/m} = H^0(\IZ/m;H^{2i}(\IZ^n_{\rho})) \to \widehat{H}^0(\IZ/m;H^{2i}(\IZ^n_{\rho}))
\]
with the isomorphism
\[- \cup \beta^n \colon \widehat{H}^0(\IZ/m;H^{2i}(\IZ^n_{\rho})) \xrightarrow{\cong}
\widehat{H}^{2n}(\IZ/m;H^{2i}(\IZ^n_{\rho})).
\]
We conclude from the Snake-Lemma that the middle
map in diagram~\eqref{filtration_commute} 
is an epimorphism and that its kernel  fits into an exact sequence
\begin{multline}
\label{exact_sequence_for_K2i}
0 \to \ker\bigl(- \cup \pi^*(\beta)^{n} \colon H^{2i}(\Gamma)  \to H^{2i+2n}(\Gamma)\bigr) \to (\Lambda^{2i} \IZ^n_{\rho})^{\IZ/m} 
\\
\to \widehat{H}^0(\IZ/m;H^{2i}(\IZ^n_{\rho})) \to 0.
\end{multline}
We have the following commutative diagram
\begin{equation}
\label{com_diagram_beta}
\xymatrix{H^{2i}(\Gamma) \ar[r]^-{\varphi^{2i}} 
\ar[d]_{- \cup \pi^*(\beta)^{n}}
&
\bigoplus_{(M) \in \calm} H^{2i}(M)
\ar[dd]^{\bigoplus_{(M) \in \calm} (- \cup (\pi \circ i_M)^*\beta^n)}_{\cong}
\\
F^{2n,2i} \ar[d]_{\iota} 
& 
\\
H^{2i+2n}(\Gamma) \ar[r]^-{\varphi^{2i+2n}}_-{\cong}
&
\bigoplus_{(M) \in \calm} H^{2i+2n}(M) 
}
\end{equation}
where $\iota$ is the inclusion, the lower horizontal map is bijective because of~\eqref{long_exact_cohomology_sequences_for_bub(Gamma)_BGamma)}
and the existence of a $n$-dimensional model for $\eub{\Gamma}$ by Lemma~\ref{lem:extension_splits},
and the left vertical map is bijective since the map $H^2(\IZ/m) \to H^2(M)$ induced by the injection $\pi \circ i_M \colon M \to \IZ/m$
sends $\beta$ to a generator.

Since $ \widehat{H}^0(\IZ/m;H^{2i}(\IZ^n_{\rho}))$ is finite and 
\[
\ker\bigl(- \cup \pi^*(\beta)^{n} \colon H^{2i}(\Gamma)  \to H^{2i+2n}(\Gamma)\bigr)
= \ker(\varphi^{2i}) \cong H^{2i}(\bub{\Gamma})
\]
by the exact sequence~\eqref{long_exact_cohomology_sequences_for_bub(Gamma)_BGamma)} and the commutative diagram~\eqref{com_diagram_beta}
we conclude for $i \ge 1$
\[
H^{2i}(\bub{\Gamma}) \cong \IZ^{r_{2i}}.
\]
From exact sequence~\eqref{long_exact_cohomology_sequences_for_bub(Gamma)_BGamma)} and the commutative diagram~\eqref{com_diagram_beta}
we obtain the exact sequence
\[
0 \to \cok\bigl(\iota \colon F^{2n,2i} \to H^{2i+2n}(\Gamma)\bigr) \to H^{2i+1}(\bub{\Gamma}) \to H^{2i+1}(\Gamma) \to 0.
\]
Since $H^{2i+1}(\Gamma)  \cong \IZ^{r_{2i+1}}$, this sequence splits.
We conclude from Theorem~\ref{the:Tate_cohomology}~\ref{the:Tate_cohomology:LSS}
\begin{eqnarray*}
\cok(\iota)
& \cong &
\bigoplus_{l = 0}^{2n-1} E^{l,2i+2n-l}_{\infty}
\\
& \cong &
\bigoplus_{l = 0}^{2n-1} H^l(\IZ/m,\Lambda^{2i+2n-l}\IZ^n_{\rho})
\\
& \cong &
\bigoplus_{l = 2i+n}^{2n-1} H^l(\IZ/m,\Lambda^{2i+2n-l}\IZ^n_{\rho})
\\
& \cong &
\bigoplus_{l = 2i+1}^{n} \widehat{H}^{2i +2n - l}(\IZ/m,\Lambda^{l}\IZ^n_{\rho})
\\
& \cong &
\bigoplus_{l = 2i+1}^{n} \widehat{H}^{l}(\IZ/m,\Lambda^{l}\IZ^n_{\rho}).
\end{eqnarray*}
We conclude for $i \ge 1$.
\[
H^{2i+1}(\bub{\Gamma})  \cong \IZ^{r_{2i+1}} \oplus \bigoplus_{l = 2i+1}^{n} \widehat{H}^{l}(\IZ/m,\Lambda^{l}\IZ^n_{\rho}).
\]
Obviously $H^0(\bub{\Gamma}) \cong \IZ^{r_0}$ holds. Since $H^1(\Gamma) = 0$ 
by assertion~\ref{the:Cohomology_of_BGamma_and_bub(Gamma):Hm(BGamma)} , we conclude $H^1(\bub{\Gamma}) = 0$
from the exact sequence~\eqref{long_exact_cohomology_sequences_for_bub(Gamma)_BGamma)}.
This finishes the proof of Theorem~\ref{the:Cohomology_of_BGamma_and_bub(Gamma)}.
\end{proof}

For a computation of the cohomology of $\Gamma$ and $\bub{\Gamma}$ in the case where
$m$ is a prime and the conjugation action of $\IZ/m$ on $\IZ^n$ is not required to be free outside the origin,
we refer to~\cite{Adem-Ge-Pan-Petrosyan(2008)} and~\cite{Adem-Duman-Gomez(2010)}.

\begin{remark}[Homology] 
\label{rem:homology}
By the universal coefficient theorems one can figure out the homology as well.
It is easier to determine the cohomology because of the multiplicative structure coming from the cup product.
\end{remark}

\typeout{----------   Section 4: Topological $K$-theory of classifying spaces ------------------------}

\section{Topological $K$-theory of classifying spaces}
\label{sec:Topological_K-theory_of_classifying_spaces}


\subsection{Comparing $K^1$ for classifying spaces}
\label{subsec:Comparing_K1(BG)_and_K1(bubB)}
\ \\

For later purpose we prove

\begin{lemma} \label{lem:K1(eub(Gamma)_to_K(BGamma)_iso_if_MN}
Suppose that the group $G$ satisfies the condition
$(N\!M_{\triv \subseteq \calfin})$ appearing
in~\cite[Notation~2.7]{Lueck-Weiermann(2007)} and that there exists a finite $G$-$CW$-model for $\eub{G}$.

Then the canonical map
\[
K_{G}^1(\eub{G}) \to K^1(BG)
\]
is bijective.
\end{lemma}
\begin{proof}
Let $\overline{R}_{\IC}(M)$ be the cokernel of the homomorphism
$R_{\IC}(\{1\}) \to R_{\IC}(M)$ given by restriction with $M \to \{1\}$. 
It corresponds under the identifications $R_{\IC}(\{1\}) = K^0_{\{1\}}(\pt) $ and $R_{\IC}(M) = K^0_{M}(\pt)$ 
to the induction homomorphism $\ind_{M \to \{1\}} \colon K^0_{\{1\}}(\pt)  \to K^0_{M}(\pt)$.  
(Notice that we are using the notation of~\cite[Section~1]{Lueck(2005c)}, in~\cite{Lueck-Oliver(2001b)}
this map is called inflation.)
Define 
\[
\iota_M \colon R_{\IC}(M) = K^0_{M}(\pt) \xrightarrow{K^0_M(\pr)} K^0_M(EM) \xrightarrow{\ind_{M \to \{1\}}^{-1}} K^0(BM).
\]
Let $\widetilde{K}^0(X)$ be the cokernel of the map
$K^0(\pt) \to K^0(X)$ induced by the projection
$X \to \pt$. The map $\iota_M$ induces an homomorphism 
\[
\overline{\iota}_M \colon \overline{R}_{\IC}(M) \to \widetilde{K}^0(BM).
\]
We obtain from the Mayer-Vietoris sequences
for $K^*_G$ and $K^*$ applied to the $G$-pushout~\eqref{Delta-pushout_for_eub(G)}
and to the pushout~\eqref{pushout_for_bub(G)}
the commutative diagram with exact rows
(compare~\cite[Proof of Theorem~7.1 on page~30]{Davis-Lueck(2010)})
\[
\xymatrix{
\bigoplus_{(M) \in \calm} \overline{R}_{\IC}(M)
\ar[d]_{\bigoplus_{(M) \in \calm} \overline{\iota}_M} \ar[r]
&
K^1(\bub{G})  \ar[r] \ar[d]_{\id}
&
K^1_{G}(\eub{G}) \ar[r] \ar[d]
& 
0
\\
\bigoplus_{(M) \in \calm}  \widetilde{K}^0(BM)  \ar[r]
&
K^{1}(\bub{G}) \ar[r]
&
K^{1}(BG) \ar[r]
&
0
}
\]

By the Five-Lemma it remains to show that the obvious map
\[
\ker\bigl(K^{1}(\bub{G}) \to K^{1}_{G}(\eub{G})\bigr)
\to 
\ker\bigl(K^{1}(\bub{G}) \to K^{1}(BG)\bigr)
\]
is surjective. The group $K^1(\bub{G})$ is finitely generated since there is a
finite $G$-$CW$-model for $\eub{G}$. Hence also $K^1(BG)$ is finitely generated.
For any non-trivial element of finite order $g \in G$ its centralizer
$C_G\langle g \rangle$ is finite because the condition $(N\!M_{\triv \subseteq \calfin})$
is satisfied. Hence the rank of the finitely generated group
$K^1(BG)$ is $\sum_{j \in \IZ} \dim_{\IQ}(H^{2j + 1}(BG;\IQ))$
by~\cite[Theorem~0.1]{Lueck(2007)}.  The rank of the finitely generated group
$K^1(\bub{G})$ is $\sum_{j \in \IZ} \dim_{\IQ}(H^{2j + 1}(\bub{G};\IQ))$ since
$\bub{G}$ is finite and there exists a rational Chern character. Since the
obvious map $BG \to \bub{G}$ induces isomorphisms $H^{j}(\bub{G)};\IQ)
\xrightarrow{\cong} H^{j}(BG;\IQ)$ for all $j \in \IZ$, the ranks of the
finitely generated abelian groups $K^1(BG)$ and $K^1(\bub{G})$ agree.  Hence the
kernel of the epimorphism $K^{1}(\bub{G}) \to K^{1}(BG)$ is finite.  This implies that there
is an integer $l$ such that multiplication with $l$ annihilates the kernel.
Therefore it remains to show for every integer $l > 0$ that the obvious
composite
\begin{multline*}
\bigoplus_{(M) \in \calm} \overline{R}_{\IC}(M) 
\xrightarrow{\bigoplus_{(M) \in \calm} \overline{\iota}_M}
\bigoplus_{(M) \in \calm}\widetilde{K}^0(BM)
\\
\to
\left(\bigoplus_{(M) \in \calm} \widetilde{K}^0(BM)\right)\left/
l \cdot \left(\bigoplus_{(M) \in \calm} \widetilde{K}^0(BM)\right)\right.
\end{multline*}
is surjective. Obviously it suffices to show for every $(M) \in \calm$
that the composite
\[
R_{\IC}(M) 
\xrightarrow{\iota_M}
K^0(BM)
\to
K^0(BM)\left/
l \cdot K^0(BM)\right.
\]
is surjective.

Let $\II_M$ be the augmentation ideal, i.e., the kernel of the ring
homomorphism $R_{\IC}(M) \to \IZ$ sending $[V]$ to $\dim_{\IC}(V)$.
If $M_p \subseteq M$ is a $p$-Sylow subgroup, restriction defines a
map $\II_M \to \II_{M_p}$. Let $\II_p(M)$ be the quotient of $\II(M)$
by the kernel of this map. This is independent of the choice of the
$p$-Sylow subgroup since two $p$-Sylow subgroups of $M$ are conjugate.
There is an obvious isomorphism from $\II_p(M) \xrightarrow{\cong}
\im(\II_M \to \II_{M_p})$.

Then there is an isomorphism of abelian groups 
(see~\cite[Theorem~0.3]{Lueck(2007)})
\[
K^0(BM)  \xrightarrow{\cong} \IZ \times \prod_{p\;\text{prime}}  \II_p(M)
\otimes_{\IZ} \IZ\widehat{_p}
\]
The map $R_{\IC}(M) \xrightarrow{\iota_M} K^0(BM)$
can be identified under this isomorphism with the obvious composite
\[
R_{\IC}(M) \xrightarrow{\cong} \IZ \times \II_M \to 
\IZ \times \prod_{p\;\text{prime}}  \II_p(M) 
\to
\IZ \times \prod_{p\;\text{prime}}  \II_p(M) \otimes_{\IZ} \IZ\widehat{_p}.
\]
The map 
$\IZ \times \II(M) \to \IZ \times \prod_{p\;\text{prime}}  \II_p(M)$
is surjective by~\cite[Lemma~3.4]{Lueck(2007)}. 
Hence it suffices to show that the map
\[
\IZ \times \prod_{p\;\text{prime}}  \II_p(M) 
\to
\left(\IZ \times \prod_{p\;\text{prime}}  
\II_p(M) \otimes_{\IZ} \IZ\widehat{_p}\right)
\left/l \cdot \left(\IZ \times \prod_{p\;\text{prime}}  
\II_p(M) \otimes_{\IZ} \IZ\widehat{_p}\right)\right.
\]
is surjective.   Since $\II_p(M)$ can be non-trivial
only for finitely many primes, namely those which divide $n$,
it suffices to show for every prime $p$ that the canonical map
\[
\II_p(M) \otimes_{\IZ} \IZ  \xrightarrow{\id \otimes \pr} \II_p(M) \otimes_{\IZ} 
\left(\IZ\widehat{_p}/ l \cdot \IZ\widehat{_p}\right)
\]
is onto. This follows from the fact that the composite
\[
\IZ \to \IZ\widehat{_p} \to \IZ\widehat{_p}/l \cdot \IZ\widehat{_p}
\]
is surjective. 
\end{proof}


\subsection{Comparing $K^0$ for classifying spaces and group $C^*$-algebras}
\label{subsec:Comparing_K0_for_classifying_spaces_and_group_C_ast-algebras}

Throughout this subsection we fix an extension of groups
\[
1 \to \IZ^n \to G \to F \to 1
\] 
such that $F$ is a finite group.

We prove some general results in this setting under assumption
that the conjugation action of $F$ on $\IZ^n$ is free outside the
origin $0 \in \IZ^n$.

\begin{theorem} \label{the:extension_of_Zn_by_F}
Suppose that the conjugation action of $F$ on
$\IZ^n$ is free outside the origin $0 \in \IZ^n$. Then

\begin{enumerate}

\item \label{the:extension_of_Zn_by_F:K1}
We obtain an isomorphism
\[
\omega_1 \colon K_1(C^*_r(G)) \xrightarrow{\cong} K_1(\bub{G});
\]

\item \label{the:extension_of_Zn_by_F:K0}
We obtain a short exact sequence
\[
0 \to \bigoplus_{(M) \in \calm} \widetilde{R}_{\IC}(M) 
\xrightarrow{\bigoplus_{(M) \in \calm} i_M}   K_{0}(C^*_r(G))   
\xrightarrow{\omega_0}  K_{0}(\bub{G}) \to 0.
\]
It splits if we invert $|F|$. 

Consider the map
\[
\overline{k_*} \oplus \bigoplus_{(M) \in \calm} i_M \colon 
\IZ \otimes_{\IZ F} K_0(C^*_r(\IZ^n)) \oplus \bigoplus_{M \in
  \calm}\widetilde{R}_{\IC}(M) \rightarrow K_{0}(C^*_r(G)),
\]
where $\overline{k_*}$ is the homomorphism induced by the inclusion $k
\colon \IZ^n \to G$. It becomes a bijection after inverting $|F|$.
\end{enumerate}
\end{theorem}
\begin{proof}
The group $G$ satisfies by~\cite[Lemma~6.3]{Lueck-Stamm(2000)}
the condition $(N\!M_{\triv \subseteq \calfin})$ appearing 
in~\cite[Notation~2.7]{Lueck-Weiermann(2007)}. Hence the claim  follows from
Corollary~\ref{cor:long_exact_sequence_for_K_ast(Cast(G))_for_G_satisfying_NM} 
except that the map
\[
\overline{k_*} \oplus \bigoplus_{(M) \in \calm} i_M \colon 
\IZ \otimes_{\IZ F} K_0(C^*_r(\IZ^n))  \oplus 
\bigoplus_{(M) \in \calm}\widetilde{R}_{\IC}(M) 
\rightarrow   K_{0}(C^*_r(G)).
\]
becomes bijective after inverting $|F|$.
To prove this, it suffices to show that $\omega_0 \circ \overline{k_*}$ is bijective after inverting $|F|$.
Using the induction structure on equivariant $K$-theory (see \cite[page 198]{Lueck(2002b)}) one checks that
this map agrees with the composite
\[
\IZ \otimes_{\IZ F}  K_0^{\IZ^n}(E\IZ^n) \xrightarrow{u_* \circ \ind_{\IZ^n \to G}} 
K_0^G(EG) \xrightarrow{\ind_{G \to \{1\}}} K_0(BG) \xrightarrow{\overline{f}_*} K_0(\bub{G})
\] 
where 
$u \colon G \times_{\IZ^n} E\IZ^n \to EG$
is the up to $G$-homotopy unique $G$-map.
Because of  Lemma~\ref{lem:K_i(eubG)_to_K_i(bub(G))_is_rationally_split_injective}
it is enough to show  that the composite
\[
\IZ \otimes_{\IZ F}  K_0^{\IZ^n}(E\IZ^n) \xrightarrow{u_* \circ \ind_{\IZ^n \to G}} 
K_0^G(EG) \xrightarrow{\ind_{G \to \{1\}}} K_0(BG)
\] 
is bijective after inverting $|F|$.
Consider the  commutative diagram
\[
\xymatrix@C=27mm{
\IZ \otimes_{\IZ F}  K_0^{\IZ^n}(E\IZ^n)  \ar[r]^{\id_{\IZ} \otimes_{\IZ F} \ind_{\IZ^n \to \{1\}}} 
\ar[d]^{u_* \circ \ind_{\IZ^n \to G}}
&
\IZ \otimes_{\IZ F}  K_0(B\IZ^n) \ar[d]^{\widehat{k_*}}
\\
K_0^G(EG) \ar[r]^{\ind_{G \to \{1\}}} 
&
K_0(BG) 
}
\]
where $\widehat{k_*}$ is induced by the inclusion $k \colon \IZ^n \to G$. 
The upper horizontal arrow is bijective since $\IZ^n$ acts freely
on $E\IZ^n$. The lower horizontal arrow is bijective since $G$ acts
freely on $EG$.  The right vertical arrow is bijective after
inverting $|F|$ because of the Leray-Serre spectral sequence applied to
the fibration $B\IZ^n \to BG \to BF$. This finishes the proof of
Theorem~\ref{the:extension_of_Zn_by_F}.
\end{proof}


\subsection{Topological $K$-theory of classifying spaces}
\label{subsec:Topological_K-theory_of_classifying_spaces}

Throughout this subsection we consider an extension
$1 \to \IZ^n \to \Gamma \to \IZ/m \to 1$ and assume that
the $\IZ/m$-action on $\IZ^n_{\rho}$ is free outside the origin.
We want to compute $K^1(B\Gamma)$ and $K^0(\bub{\Gamma})$.
These turn out to be finitely generated free abelian groups.

We will use the  \emph{Leray-Serre spectral sequence} for
topological $K$-theory (see~\cite[Chapter~15]{Switzer(1975)}) of the
fibration $B\IZ^r \to B\Gamma \to B\IZ/m$ associated to the extension $1 \to \IZ^n \to \Gamma \to \IZ/m \to 1$.  
Recall that its
$E_2$-term is $E^{i,j}_2 = H^i(\IZ/m; K^j(B\IZ^n))$ and it converges to
$K^{i+j}(B\Gamma)$ with no $\text{lim}^1$-term since the inverse system $\{K^{l}(B\Gamma^{(n)}) \mid n \ge 0\}$
given by the inclusions of skeletons of $B\Gamma$ satisfies the Mittag-Leffler condition 
because of~\cite[Theorem~6.5]{Lueck-Oliver(2001b)}.

\begin{lemma} \label{lem:K1(BG)_n_square-free}
  Suppose that the $\IZ/m$-action on   $\IZ^r$ is free outside the origin. Then
  \begin{enumerate}
  \item \label{lem:K1(BG)_n_square-free:collaps}
  All differentials in the Leray-Serre
  spectral sequence for topological $K$-theory associated to 
  the extension are trivial;

  \item \label{lem:K1(BG)_n_square-free:K1(BGamma)}
  The canonical map
  \[
K^1(B\Gamma) \xrightarrow{\cong} K^1(B\IZ^n)^{\IZ/m}
\]
  is bijective. 
  In particular 
  \[
K^1(B\Gamma) \cong \IZ^{s_1},
\]
  where 
  \[
s_1 = \sum_{i \in \IZ} \dim_{\IQ} \left(H^{2i+1}(\IZ^n;\IQ)^{\IZ/m}\right) 
  = \sum_{i \in \IZ} \rk_{\IZ} \left(\bigl(\Lambda^{2i+1} \IZ^n\bigr)^{\IZ/m}\right).
\]
\end{enumerate}
\end{lemma}
\begin{proof}\ref{lem:K1(BG)_n_square-free:collaps} 
  The proof is analogous to the
  proof of~\cite[Lemma~3.5~(iii)]{Davis-Lueck(2010)} but now using Theorem~\ref{the:Tate_cohomology}.
  instead  of~\cite[Lemma~3.5~(ii)]{Davis-Lueck(2010)}.  
  \\[1mm]~\ref{lem:K1(BG)_n_square-free:K1(BGamma)} The argument is analogous
  to the one of the proof
  of~\cite[Theorem~3.3~(iii)]{Davis-Lueck(2010)} using
  assertion~\ref{lem:K1(BG)_n_square-free:collaps} and Theorem~\ref{the:Tate_cohomology}.
  \end{proof}

\begin{lemma}\label{lem:K0(bubG)_n_square-free}
\begin{enumerate}
\item \label{lem:K0(bubG)_n_square-free:K0(bub(G)} 
We have
\[
K^0(\bub{\Gamma}) \cong \IZ^{t_0},
\]
where
\[
t_0 = \sum_{i \in \IZ} \dim_{\IQ} \left(H^{2i}(\IZ^r;\IQ)^{\IZ/m}\right) 
= \sum_{i \in \IZ} \rk_{\IZ} \left(\bigl(\Lambda^{2i}
  \IZ^r\bigr)^{\IZ/m}\right);
\]
\item \label{lem:K0(bubG)_n_square-freeKo_G(eub(G)} 
We have
\[
K^0_{\Gamma}(\eub{\Gamma}) \cong \IZ^{s_0}.
\]
where $s_0 = t_0 + \bigl(\sum_{(M) \in \calm} (|M|-1)\bigr)$.
\end{enumerate}
\end{lemma}
\begin{proof}~\ref{lem:K0(bubG)_n_square-free:K0(bub(G)} 
  We first show that $K^0(\bub{\Gamma})$ is a finitely generated free
  abelian group. There is a finite $CW$-complex model for $\bub{\Gamma}$
  since $\Gamma$ can be mapped with a finite kernel onto a crystallographic
  group.  This implies that $K^0(\bub{\Gamma})$ is finitely generated
  abelian. Next we show that $K^0(\bub{\Gamma})$ is isomorphic to $\IZ^{t_0}$.  We explain how one has to modify arguments
  in~\cite{Davis-Lueck(2010)} to prove this.

The \emph{Atiyah-Hirzebruch
  spectral sequence} (see~\cite[Chapter~15]{Switzer(1975)}) for topological
$K$-theory
\[
E^{i,j}_2 = H^{i}(\bub \Gamma ;K^j(\pt)) \Rightarrow K^{i+j}(\bub \Gamma)
\]
converges since $\bub \Gamma$ has a model which is a finite dimensional
$CW$-complex.  Because of the computations of 
$H^i(\bub{\Gamma})$ of 
Theorem~\ref{the:Cohomology_of_BGamma_and_bub(Gamma)}~\ref{the:Cohomology_of_BGamma_and_bub(Gamma):bub(Gamma)},
the $E^2$-term in the Atiyah-Hirzebruch spectral sequence
converging to $K^*(\bub{\Gamma})$ looks like
\[
E^{i,j}_{2} \cong
\begin{cases}
  \IZ^{\dim_{\IQ} (H^{i}(\IZ^n;\IQ)^{\IZ/m})} & i \; \text{even}, j\; \text{even};
  \\
  \IZ^{\dim_{\IQ} (H^{i}(\IZ^n;\IQ)^{\IZ/m})}  \oplus A_i' & i \; \text{odd}, i \ge 3, j\; \text{even};
  \\
  0 & i=1, j \;\text{even};
  \\
  0 & j\; \text{odd}.
\end{cases}
\]
where each $A_i'$ is a finite abelian group.  The argument in the proof of~\cite[Lemma~3.4]{Davis-Lueck(2010)}
carries over and shows
\[
E^{i,j}_{\infty} \cong
\begin{cases}
  \IZ^{\dim_{\IQ} (H^{i}(\IZ^n;\IQ)^{\IZ/m})} & i \; \text{even}, j\; \text{even};
  \\
  \IZ^{\dim_{\IQ} (H^{i}(\IZ^n;\IQ)^{\IZ/m})}  \oplus A_i & i \; \text{odd}, i \ge 3, j\; \text{even};
  \\
  0 & i=1, j \;\text{even};
  \\
  0 & j\; \text{odd}.
\end{cases}
\]
where each $A_i$ is a finite abelian group. 
Now assertion~\ref{lem:K0(bubG)_n_square-free:K0(bub(G)} 
follows by inspecting the Atiyah-Hirzebruch spectral sequence.
\\[1mm]~\ref{lem:K0(bubG)_n_square-freeKo_G(eub(G)} 
The group $\Gamma$ satisfies by~\cite[Lemma~6.3]{Lueck-Stamm(2000)}
the condition $(N\!M_{\triv \subseteq \calfin})$. Hence the $\Gamma$-pushout~\eqref{Delta-pushout_for_eub(G)} reduces to the $\Gamma$-pushout
\begin{eqnarray}
\xymatrix{
\coprod_{(M) \in \calm} \Gamma \times_{N_{\Gamma} M}  EM 
\ar[r]^-{i} 
\ar[d]
& E\Gamma \ar[d]
\\
\coprod_{(M) \in \calm} \Gamma/M
\ar[r]_-{j} 
& \eub{\Gamma}
}
\label{Delta-pushout_for_eub(G)_NM}
\end{eqnarray}
We obtain from the Mayer-Vietoris sequences
for $K^*_{\Gamma}$ and $K^*$ applied to the $\Gamma$-pushout~\eqref{Delta-pushout_for_eub(G)_NM}
and to the pushout~\eqref{pushout_for_bub(G)_NM}
analogous to the construction in~\cite[Lemma~7.2~(i)]{Davis-Lueck(2010)}
the long exact sequence 
\[
0 \to K^0(\bub{\Gamma}) \to K^0_{\Gamma}(\eub{\Gamma}) \to \bigoplus_{(M) \in \calm} \overline{R}_{\IC}(M) \to
K^1(\bub{\Gamma}) \to K^1_{\Gamma}(\eub{\Gamma}) \to 0,
\]
where $\overline{R}_{\IC}(M)$ is the cokernel of the homomorphism
$R_{\IC}(\{1\}) \to R_{\IC}(M)$ given by restriction with $M \to \{1\}$. 
We have already checked in the proof of Lemma~\ref{lem:K1(eub(Gamma)_to_K(BGamma)_iso_if_MN}
that the kernel of the map $K^1(\bub{\Gamma}) \to K^1_{\Gamma}(\eub{\Gamma})$ is finite.
Since $\overline{R}_{\IC}(M)  \cong \IZ^{|M|-1}$, the claim follows from
assertion~\ref{lem:K0(bubG)_n_square-free:K0(bub(G)}.
\end{proof}

Define for $i \in \IZ$
\[
s_i = 
\begin{cases}
  \bigl(\sum_{(M) \in \calm} (|M|-1)\bigr) + \sum_{l \in \IZ}
  \rk_{\IZ}\bigl((\Lambda^{2l} \IZ^n)^{\IZ/m}\bigr) & \text{if}\; i \;
  \text{even};
  \\
  \sum_{l \in \IZ} \rk_{\IZ}\bigl((\Lambda^{2l+1} \IZ^n)^{\IZ/m}\bigr)
  & \text{if}\; i \; \text{odd.}
\end{cases}
\]

\begin{lemma} \label{lem:equivariant_K-cohomology_of_eubG_square-free}
Consider $m \in \IZ$. Then
\begin{enumerate}
\item \label{lem:equivariant_K-cohomology_of_eubG_square-free:K_m(Zr)Zn}
The abelian group $K_i(C^*_r(\IZ^n))^{\IZ/m}$ is finitely generated free of rank\\
$\sum_{l \in \IZ} \rk_{\IZ}\bigl((\Lambda^{2l+i} \IZ^n)^{\IZ/m}\bigr)$;
\item \label{lem:equivariant_K-cohomology_of_eubG_square-free:cohomology}
The abelian group $K^i_{\Gamma}(\eub{\Gamma})$ is finitely generated free of rank $s_i$;
\item \label{lem:equivariant_K-cohomology_of_eubG_square-free:homology}
The abelian group $K_i^{\Gamma}(\eub{\Gamma})$ is finitely generated free  of rank $s_i$;
\item \label{lem:equivariant_K-cohomology_of_eubG_square-free:Cast}
The abelian group $K_i^{\Gamma}(C^*_r(\Gamma))$ is finitely generated free  of rank $s_i$.
\end{enumerate}
\end{lemma}
\begin{proof}~\ref{lem:equivariant_K-cohomology_of_eubG_square-free:K_m(Zr)Zn}
  Obviously $K_i(C^*_r(\IZ^n))^{\IZ/m}$ is finitely generated free since
  $K_i(C^*_r(\IZ^n))$ is finitely generated free. By the
  Baum-Connes Conjecture $K_i(C^*_r(\IZ^n))^{\IZ/m} \cong
  K_i(B\IZ^n)^{\IZ/m}$.  For every $CW$-complex $X$, e.g., $X =
  B\IZ^n$, the Chern character yields a natural isomorphism
\[
\bigoplus_{k \in \IZ} H_{i+2k}(X;\IQ) \xrightarrow{\cong} K_i(X)\otimes_{\IZ} \IQ.
\]
Now the claim follows from the isomorphisms
\[
H_i(B\IZ^n;\IQ)^{\IZ/m} \cong H^i(B\IZ^n;\IQ)^{\IZ/m}
\cong \left((\Lambda^i \IZ^n) \otimes_{\IZ} \IQ\right)^{\IZ/m}.
\]%
\ref{lem:equivariant_K-cohomology_of_eubG_square-free:cohomology} This
follows from Lemma~\ref{lem:K1(eub(Gamma)_to_K(BGamma)_iso_if_MN},
Lemma~\ref{lem:K1(BG)_n_square-free} and
Lemma~\ref{lem:K0(bubG)_n_square-free}.
\\[1mm]~\ref{lem:equivariant_K-cohomology_of_eubG_square-free:homology}
This follows from
assertion~\ref{lem:equivariant_K-cohomology_of_eubG_square-free:cohomology}
and the universal coefficient theorem for equivariant $K$-theory as
explained in~\cite[Theorem~8.3~(ii)]{Davis-Lueck(2010)}.
\\[1mm]~\ref{lem:equivariant_K-cohomology_of_eubG_square-free:Cast}
This follows from
assertion~\ref{lem:equivariant_K-cohomology_of_eubG_square-free:homology}
and the Baum-Connes Conjecture which holds for $\Gamma$~(see~\cite{Higson-Kasparov(2001)}).
\end{proof}

Now we can give the proof of Theorem~\ref{the:Topological_K-theory_for_Zr_rtimes_Z/n}.

\begin{proof}[Proof of
  Theorem~\ref{the:Topological_K-theory_for_Zr_rtimes_Z/n}]
  \ \\~\ref{the:Topological_K-theory_for_Zr_rtimes_Z/n:K1_abstract}
  The map $\omega_1\colon K_1(C^*_r(\Gamma)) \xrightarrow{\cong}
  K_1(\bub{\Gamma})$ is bijective by
  Theorem~\ref{the:extension_of_Zn_by_F}~\ref{the:extension_of_Zn_by_F:K1}.
  The composite
\[
\IZ \otimes_{\IZ[\IZ/m]} K_1(C_r^*(\IZ^n)) \xrightarrow{k_*}
K_1(C^*_r(\Gamma)) \xrightarrow{k^*} K_1(C^*_r(\IZ^n))^{\IZ/m}
\] 
is
multiplication with the norm element $N_{\IZ/m}$ by the Double Coset Formula
(see~\cite[Lemma~3.2]{Echterhoff-Lueck-Phillips-Walters(2010)}).
The kernel of this map is $\widehat{H}^{-1}(\IZ/m;K_1(C^*_r(\IZ^n)))$ and
its cokernel of this map is $\widehat{H}^0(\IZ/m;K_1(C^*_r(\IZ^n)))$.
The group $\widehat{H}^0(\IZ/m;K_1(C^*_r(\IZ^n)))$ vanishes 
by Theorem~\ref{the:Tate_cohomology}  and
$\widehat{H}^1(\IZ/m;K_1(C^*_r(\IZ^n)))$ is annihilated by $m$.  Hence
the composite above and therefore $k^* \colon K_1(C^*_r(\Gamma)) \to
K_1(C^*_r(\IZ^n))^{\IZ/m}$ are surjective.  The abelian group
$K_1(C^*_r(\Gamma))$ is a finitely generated free abelian group by
Lemma~\ref{lem:equivariant_K-cohomology_of_eubG_square-free}.
Obviously this is also true for $K_1(C^*_r(\IZ^n))^{\IZ/m}$.  They have
the same rank by
Lemma~\ref{lem:equivariant_K-cohomology_of_eubG_square-free}.  Hence
the epimorphism $k^* \colon K_1(C^*_r(\Gamma)) \to
K_1(C^*_r(\IZ^n))^{\IZ/m}$ must be bijective. This implies that the map
$k_* \colon \IZ \otimes_{\IZ[\IZ/m]} K_1(C_r^*(\IZ^n)  \to K_1(C^*_r(\Gamma))$
is surjective and its kernel is $\widehat{H}^{-1}(\IZ/m;K_1(C^*_r(\IZ^n)))$.
\\[1mm]~\ref{the:Topological_K-theory_for_Zr_rtimes_Z/n:K0_abstract}
The composite
\[
\overline{k^*} \circ \overline{k_*} \colon 
\IZ  \otimes_{\IZ[\IZ/m]} \widetilde{K}_{0}(C^*_r(\IZ^n))  
\to \widetilde{K}_0(C^*_r(\IZ^n))^{\IZ/m}
\]
is by   given by multiplication
with the norm element $N_{\IZ/m}$ of $\IZ/m$ by the Double Coset Formula 
(see~\cite[Lemma~3.2]{Echterhoff-Lueck-Phillips-Walters(2010)}). Since $\widehat{H}^{-1}(\IZ/m,K_0(C^*_r(\IZ^n))$
vanishes by Theorem~\ref{the:Tate_cohomology}, we conclude that it is injective.
Hence $ \IZ \otimes_{\IZ[\IZ/m]} \widetilde{K}_{0}(C^*_r(\IZ^n)) $
is finitely generated free. Now apply
Theorem~\ref{the:extension_of_Zn_by_F}~\ref{the:extension_of_Zn_by_F:K0}. 
\\[1mm]~\ref{the:Topological_K-theory_for_Zr_rtimes_Z/n:explicite}
This follows from
Lemma~\ref{lem:equivariant_K-cohomology_of_eubG_square-free}%
~\ref{lem:equivariant_K-cohomology_of_eubG_square-free:Cast}.
\\[1mm]~\ref{the:Topological_K-theory_for_Zr_rtimes_Z/n:non-torsion:n_even}
Since $m$ is even, $\IZ/2$ is a subgroup of $\IZ/m$.  Since the
$\IZ/m$-action on $\IZ^n$ is free outside the origin, the
$\IZ/2$-action on $\IZ^n$ must be given by $-\id$.  Hence the induced
$\IZ/2$-action on $\Lambda^i \IZ^n$ is given by $- \id$ for odd
$i$. Hence $(\Lambda^i \IZ^n)^{\IZ/m}$ vanishes for odd $i$.  This
implies $s_1 = 0$. Now the claim follows from
assertion~\ref{the:Topological_K-theory_for_Zr_rtimes_Z/n:explicite}.
This finishes the proof of
Theorem~\ref{the:Topological_K-theory_for_Zr_rtimes_Z/n}
\end{proof}


\typeout{-------- Section 5: Conjugacy classes of finite subgroups and cohomology  ---------------}

\section{Conjugacy classes of finite subgroups and cohomology}
\label{sec:Conjugacy_classes_of_finite_subgroups_and_cohomology}

Consider the group extension $1 \to \IZ^n \to \Gamma \xrightarrow{\pi} \IZ/m\to 1$ as it appears in
Theorem~\ref{the:Topological_K-theory_for_Zr_rtimes_Z/n}. In particular we assume that the conjugation action of
$\IZ/m$ on $\IZ^n$ is free outside the origin. For the remainder of this paper
we will abbreviate $L = \IZ^n$ and $G = \IZ/m$.

\begin{lemma}\label{lem:maximal_maximal_finite_subgroups} \
  \begin{enumerate}

  \item \label{lem:maximal_maximal_finite_subgroups:H_max} 
    Let $H \subseteq  \Gamma$ be a non-trivial finite subgroup. Then there is a subgroup 
   $H_{\max} \subseteq \Gamma$ which is uniquely determined by the properties
    that $H \subseteq H_{\max}$ and $H_{\max}$ is a maximal finite subgroup,
    i.e., for every finite subgroup $K \subseteq \Gamma$ with 
    $H_{\max} \subseteq K$ we have $K = H_{\max}$;

  \item \label{lem:maximal_maximal_finite_subgroups:N_H_max} If $H$ is a maximal
    finite subgroup of $\Gamma$, then $N_{\Gamma} H = H;$

  \item \label{lem:maximal_maximal_finite_subgroups:NH} If $H \subseteq \Gamma$
    is a non-trivial finite subgroup of $\Gamma$, then $H_{\max} = N_{\Gamma}H$.

  \end{enumerate}
\end{lemma}
\begin{proof}
Assertions~\ref{lem:maximal_maximal_finite_subgroups:H_max} and~\ref{lem:maximal_maximal_finite_subgroups:N_H_max}
follow from~\cite[Lemma~6.3]{Lueck-Stamm(2000)}. They imply assertion~\ref{lem:maximal_maximal_finite_subgroups:NH}
since the action of $G$ on $L$ is free outside the origin and therefore $N_{\Gamma}H \cap L = \{0\}$ and hence
$|N_{\Gamma}H|< \infty$.
\end{proof}

\begin{notation}
  Define for a subgroup $C \subseteq G$
  \[
  \calc(C) := \{(H) \mid H \subseteq \Gamma, \pi(H) = C\},
  \]
  where $(H)$ denotes the conjugacy class of the subgroup $H \subseteq \Gamma$
  within $\Gamma$.  Define
  \[
  \calm(C) := \{(H) \in \calc(C) \mid H \; \text{ is a maximal finite subgroup
    of} \; \Gamma\}.
  \]
  Put
  \[
   \calm := \{(H)  \mid H \; \text{ is a maximal finite subgroup of} \; \Gamma\}.
   \]
\end{notation}

 We have $\calm = \coprod_{\{1\} \subsetneq C \subseteq G} \calm(C)$.

If $K \subseteq G$ is a subgroup, then we will abbreviate $H^1(K; \res_G^HL)$
by $H^1(K;L)$
when it is clear from the context to which subgroup the $G$-action on $L$ is restricted.

\begin{lemma} \label{lem:bijection_cohomology_calc}
Let $t$ be a generator $t \in G$. Consider a non-trivial subgroup $C \subseteq G$.  
Put $k = [G:C]$. Then $t^k$ is a generator of $C$ and there are bijections
\[
H^1(C;L)/G \xrightarrow{\cong} \cok\bigl((t^k - 1) \colon L \to L\bigl)/G 
\xrightarrow{\cong} \calc(C),
\]
where the $G$-operation on the first and second term above comes
from the $G$-action on $L$.
\end{lemma}
\begin{proof}
  By definition $H^1(C;L)$ is the quotient of the kernel of the map $N
  \colon L \to L$ coming from multiplication with the norm element 
  $N_C = \sum_{g     \in C} g$ by the image of the map $(t^k - 1) \colon L \to L$. For any $l \in
  L$ the element $N_C\cdot l$ lies in $L^C$. Since the $G$-action on $L$ is free
  outside the origin and $C$ is non-trivial, $L^C = \{0\}$. We conclude that 
  $N_C   \colon L \to L$ is the trivial map. This explains the first isomorphism.

  Fix an element $\gamma \in \Gamma$ of order $m = |G|$ with $\pi(\gamma) = t$.
  The second map sends the class of
  $l$ to the conjugacy class of the subgroup $\langle \gamma^{-k} \cdot  l\rangle$ of
  $\Gamma$ generated by $\gamma^{-k} \cdot l$.  We have to show that it is well-defined.
  We compute $(\gamma^{-k} \cdot l)^{|C|} =  \gamma^{-k \cdot |C|} \cdot (N_C \cdot l) = \gamma^{-|G|} = 1$. 
  Hence $\langle \gamma^{-k} \cdot  l\rangle$
  is a finite subgroup of $\Gamma$ whose image under the projection $\pi \colon \Gamma \to G$
  is $C$.   Now suppose for two elements $l_0,l_1 \in L$ that they yield 
  the same class in  $\cok\bigl((t^k - 1) \colon L \to L\bigl)/G$.
  Then there exists $j \in \IZ/m$ and $l' \in L$ with $l_1 = t^j \cdot l_0 + (t^k -1) \cdot l'$. Put
  $l'' = t^{-j} \cdot l'$. Then $t^j \cdot \bigl((t^k -1) \cdot  l'' + l_0\bigr) = l_1$ and hence
  $(\gamma^j \cdot l'')(\gamma^{-k}   l_0)(\gamma^j \cdot l'')^{-1} = \gamma^{-k} l_1$. This implies
  that $\langle \gamma^{-k} l_0\rangle$ and $\langle \gamma^{-k} l_1\rangle$ are
  conjugated in $\Gamma$. 

  The second map is surjective since any finite subgroup $H \subseteq \Gamma$
  with $\pi(H) = C$ can be written $\langle \gamma^{-k} \cdot l\rangle $ for an
  appropriate element $l \in L$.
  
  Finally we prove injectivity.  Notice for the sequel that any element in
  $\Gamma$ can be written uniquely in the form $\gamma^j \cdot l'$ for some $j
  \in \IZ/m$ and $l' \in L$. Suppose that $\langle \gamma^{-k} l_0\rangle$ and
  $\langle \gamma^{-k} l_1\rangle$ are conjugated in $\Gamma$ for $l_0,l_1 \in
  L$. Then there exists a natural number $i \in \IZ/|C|^{\times}$ such that for
  appropriate $j \in \IZ/m$ and $l' \in L$ we have $(\gamma^j \cdot l')(\gamma^{-k}
  l_0)(\gamma^j \cdot l')^{-1} = (\gamma^{-k} l_1)^i$.  Applying $\pi$ to this
  equation, yields $t^{-k} = t^{-ki}$ and hence $i = 1 \mod |C|$. Hence get
  $(\gamma^j \cdot l')(\gamma^{-k} l_0)(\gamma^j \cdot l')^{-1} = \gamma^{-k} l_1$.
  We conclude $\gamma^j \cdot \bigl(\gamma^k l'\gamma^{-k}  l_0(l')^{-1}\bigr)\gamma^{-j} = l_1$. 
  This yields, when we use in $L$ the
  additive notation, $t^j \cdot \bigl((t^k -1) l' + l_0\bigr) = l_1$. If we put
  $l'' = t^j \cdot l'$, we conclude
  \[
  l_1 = t^j \cdot l_0 + (t^k -1) \cdot l''.
  \]
  This means that $l_0$ and $l_1$ define the same class in 
  $\cok\bigl((t^k - 1)  \colon L \to L\bigl)/G$.
   
\end{proof}

Let $C \subseteq D \subseteq G$ be  subgroups of $G$. Since 
$\pi^{-1}(C) \subseteq \Gamma$ is normal in $\Gamma$, we can define a map
\begin{eqnarray}
i_{C \subseteq D} \colon \calc(D) \to \calc(C) & & (H) \mapsto (H \cap \pi^{-1}(C)).
\label{calc(D)_to_calc(C)}
\end{eqnarray}

\begin{lemma} \label{lem:calc(D)_to_calc(C)} The map $i_{C\subseteq D} \colon
  \calc(D) \to \calc(C)$ is injective if $C$ is non-trivial.
\end{lemma}
\begin{proof}
  Consider $(H_1)$ and $(H_2)$ such that $i_{C\subseteq D}((H_1)) =
  i_{C\subseteq D}((H_2))$ holds.  This means $(H_1 \cap \pi^{-1}(C)) = (H_2  \cap \pi^{-1}(C))$. 
  We can choose the representatives $H_1$ and $H_2$ such
  that $H_1 \cap \pi^{-1}(C) = H_2 \cap \pi^{-1}(C)$. Since $\pi$ maps $H_1 \cap
  \pi^{-1}(C)$ onto $C$ and $C$ is non-trivial, $H_1 \cap \pi^{-1}(C)$ is
  non-trivial. Let $K$ be the maximal finite subgroup uniquely determined by the
  property $H_1 \cap \pi^{-1}(C) \subseteq K$. The existence of $K$ and the fact
  that $K$ contains both $H_1$ and $H_2$ as subgroups follows from
  Lemma~\ref{lem:maximal_maximal_finite_subgroups}~\ref{lem:maximal_maximal_finite_subgroups:H_max}.
  Hence $H_1 = K \cap \pi^{-1}(D) = H_2$. This implies $(H_1) = (H_2)$.
\end{proof}

\begin{notation} \label{calc(C,D)}
For subgroups $C \subseteq D \subseteq G$ put
\[
\calc(C;D) = \im\bigl(i_{C \subseteq D} \colon \calc(D) \to \calc(C)\bigr)
\]
\end{notation}
\begin{lemma}
\label{lem:calc(C;d_1)_capcalc(C,D_2)}
Let $C, D_1, D_2$ be subgroups of $G$ with $C \not= \{1\}$, $C \subseteq D_1$
and $C \subseteq D_2$. Then
\[
\calc(C;D_1) \cap \calc(C,D_2) = \calc(C;\langle D_1,D_2\rangle).
\]
\end{lemma}
\begin{proof}
Consider $(H) \in \calc(C;D_1) \cap \calc(C,D_2)$. Choose $(H_1) \in
\calc(D_1)$ and $(H_2) \in \calc(D_2)$ with $(H_1 \cap \pi^{-1}(C)) = (H_2
\cap \pi^{-1}(C)) = (H)$. We can choose the representatives $H$, $H_1$ and $H_2$
such that $H_1 \cap \pi^{-1}(C) = H_2 \cap \pi^{-1}(C) = H$. The group $H$ is
non-trivial because of $\pi(H) = C$.  We conclude from
Lemma~\ref{lem:maximal_maximal_finite_subgroups}~\ref{lem:maximal_maximal_finite_subgroups:H_max}
that there exists a maximal finite subgroup $K$ which contains $H$, $H_1$ and
$H_2$ as subgroup.  In particular $K$ contains $\langle H_1,H_2 \rangle$. Hence
$\langle H_1,H_2\rangle$ is a finite subgroup of $\Gamma$. It 
satisfies $\pi(\langle H_1,H_2\rangle) = \langle D_1,D_2\rangle$.  This implies that
$\langle H_1,H_2 \rangle \cap \pi^{-1}(C) = H$ and hence that $(H)$ belongs to
$\calc(C;\langle D_1,D_2 \rangle)$. We conclude
$\calc(C;D_1) \cap \calc(C;D_2) \subseteq \calc(C;\langle D_1, D_2\rangle)$.
Since $\calc(C;D_1) \cap \calc(C;D_2) \supseteq \calc(C;\langle D_1,
D_2\rangle)$ is obviously true, Lemma~\ref{lem:calc(C;d_1)_capcalc(C,D_2)}
follows.
\end{proof}

\begin{theorem}[The order of $\calc(M)$] \label{the:calm(C)} Let $C \subseteq G$
  be a non-trivial subgroup. Let $p_1, p_2, \ldots , p_s$ be the prime numbers
  dividing $[G:C]$ for $i = 1,2 \ldots , s$. Denote by $D_i \subseteq G$ the
  subgroup such that $C \subseteq D$ and $[D:C] = p_i$.  For a non-trivial
  subset $I \subseteq \{1,2, \ldots, s\}$ denote by $D_I$ the subgroup of $G$
  uniquely determined by $[D:C] = \prod_{i \in I} p_i$.  Then
\begin{enumerate}

\item \label{the:calm(C):calm(C)}
$\calm(C) = \calc(C) \setminus \left(\bigcup_{i=1}^s \calc(C;D_i)\right);$

\item \label{the:calm(C):|calm(C)|}
$|\calm(C)| = |\calc(C)| + \sum_{I \subseteq \{1,2, \ldots, s\}, I \not= \emptyset} \;(-1)^{|I|} \cdot |\calc(D_I)|$.

\item \label{the:calm(C):|calm(C)|_and_H} 
We have
\[
|\calc(C)| = |H^1(C; L)/(G/C)| = \left|\cok\bigl((t^k - 1) \colon L \to L\bigl)/(G/C)\right|,
\]
where $t$ is a generator of $G$, $k = [G:C]$ and the $G/C$-action comes from the $G$-action on $L$;

\item \label{the:calm(C):|calm(C)|_and_H:|calm(C)_and_homology}
Put $k = [G:C]$. For $I \subseteq \{1,2, \ldots, s\}, I \not= \emptyset$ 
define $k_I := \frac{k}{\prod_{i \in I} p_i}$. Then
\begin{eqnarray*}
|\calm(C)| 
& = &
 |H^1(C; L)/(G/C)| 
\\
& & 
\quad \quad + \sum_{I \subseteq \{1,2, \ldots, s\}, I \not= \emptyset} \;(-1)^{|I|} \cdot |H^1(D_I;L)/(G/D_I)|
\\
& = &  
\left|\cok\bigl((t^{k} - 1) \colon L \to L\bigl)/(G/C)\right|
\\
& & 
\quad \quad + \sum_{I \subseteq \{1,2, \ldots, s\}, I \not= \emptyset} \;(-1)^{|I|} 
\cdot \left|\cok\bigl((t^{k_I} - 1) \colon L \to L\bigl)/(G/D_I)\right|.
\end{eqnarray*}
\end{enumerate}
\end{theorem}
\begin{proof}~\ref{the:calm(C):calm(C)}
If $D \subseteq G$ is any subgroup with $C \subsetneq D$, 
then there is $i \in \{1,2, \ldots, s\}$ with $C \subseteq D_i \subseteq D$.
Since $\calm(C)$ is the complement in $\calc(C)$ of the union of the subsets
$\calc(C;D)$ for all $D$ with $C \subsetneq D$, assertion~\ref{the:calm(C):calm(C)} follows.
\\[1mm]~\ref{the:calm(C):|calm(C)|} Obviously
\[
|\calm(C)| = |\calc(C)| - \left|\bigcup_{i = 1}^s \calc(C,D_i) \right|.
\]
By the classical Inclusion-Exclusion Principle we get
\[
\left|\bigcup_{i = 1}^s \calc(C,D_i) \right|
= \sum_{I \subseteq \{1,2, \ldots, s\}, I \not= \emptyset} \;(-1)^{|I|} \cdot 
\left|\bigcap_{i \in I} \calc(C;D_i) \right|.
\]
Since $\bigcap_{i \in I} \calc(C;D_i)  = \calc(C;D_I)$ 
by Lemma~\ref{lem:calc(C;d_1)_capcalc(C,D_2)}
and $\calc(C;D_I) = \calc(D_I)$ by Lemma~\ref{lem:calc(D)_to_calc(C)},
assertion~\ref{the:calm(C):|calm(C)|} follows.
\\[1mm]~\ref{the:calm(C):|calm(C)|_and_H}
This is a direct consequence of Lemma~\ref{lem:bijection_cohomology_calc}.
\\[1mm]~\ref{the:calm(C):|calm(C)|_and_H:|calm(C)_and_homology}
This follows directly from assertions~\ref{the:calm(C):|calm(C)|}
and~\ref{the:calm(C):|calm(C)|_and_H}.
\end{proof}

\begin{remark} Theorem~\ref{the:calm(C)} gives an explicit
formula how one can determine $|\calm(C)|$ for all subgroups $C \subseteq G$ if one knows
the numbers  
\[
|H^1(C; L)/(G/C)| = \left|\cok\bigl((t^{[G:C]} - 1) \colon L \to L\bigl)/(G/C)\right|
\]
for all subgroups $C \subseteq G$. One can even identify
$\calm(C)$ for all subgroups $C \subseteq G$ if one knows
the sets
\[
H^1(C; L)/(G/C) = \cok\bigl((t^{[G:C]} - 1) \colon L \to L\bigl)/(G/C)
\]
for all subgroups $C \subseteq G$.
\end{remark}


\typeout{-------------------------------- Section 6: The prime power case  --------------------}

\section{The prime power case}
\label{sec:The_prime_power_case}

The situation simplifies if one considers the case where $m = |G|$ is a prime power.
We conclude from Theorem~\ref{the:calm(C)}:

\begin{example}\label{exa:p-power-order}
Suppose that $|G| = p^r$ for some prime $p$ and some natural number $r$. 
Let $G_j \subseteq G$ be the subgroup of order $p^j$ for $j = 1,2, \ldots r$.
Then
\[
|\calm(G_j)| 
 = 
\begin{cases}
 |\calc(G_j)| - |\calc(G_{j+1})| & \text{if}\;  1 \le j < r;
\\
|\calc(G)| & \text{if}\;  j =  r,
\end{cases}
\]
and we have
\[
|\calc(G_j)| = \left|H^1\bigl(G_j; L)/(G/G_j\bigr)\right| 
= \left|\cok\bigl((t^{p^{r-j}} - 1) \colon L \to L\bigl)/(G/G_j)\right|.
\]
\end{example}

\begin{lemma} \label{lem:z[zeta]}
Suppose that $m = |G|$ is a prime power $p^r$. Put $\zeta = \exp(2 \pi i/p^r)$.
Let $G_j \subseteq G$ be the subgroup of order $p^j$ for $j = 1,2, \ldots , r$. 

\begin{enumerate}

\item \label{lem:z[zeta]:ZG/G_1j_iso}
There is an isomorphism
\[
H^1\bigl(G_j;\IZ[\zeta]\bigr) \cong \IZ/p[G/G_j],
\]
which is compatible with the $G/G_j$-action on the source coming from the $G$-action on
$L=\IZ[\zeta]$ and the  $G/G_j$-permutation action on the target;

\item \label{lem:z[zeta]:order}
We get for all $j = 1,2, \ldots, r$ and natural numbers $k$
\[
\left|H^1\bigl(G_j;\IZ[\zeta]^k\bigr)/(G/G_j)\right| = \left|\IZ/p[G/G_j]^k/(G/G_j)\right|,
\]
where $\IZ[\zeta]^k$ and $\IZ/p[G/G_j]^k$ denote the $k$-fold direct sum, or, 
equivalently, $k$-fold direct product.
\end{enumerate}
\end{lemma}
\begin{proof}~\ref{lem:z[zeta]:ZG/G_1j_iso}
 There is an exact sequence of $\IZ G$-modules
$0 \to \IZ[G/G_1] \to \IZ G \to \IZ[\zeta] \to 0$
which comes from the $G$-isomorphism $\IZ G / T \cdot \IZ G \cong \IZ[\zeta]$, where
$t \in G$ is  a generator and $T=1+t^{p^{r-1}}+t^{2p^{r-1}}+\dots+t^{(p-1)p^{r-1}}\in\IZ G$.
The associated long exact cohomology sequence induces an isomorphism, 
compatible with the $G/G_j$-actions coming from the $G$-actions on 
$\IZ[G/G_1]$ and $\IZ[\zeta]$,
\[
H^1\bigl(G_j;\IZ[\zeta]\bigr) 
\cong 
H^2\bigl(G_j;\IZ[G/G_1] \bigr). 
\]

Let $q \colon G/G_1 \to G/C_j$ be the projection. Fix a map of sets
$\sigma \colon G/G_j \to G/G_1$ such that $eG_j$ is sent to $eG_1$ and $q \circ \sigma = \id$. 
Next we define  to one another inverse $G_j/G_1$-maps
\begin{align*}
\phi \colon G/G_1 & \xrightarrow{\cong}  \coprod_{G/G_j} G_j/G_1;
\\
\psi \colon \coprod_{G/G_j} G_j/G_1 & \xrightarrow{\cong} G/G_1.
\end{align*}
The map $\phi$ sends $g G_1$ to the element $g \cdot \sigma \circ q(gG_1)^{-1} \in G_j/G_1$ 
in the summand belonging to $q(g G_1)$. The map $\psi$ sends the
element $u G_1 \in G_j/G_1$ in the summand belonging to $g G_j$ to $u
\sigma(gG_j)$. There is an obvious $G$-action on $G/G_1$.  There is precisely
one $G$-action on $\coprod_{G/G_j} G_j/G_1$ for which $\phi$ and $\psi$ become
$G$-maps.  Namely, given $g_0 \in G$ and an element $u G_1 \in G_j/G_1$ in the
summand belonging to $g G_j$, the action of $g_0$ on $u G_1 \in G_j/G_1$ yields
the element $u_0 \cdot uG_1$ in the summand belonging to $g_0g G_j$ for the
element $u_0 = g_0 \cdot \sigma(gG_j) \cdot \sigma(q(g_0) \cdot
gC_{p^{j}})^{-1}$ in $G_j$.  The map $H^2\bigl(G_j,\IZ[G_j/G_1]\bigr) \to
H^2\bigl(G_j,\IZ[G_j/G_1]\bigr)$ induced by multiplication with $u_0$ on
$\IZ[G_j/G_1]$ is the identity since for a projective $\IZ[G_j]$-resolution
$P_*$ the $\IZ[G_j]$-chain map $P_* \to P_*$ given by multiplication with $u_0$
induces the identity on homology and hence is $\IZ[G_j]$-chain homotopic to the
identity. Hence the isomorphism induced by $\phi$
\[
H^2\bigl(G_j;\IZ[G/G_1] \bigr) \xrightarrow{\cong}
\bigoplus_{G/G_j} H^2\bigl(G_j;\IZ[G_j/G_1] \bigr)
\]
is compatible with the $G/G_j$-actions if we use on the source the one coming from the $G$-action on 
$\IZ[G/G_1]$ and on the target the $G/G_j$-permutation action. By Shapiro's lemma
(see~\cite[(5.2) in VI.5 on page~136]{Brown(1982)})
\[
H^2\bigl(G_j;\IZ[G_j/G_1] \bigr) \cong H^2\bigl(G_1;\IZ\bigr) \cong H^2/\IZ/p) \cong \IZ/p.
\]
Now assertion~\ref{lem:z[zeta]:ZG/G_1j_iso} follows.
\\[1mm]~\ref{lem:z[zeta]:order} 
We get from assertion~\ref{lem:z[zeta]:ZG/G_1j_iso} bijection of $G/G_j$-sets
\[
H^1\bigl(G_j;\IZ[\zeta]^k\bigr)
\cong 
H^1\bigl(G_j;\IZ[\zeta]\bigr)^k
\cong \IZ/p\bigl[G/G_j\bigr]^k.
\]
It induces a bijection
\[
H^1\bigl(G_j;\IZ[\zeta]^k\bigr)/(G/G_j) \xrightarrow{\cong}
\IZ/p\bigl[G/G_j\bigr]^k/(G/G_j).
\]
\end{proof}

\begin{lemma} \label{lem:map(Hk,P)/H)_parity} 
   Let $p$ be a prime. Put $P = \{0,1, \ldots, (p-1)\}$. 
  Let $H$ be a finite cyclic group of order $p^r$ for $r \ge 1$. Let $k$ be a natural number. 
  Define for a set $T$
  \begin{eqnarray*}
  map(T,P)_{\ev} &:= & \bigl\{f \colon T \to P \mid \sum_{h \in H} f(h) \equiv 0 \mod 2\bigr\};
   \\
  map(T,P)_{\odd} &:= & \bigl\{f \colon T \to P \mid \sum_{h \in H} f(h) \equiv 1 \mod 2\bigr\}.
  \end{eqnarray*}
  Denote by $H_j$ the subgroup of $H$ of order $p^j$
  for $j = 0,1,2, \ldots, r$. Then

\begin{enumerate}

\item \label{lem:map(Hk,P)/H):ev-odd} 
We get in the Burnside ring $A(H)$ for $p \not= 2$
\begin{eqnarray*}
\bigl(\map(H,P)_{\ev} - \map(H,P)_{\odd}\bigr)^k 
& = & 
[H/H], 
\end{eqnarray*}
and for $p = 2$
\begin{multline*}
\left(\bigl[\map(H,P)_{\ev}\bigr] - \bigl[\map(H,P)_{\odd}\bigr]\right)^k
\\
 =  
2^k \cdot [H/H]  - 2^{k2^{r-1}-r} \cdot [H] +  \sum_{i=1}^{r-1} \bigl(2^{k2^{r - i}-r +i} - 2^{k2^{r-i-1}-r+i}\bigr) \cdot [H/H_i];
\end{multline*}

\item \label{lem:map(Hk,P)/H):all} 
We get for all primes $p$ in $A(H)$
\begin{eqnarray*}
\bigl[\map(H,P)^k\bigr] 
& = & 
p^k \cdot [H/H] + \sum_{i=0}^{r-1} \bigl(p^{kp^{r - i}-r +i} - p^{kp^{r-i-1}-r +i}\bigr) \cdot [H/H_i].
\end{eqnarray*}
Moreover, we have for all primes $p$ the equality of integers
\[
\bigl|\map(H,P)^k/H\bigr| = p^k + \sum_{i=0}^{r-1} \bigl(p^{kp^{r - i}-r+i} - p^{kp^{r-i-1}-r+i}\bigr).
\]
\end{enumerate}
\end{lemma}
\begin{proof}
 For $i = 0,1,2, \ldots ,r$ we have the ring homomorphism
  \[
   \ch_i \colon A(H) \to  \IZ, \quad [S] \mapsto \bigl|S^{H_i}\bigr|.
   \]
  Consider the element in $A(H)$
  \[
  x = \sum_{j = 0}^r a_j \cdot [H/H_j].
  \]
  Then we get $\ch_i(x) = \sum_{j = i}^r a_j \cdot p^{r-j}$ for $i = 0,1,2, \ldots, r$.
  This implies 
  \begin{eqnarray}
   a_j = 
   \begin{cases}
    \ch_j(x) & j = r;
    \\
    \frac{\ch_j(x) - \ch_{j+1}(x)}{p^{r-j}}
    & j = 0,1,2, \ldots, (r-1).
  \end{cases}
  \label{a_j_from_ch}
  \end{eqnarray}  
 One easily checks for a set $T$ by induction over its cardinality
 \begin{eqnarray}
 \map(T,P)_{\ev} 
  & = &
 \begin{cases}
  \frac{p^{|T|}+1}{2} & p \not= 2;
  \\
  2^{|T|-1}& p =2; 
\end{cases}
  \label{cardinality_even}
  \\
  \map(T,P)_{\odd} 
  & = &
 \begin{cases}
  \frac{p^{|T|}-1}{2} & p \not= 2;
  \\
  2^{|T|-1} & p =2.
\end{cases}
\label{cardinality_odd}
\end{eqnarray}
Let $\pr_j \colon H\to H/H_j$ be the projection. It induces a bijection
\begin{eqnarray}
& \pr_j^* \colon \map(H/H_j,P) \xrightarrow{\cong} \map(H,P)^{H_j}. &
\label{map(H,P)H__andmap(H(H_j,P)}
\end{eqnarray}
Given $f \colon H/H_j \to P$, we get
\[
|H_j| \cdot \sum_{hH_j \in H/H_j} f(hH_j) = \sum_{h \in H} \pr_j^*(f)(h).
\]
Hence $\pr_j^*$ induces bijections
\begin{eqnarray}
\map(H/H_j,P)_{\ev} & \cong & \bigl(\map(H,P)_{\ev}\bigr)^{H_j};
\label{ev_and_quotients_p_not_2}
\\
\map(H/H_j,P)_{\odd} & \cong & \bigl(\map(H,P)_{\odd}\bigr)^{H_j},
\label{odd_and_quotients_p_not_2}
\end{eqnarray}
provided that $p \not= 2$. If $p = 2$, we get for $j \ge 1$
\begin{eqnarray}
\bigl(\map(H,P)_{\odd}\bigr)^{H_j} = \emptyset.
\label{odd_for_p_is_2}
\end{eqnarray}      
We conclude from~\eqref{cardinality_even} and~\eqref{ev_and_quotients_p_not_2} for $p \not= 2$
\begin{eqnarray}
\ch_j\bigl([\map(H,P)_{\ev}]\bigr)
\label{ch_j(ev_p_odd)}
& = & 
\left|\map(H,P)_{\ev}^{H_j}\right|
\\
& = & 
\left|\map(H/H_j,P)_{\ev}\right|
\nonumber
\\
& = & 
\frac{p^{p^{r-j}} + 1}{2}.
\nonumber
\end{eqnarray}
Analogously we conclude from~\eqref{cardinality_odd} and~\eqref{odd_and_quotients_p_not_2} for $p\not= 2$
\begin{eqnarray}
\ch_j\bigl([\map(H,P)_{\odd}]\bigr)
& = &
\frac{p^{p^{r-j}}-1}{2},
\label{ch_j(odd_p_odd)}
\end{eqnarray}
We conclude from~\eqref{map(H,P)H__andmap(H(H_j,P)} for all primes $p$
\begin{eqnarray}
\ch_j\bigl([\map(H,P)^k]\bigr)
=
p^{kp^{r-j}}.
\label{ch_j_all}
\end{eqnarray}
We conclude from~\eqref{odd_for_p_is_2} for $p = 2$ for $j \ge 1$
\begin{eqnarray}
\ch_j\bigl(\map(H,P)_{\odd}\bigr)
& = &
0. 
\label{ch_j_p_is_2)}
\end{eqnarray}
Now we are ready to prove assertion~\ref{lem:map(Hk,P)/H):ev-odd}. 
We conclude from~\eqref{ch_j(ev_p_odd)} and~\eqref{ch_j(odd_p_odd)} for $p\not= 2$ and $j = 0,1,\ldots, r$
\begin{eqnarray*}
\lefteqn{\ch_j\left(\big([\map(H,P)_{\ev}] - [\map(H,P)_{\odd}]\bigr)^k\right)}
& & 
\\
& = &
\left(\ch_j\bigl([\map(H,P)_{\ev}]\bigr) - \ch_j\bigl(l[\map(H,P)_{\odd}]\bigr)\right)^k
\\
& = &
\left(\frac{p^{p^{r-j}}+1}{2}  - \frac{p^{p^{r-j}}-1}{2}\right)^k
\\
& = & 
1.
\end{eqnarray*}
Hence we obtain  from~\eqref{a_j_from_ch} for $p\not= 2 $  the equation in $A(H)$
\begin{eqnarray*}
\left([\map(H,P)_{\ev}] - [\map(H,P)_{\odd}]\right)^k 
& = & 
[H/H].
\end{eqnarray*}
We conclude from~\eqref{ch_j_all}  and~\eqref{ch_j_p_is_2)} for $p = 2$ and $j \ge 1$
\begin{eqnarray}
\label{chi_j_even-odd_p_is_2_j_ge_1}
\lefteqn{\ch_j\left(\bigl([\map(H,P)_{\ev}] - [\map(H,P)_{\odd}]\bigr)^k\right)}
& & 
\\
& = & 
\ch_j\left(\bigl([\map(H,P)] - 2 \cdot [\map(H,P)_{\odd}]\bigr)^k\right)
\nonumber
\\ 
& = & 
\left(\ch_j\bigl([\map(H,P)]\bigr) - 2 \cdot \ch_j\bigl([\map(H,P)_{\odd}]\bigr)\right)^k 
\nonumber
\\
& = & 
\ch_j\bigl([\map(H,P)^k]\bigr)
\nonumber
\\ 
& = & 
2^{k2^{r-j}}.
\nonumber
\end{eqnarray}
For $p = 2 $ and $j = 0$ we conclude from~\eqref{cardinality_even} and~\eqref{cardinality_odd}
\begin{eqnarray}
\label{chi_j_even-odd_p_is_2_j_is-0}
\lefteqn{\ch_0\left(\bigl([\map(H,P)_{\ev}] - [\map(H,P)_{\odd}]\bigr)^k\right)}
& & 
\\
& = & \left(\bigl|\map(H,P)_{\ev}\bigr| - \bigl|\map(H,P)_{\odd}\bigr|\right)^k
\nonumber
\\
& = & 
0^k
\nonumber
\\ 
& =& 
0.
\nonumber
\end{eqnarray}
Equations~\eqref{a_j_from_ch}, \eqref{chi_j_even-odd_p_is_2_j_ge_1} and~\eqref{chi_j_even-odd_p_is_2_j_is-0} imply for $p = 2$
\begin{multline*}
\bigl([\map(H,P)_{\ev}] - [\map(H,P)_{\odd}]\bigr)^k
\\
 =  
2^k \cdot [H/H]  - 2^{k2^{r-1}-r} \cdot [H] +  \sum_{i=1}^{r-1} \bigl(2^{k2^{r - i}-r +i} - 2^{k2^{r-i-1}-r+i}\bigr) \cdot [H/H_i]. 
\end{multline*}
This finishes the proof of assertion~\ref{lem:map(Hk,P)/H):ev-odd}.

Finally we prove assertion~\ref{lem:map(Hk,P)/H):all}.
We conclude from~\eqref{a_j_from_ch}  and~\eqref{ch_j_all}
\begin{eqnarray}
\quad \quad [\map(H,P)^k] 
& = & 
p^k \cdot [H/H] + \sum_{i=0}^{r-1} \frac{p^{kp^{r - i}} - p^{kp^{r-i-1}}}{p^{r-i}} \cdot [H/H_i]
\label{[map(H,P)k_in_A(H)}
\\
& = & 
p^k \cdot [H/H] + \sum_{i=0}^{r-1} \bigl(p^{kp^{r - i}-r+i} - p^{kp^{r-i-1}-r +i}\bigr)\cdot [H/H_i].
\nonumber
\end{eqnarray}
The last  equation implies
\[
\bigl|\map(H,P)^k/H\bigr| = p^k + \sum_{i=0}^{r-1} \left(p^{kp^{r - i}-r+i} - p^{kp^{r-i-1}-r+i}\right).
\]
This finishes the proof of Lemma~\ref{lem:map(Hk,P)/H)_parity}.
\end{proof}

\begin{theorem}[Prime power case]\label{the:prime_power_case}
Suppose that $m = |G|$ is a prime power $p^r$ for $r \ge 1$. Let $G_j \subseteq G$ be the subgroup of order
$p^j$ for $j = 0,1,2 \ldots,  r$.

Then there is natural number $k$ uniquely determined by $n = p^{r-1} \cdot (p-1) \cdot k$. 
We obtain  for  $ j = 1,2, \ldots,  r$ in $A(G)$
\[
\bigl[H^1(G_j;L)\bigr] = 
p^k \cdot [G/G] + \sum_{i = 0}^{r-j-1} \bigl(p^{kp^{r-j-i}-r +j+i} - p^{kp^{r-j-1-i}-r+j+i}\bigr) \cdot [G/G_{i+j}],
\]
and 
\[
\bigl|\calm(G_j)\bigr| 
=  \begin{cases}
p^{kp^{r-j}-r+j} - p^{kp^{r-j -1}-r+j} & j= 1,2 \ldots, (r-1);
\\
p^k   & j = r.
\end{cases}
\]
\end{theorem}
\begin{proof}
We get from Theorem~\ref{the:calm(C)}~\ref{the:calm(C):|calm(C)|_and_H}, 
\begin{eqnarray}
\bigl|\calc(G_j)\bigr|
 =  
\bigl|H^1(G_j;L)/(G/G_j)\bigr|.
\label{calc(G_j)_and_H(G_j;L)}
\end{eqnarray}
Choose $k$ such that the $\IZ_{(p)}[G]$-module $L_{(p)}$ is isomorphic to 
$(\IZ[\zeta]_{(p)})^k$. Then 
\begin{multline*}
n = \rk_{\IZ}(L) = \rk_{\IZ_{(p)}}(L_{(p)}) = \rk_{\IZ_{(p)}}(\IZ[\zeta]_{(p)}^k) 
= \rk_{\IZ}(\IZ[\zeta]^k)  
\\
= k \cdot \rk_{\IZ}(\IZ[\zeta]) = k \cdot p^{r-1} \cdot (p-1).
\end{multline*}
We obtain from~\cite[Theorem~10.3 in~III.10. on page~84]{Brown(1982)} an isomorphism of $\IZ [G/G_j]$-modules
\[
H^1(G_j;L)
\cong
H^1(G_j; \IZ[\zeta]^k).
\]
Hence we conclude from~\eqref{calc(G_j)_and_H(G_j;L)}
\begin{eqnarray}
\bigl|\calc(G_j)\bigr|
& =&
\bigl|H^1(G_j; \IZ[\zeta]^k)/(G/G_j)\bigr|,
\label{calc_and_h1}
\end{eqnarray}
and we get in $A(G/G_j)$
\begin{eqnarray}
\bigl[H^1(G_j; L)\bigr] 
& = & 
\bigl[H^1(G_j; \IZ[\zeta]^k)\bigr].
\label{eqa_in_A/G(G_j)_L_to_zeta}
\end{eqnarray}
From Lemma~\ref{lem:z[zeta]}~\ref{lem:z[zeta]:ZG/G_1j_iso}
and Lemma~\ref{lem:map(Hk,P)/H)_parity}~\ref{lem:map(Hk,P)/H):all}   applied to $H = G/G_j$ we obtain
the equation in $A(G/G_j)$ for $j = 1,2, \ldots, r$
\begin{eqnarray}
\label{eqa_in_A/G(G_j)_A(G/G_j)}
& & 
\\
\bigl[H^1(G_j;\IZ[\zeta]^k)\bigr]
& = & 
\bigl[H^1(G_j;\IZ[\zeta])^k\bigr] 
\nonumber
\\
& = & 
\bigl[\map(G/G_j,\IZ/p)^k\bigr]
\nonumber
\\
& =  &
p^k \cdot [(G/G_j)/(G/G_j)] 
\nonumber
\\& & \quad  + \sum_{i = 0}^{r-j-1} \bigl(p^{kp^{r-j-i}-r+j+i}-p^{kp^{r-j-i-1}-r+j+i}\bigr) \cdot [(G/G_j)/G_{i+j}/G_j].
\nonumber
\end{eqnarray}

From~\eqref{eqa_in_A/G(G_j)_L_to_zeta} and~\eqref{eqa_in_A/G(G_j)_A(G/G_j)} we obtain the following equality in $A(G)$
for $j = 1,2, \ldots, r$
\[
\bigl[H^1(G_j;L)\bigr] 
= p^k \cdot [G/G] + \sum_{i = 0}^{r-j-1} \bigl(p^{kp^{r-j-i}-r+j-i}-p^{kp^{r-j-i-1}-r+j+i}\bigr) \cdot [G/G_{i+j}].
\]
This together with~\eqref{calc(G_j)_and_H(G_j;L)}  implies for $j = 1,2,\ldots, r$
\begin{eqnarray*}
\bigl|\calc(G_j)\bigr|
& = & 
 p^k  + \sum_{i=0}^{r-j-1} p^{kp^{r-j - i}-r +j+i} - p^{kp^{r-j - i-1}-r+j+i}.
\end{eqnarray*}
Now the claim follows from Example~\ref{exa:p-power-order} by the following calculation  for $j = 1,2, \ldots, (r-1)$
\begin{eqnarray*}
\lefteqn{p^k  + \sum_{i=0}^{r-j-1} \bigl(p^{kp^{r-j - i}-r+j+i} - p^{kp^{r-j - i-1}-r+j+i}\bigr)}
\\ 
& & \quad -p^k  - \sum_{i=0}^{r-j-2} \bigl(p^{kp^{r-j -1 - i}+r +j+1+i} - p^{kp^{r-j -1 - i-1}-r+j+1+i}\bigr)
\\
& = & 
\sum_{i=0}^{r-j-1} \bigl(p^{kp^{r-j - i}-r+j+i} - p^{kp^{r-j - i-1}-r+j+i}\bigr)
\\
& & \quad - \sum_{i=0}^{r-j-2} \bigl(p^{kp^{r-j -1 - i}+r +j+1+i} - p^{kp^{r-j -1 - i-1}-r+j+1+i}\bigr)
\\
& = & 
\sum_{i=0}^{r-j-1} \bigl(p^{kp^{r-j - i}-r+j+i} - p^{kp^{r-j - i-1}-r+j+i}\bigr)
\\
& & \quad - \sum_{i=1}^{r-j-1} \bigl(p^{kp^{r-j - i}+r +j+i} - p^{kp^{r-j - i-1}-r+j+i}\bigr)
\\
& = & 
p^{kp^{r-j}-r+j} - p^{kp^{r-j -1}-r+j}.
\end{eqnarray*}
\end{proof}

\begin{notation}\label{not:lambda_something}
Given a commutative ring $R$ and an $R$-module $M$, denote by $\Lambda^j M = \Lambda_R^j M$
its $j$-th exterior power. Define $R$-modules
\begin{eqnarray*}
\Lambda^{\odd} M & := & \bigoplus_{l \ge 0} \Lambda^{2l + 1} M;
\\
\Lambda^{\ev} M & := & \bigoplus_{l \ge 0} \Lambda^{2l} M;
\\
\Lambda^{\all} M & := &  \bigoplus_{j\ge 0} \Lambda^{j} M.
\end{eqnarray*}
Define for a rational $G$-representation $V$  classes in $R_{\IQ}(G)$
\begin{eqnarray*}
\bigl[\Lambda^{\all} V\bigr] & := &  \sum_{j\ge 0} \bigl[ \Lambda^{j} V\bigr];
\\
\bigl[\Lambda^{\alt} V\bigr] & := &  \sum_{j\ge 0} (-1)^j \cdot \bigl[ \Lambda^{j} V\bigr];
\end{eqnarray*}
\end{notation}

\begin{lemma} \label{lem:Q[zeta]-in_A(G)_p_group_G}
Suppose that $m = p^r$ for a prime $p$ and an integer $r \ge 1$. Let $k$ be a natural number. 
Put $\xi = \exp(2\pi i/m)$. Let $\IQ$ be the trivial $\IQ G$-module of dimension one.

\begin{enumerate}
\item \label{lem:Q[zeta]-in_A(G)_p_group_G:all_p_not_2}
If $p \not= 2$, we get in $R_{\IQ}(G)$
\begin{eqnarray*}
\bigl[\Lambda^{\all} (\IQ[\zeta]^k)\bigr]
& = & 
[\IQ] +  \frac{2^{k \cdot (p-1) \cdot p^{r-1}}-1}{p^r} \cdot [\IQ G];
\end{eqnarray*}
\item \label{lem:Q[zeta]-in_A(G)_p_group_G:all_p_is_2}
Suppose $p = 2$. If $r = 1$, we get in $R_{\IQ}(G)$ 
\[
\bigl[\Lambda^{\all} (\IQ[\zeta]^k)\bigr]
=  2^{k-1} \cdot [\IQ G]. 
\]
If $r \ge 2$, we get in $R_{\IQ}(G)$ 
\begin{multline*}
\bigl[\Lambda^{\all} (\IQ[\zeta]^k)\bigr]
 =
2^k \cdot [\IQ ]   - 2^{k2^{r-2} -r +1} \cdot \bigl[\IQ[G/G_1]\bigr]
\\
+  \sum_{i=2}^{r-1} \bigl(2^{k2^{r - i}-r +i} - 2^{k2^{r-1-i}-r+i}\bigr) \cdot \bigl[\IQ[G/G_{i}]\bigr]  + 2^{k \cdot 2^{r-1}-r} \cdot  [\IQ G];
\end{multline*}
\item \label{lem:Q[zeta]-in_A(G)_p_group_G:alternating}
We get in $R_{\IQ}(G)$
\begin{multline*}
\bigl[\Lambda^{\alt} (\IQ[\zeta]^k)\bigr]
= 
p^k \cdot [\IQ] - p^{kp^{r -1}-r} \cdot \bigl[\IQ G\bigr]
\\
+\sum_{i=1}^{r-1} \bigl(p^{kp^{r -i}-r+i} 
- p^{kp^{r-i-1}-r + i}\bigr)\cdot \bigl[\IQ[G/G_{i}]\bigr]. 
\end{multline*}
\end{enumerate}
\end{lemma}
\begin{proof}~\ref{lem:Q[zeta]-in_A(G)_p_group_G:all_p_not_2} and~\ref{lem:Q[zeta]-in_A(G)_p_group_G:all_p_is_2}
Consider an integer $j \ge 1$. We conclude from~\cite[(1.5) and Lemma~1.6]{Langer-Lueck(2011_homology)} 
(using the notation of this paper)  that there exists
a sequence of $\IQ G$-modules
\[
0 \to \bigoplus_{\sigma \in S_{<p}} \Gamma^j \IQ[G/G_1]_{\sigma} \to F_1 
\to F_2 \to \cdots \to F_j \to \Lambda^j \IQ[\zeta] \to 0
\]
where each $F_l$ is a finitely generated free $\IQ G$-module and $S_{< p}$ is
the subset of $S$ consisting of elements represented by functions $f \colon
\IZ/p^{r-1} \to \{0,1, \ldots , (p-1)\}$ and the $G$-action on $G/G_1$ comes
from the canonical projection $G \to G/G_1$. Hence we get for some $a_j \in\IZ$
in $R_{\IQ}(G)$
\[
\bigl[\Lambda^{j} \IQ[\zeta]\bigr]
= (-1)^j \cdot \left[\bigoplus_{\sigma \in S_{<p}}  \Gamma^j \IQ[G/G_1]_{\sigma}\right]
+ a_j \cdot [\IQ G].
\]
Thus we get for an appropriate integer $a'$ in $R_{\IQ}(G)$
\begin{eqnarray*}
\bigl[\Lambda^{\all} \IQ[\zeta]\bigr]
& = & 
\left[\bigoplus_{l \ge 0} \bigoplus_{\sigma \in S_{<p}}  \Gamma^{2l} \IQ[G/G_1]_{\sigma}\right] 
-
\left[\bigoplus_{l \ge 0} \bigoplus_{\sigma \in S_{<p}}  \Gamma^{2l+1} \IQ[G/G_1]_{\sigma}\right] + a' \cdot [\IQ G].
\end{eqnarray*}

Fix a generator $t \in G$. Denote by $[t^s]$ the class of $t^s$ in $G/G_1$.
A $\IZ$-basis for $\bigoplus_{\sigma \in S_{<p}} \Gamma^j \IQ[G/G_1]_{\sigma}$ is given by the set
\[
\left\{[t^0]^{[e_0]}[t^1]^{[e_1]}\ldots [t^{p^{r-1}-1}]^{[e_{p^{r-1}}]} \,\left|\,
e_0,e_1, \ldots, e_{p^{r-1}} \in P, \sum_{i=0}^{p^{r-1}-1} e_i = j\right.\right\}.
\]
Hence a $\IZ$-basis for 
$\bigoplus_{l \ge 0} \bigoplus_{\sigma \in S_{<p}} \Gamma^{2l} \IQ[G/G_1]_{\sigma}$ is given by the set
\[
\left\{[t^0]^{[e_0]}[t^1]^{[e_1]}\ldots [t^{p^{r-1}-1}]^{[e_{p^{r-1}}]} \,\left|\, 
e_0,e_1, \ldots, e_{p^{r-1}} \in P, \sum_{i=0}^{p^{r-1}-1} e_i \equiv 0 \mod 2\right.\right\}.
\]
We can identify this set with $\map(G/G_1,P)_{\ev}$ in the obvious way. There is
an obvious $G/G_1$-action on $\map(G/G_1,P)_{\ev}$ which yields a $G$-action by
the projection $G \to G/G_1$. Thus we get in $R_{\IQ}(G)$
\begin{eqnarray*}
\left[\bigoplus_{l \ge 0} \bigoplus_{\sigma \in S_{<p}} \Gamma^{2l} \IQ[G/G_1]_{\sigma}\right] 
& =  &
\bigl[\IQ \map(G/G_1,P)_{\ev} \bigr];
\\
\left[\bigoplus_{l \ge 0} \bigoplus_{\sigma \in S_{<p}} \Gamma^{2l+1} \IQ[G/G_1]_{\sigma} \right] 
& = &
\bigl[\IQ \map(G/G_1,P)_{\odd} \bigr].
\end{eqnarray*}
Hence we obtain the following equation in $R_{\IQ}(H)$ for some integer $a'$
\begin{eqnarray}
\quad\quad  \quad\bigl[\Lambda^{\all} \IQ[\zeta]\bigr]
& = & 
\bigl[\IQ \map(G/G_1,P)_{\ev} \bigr] 
- \bigl[\IQ \map(G/G_1,P)_{\odd} \bigr] +  a' \cdot [\IQ G].
\label{Lambda_and_Gamma}
\end{eqnarray}

By the exponential law we get an isomorphism of $\IQ G$-modules
\[
\Lambda^{\all} (\IQ[\zeta]^k) 
\cong \bigotimes_{l = 1} ^k \Lambda^{\all} \IQ[\zeta]
\]
and hence we get in $R_{\IQ}(G)$
\[
\bigl[\Lambda^{\all} (\IQ[\zeta]^k)\bigr]
= \prod_{l= 1}^k \bigl[\Lambda^{\all} \IQ[\zeta]\bigr].
\]

For every $\IQ G$-module $V$ the $\IQ G$-module $V \otimes_{\IQ} \IQ G$  is
a free $\IQ G$-module. Hence we obtain in $R_{\IQ}(G)$ 
for an appropriate integer $a$
\begin{multline}
\label{sum_allLambdaj}
\bigl[\Lambda^{\all} (\IQ[\zeta]^k)\bigr]
\\
 = 
\left(\bigl[\IQ \map(G/G_1,P)_{\ev} \bigr]  - \bigl[\IQ \map(G/G_1,P)_{\odd} \bigr]\right)^k +  a \cdot [\IQ G].
\end{multline}
Next we consider the case $p \not= 2$. Lemma~\ref{lem:map(Hk,P)/H)_parity}~\ref{lem:map(Hk,P)/H):ev-odd}   and~\eqref{sum_allLambdaj} imply
\begin{eqnarray*}
\bigl[\Lambda^{\all} (\IQ[\zeta]^k)\bigr]
& = & 
[\IQ] +  a \cdot [\IQ G] 
\end{eqnarray*}
Since
\begin{eqnarray*}
\sum_{j \ge 0}  \dim_{\IQ}\left(\Lambda^j \IQ[\zeta]^k \right) & = & 2^{\dim_{\IQ}(\IQ[\zeta]^k)} = 2^{k \cdot (p-1) \cdot p^{r-1}},
\\
\dim_{\IQ}\left(\IQ G \right) & = & p^r,
\end{eqnarray*}
we conclude
\[
a = 
\frac{2^{k \cdot (p-1) \cdot p^{r-1}}-1}{p^r}
\]
This finishes the proof of assertion~\ref{lem:Q[zeta]-in_A(G)_p_group_G:all_p_not_2}.

Next we treat assertion~\ref{lem:Q[zeta]-in_A(G)_p_group_G:all_p_is_2}, i.e., the case $p = 2$.
We begin with the case $r = 1$, Then $G/G_1$ is trivial and~\eqref{sum_allLambdaj} implies
\[
\bigl[\Lambda^{\all} (\IQ[\zeta]^k)\bigr] = a \cdot \IQ G.
\] 
Taking rational dimensions yields $2^k = 2a$ and hence $a  = 2^{k-1}$. Hence we get
\begin{eqnarray*}
\bigl[\Lambda^{\all} (\IQ[\zeta]^k)\bigr]
& = & 
2^{k-1} \cdot [\IQ G] \quad \text{if}\; r = 1.
\end{eqnarray*}
If $r \ge 2$, then $G/G_1$ is non-trivial and hence 
Lemma~\ref{lem:map(Hk,P)/H)_parity}~\ref{lem:map(Hk,P)/H):ev-odd}  and~\eqref{sum_allLambdaj} imply for some integer $a$
\begin{eqnarray}
\label{Lambda_all_in_A(G)}
& & 
\\
\bigl[\Lambda^{\all} (\IQ[\zeta]^k)\bigr]
& = & 
2^k \cdot [\IQ ]   - 2^{k2^{r-2} -r +1} \cdot \bigl[\IQ[G/G_1]\bigr]
\nonumber
\\ & & 
\quad +  \sum_{i=1}^{r-2} \bigl(2^{k2^{r -1 - i}-r + 1+i} - 2^{k2^{r-2-i}-r+1 +i}\bigr) \cdot \bigl[\IQ[G/G_{i+1}]\bigr]  + a \cdot  [\IQ G]
\nonumber
\\
& = & 
2^k \cdot [\IQ ]   - 2^{k2^{r-2} -r +1} \cdot \bigl[\IQ[G/G_1]\bigr]
\nonumber
\\ & & 
\quad +  \sum_{i=2}^{r-1} \bigl(2^{k2^{r - i}-r +i} - 2^{k2^{r-1-i}-r+i}\bigr) \cdot \bigl[\IQ[G/G_{i}]\bigr]  + a \cdot  [\IQ G].
\nonumber
\end{eqnarray}
If $ r = 2$, then taking rational dimensions in~\eqref{Lambda_all_in_A(G)} yields
\[
2^{2k} = 2^k - 2^{k-1} \cdot 2 + 2^2 \cdot a
\]
and hence $a  = 2^{2k-2}$.  If $r \ge 3$, taking rational dimension in~\eqref{Lambda_all_in_A(G)} yields
\begin{eqnarray*}
2^{k \cdot 2^{r-1}} 
& = & 
2^k -  2^{k2^{r-2} -r +1} \cdot 2^{r-1} 
+  \sum_{i=2}^{r-1} \bigl(2^{k2^{r - i}-r +i} - 2^{k2^{r-1-i}-r+i}\bigr) \cdot 2^{r-i}  + a \cdot  2^r
\\
& = & 
2^k -  2^{k2^{r-2}} 
+  \sum_{i=2}^{r-1} \bigl(2^{k2^{r - i}} - 2^{k2^{r-1-i}}\bigr)   + a \cdot  2^r
\\
& = & 
2^k -  2^{k2^{r-2}} +  a \cdot  2^r
+  \sum_{i=2}^{r-1} 2^{k2^{r - i}} -   \sum_{i=2}^{r-1}  2^{k2^{r-1-i}}
\\
& = & 
2^k -  2^{k2^{r-2}} +  a \cdot  2^r
+  \sum_{i=2}^{r-1} 2^{k2^{r - i}} -   \sum_{i=3}^{r}  2^{k2^{r-i}}
\\
& = & 
2^k -  2^{k2^{r-2}} +  a \cdot  2^r + 2^{k2^{r - 2}} -  2^{k}
+  \sum_{i=3}^{r-1} 2^{k2^{r - i}} -   \sum_{i=3}^{r-1}  2^{k2^{r-i}}
\\
& = & 
 a \cdot  2^r
\end{eqnarray*}
and hence
\[
a = 2^{k \cdot 2^{r-1}-r} \quad \text{if}\; r \ge 2.
\]
This together with~\eqref{Lambda_all_in_A(G)} implies for $r \ge 2$
\begin{multline*}
\bigl[\Lambda^{\all} (\IQ[\zeta]^k)\bigr]
 =
2^k \cdot [\IQ ]   - 2^{k2^{r-2} -r +1} \cdot \bigl[\IQ[G/G_1]\bigr]
\\
+  \sum_{i=2}^{r-1} \bigl(2^{k2^{r - i}-r +i} - 2^{k2^{r-1-i}-r+i}\bigr) \cdot \bigl[\IQ[G/G_{i}]\bigr]  + 2^{k \cdot 2^{r-1}-r} \cdot  [\IQ G].
\end{multline*}
This finishes the proof of assertion~\ref{lem:Q[zeta]-in_A(G)_p_group_G:all_p_is_2}.
\\[1mm]~\ref{lem:Q[zeta]-in_A(G)_p_group_G:alternating}
One shows analogously to~\eqref{Lambda_and_Gamma} for some integer $b'$
\begin{eqnarray*}
\bigl[\Lambda^{\alt} \IQ[\zeta]\bigr]
& = & 
\bigl[\IQ \map(G/G_1,P)_{\ev} \bigr] 
+ \bigl[\IQ \map(G/G_1,P)_{\odd} \bigr] +  b' \cdot [\IQ G].
\\
& = & 
\bigl[\IQ \map(G/G_1,P) \bigr] +  b' \cdot [\IQ G].
\end{eqnarray*}
Again by the exponential law, we get
\begin{eqnarray*}
\bigl[\Lambda^{\alt} (\IQ[\zeta]^k)\bigr]
& = &
\prod_{l = 1}^k\, \bigl[\Lambda^{\alt} \IQ[\zeta]\bigr].
\end{eqnarray*}
Hence we get for  some appropriate integer $b \in \IZ$ in $R_{\IQ}(G)$.
\begin{eqnarray*}
\bigl[\Lambda^{\alt} (\IQ[\zeta]^k)\bigr]
& = &
\bigl[\IQ \map(G/G_1,P)^k \bigr] +  b \cdot [\IQ G].
\end{eqnarray*}
This implies by Lemma~\ref{lem:map(Hk,P)/H)_parity}~\ref{lem:map(Hk,P)/H):all} 
 applied to $H = G/G_1$
\begin{eqnarray*}
\bigl[\Lambda^{\alt} (\IQ[\zeta]^k)\bigr]
& = &
p^k \cdot [\IQ] + \sum_{i=0}^{r-2} \bigl(p^{kp^{r -i-1}-r+i+1} - p^{kp^{r-i-2}-r + i +1}\bigr)\cdot \bigl[\IQ[G/G_{i+1}]\bigr] +  b \cdot [\IQ G]
\\
& = &
p^k \cdot [\IQ] + \sum_{i=1}^{r-1} \bigl(p^{kp^{r -i}-r+i} - p^{kp^{r-i-1}-r + i }\bigr)\cdot \bigl[\IQ[G/G_{i}]\bigr] +  b \cdot [\IQ G].
\end{eqnarray*}
Since
\begin{eqnarray*}
\sum_{j \ge 0}  (-1)^j \cdot \dim_{\IQ}\left(\Lambda^{j} (\IQ[\zeta]^k\right) & = & 0,
\end{eqnarray*}
we get
\[
b = - p^{k-r} - \sum_{i=1}^{r-1} \big(p^{kp^{r - i}-r} - p^{kp^{r-i-1}-r}\bigr) = -p^{kp^{r -1}-r}.
\]
This finishes the proof of Lemma~\ref{lem:Q[zeta]-in_A(G)_p_group_G}.
\end{proof}

\begin{lemma}
\label{lem:sum_rk_Z(Lambda_j_L_G)}
Suppose that $m = p^r$ for a prime $p$ and an integer $r \ge 1$. Let $k$ be a natural number
uniquely determined by $n = k \cdot (p-1) \cdot p^{r-1}$.
Then
\[
\sum_{l \ge 0} \, \rk_{\IZ}(\Lambda^l L)^G
= 
\begin{cases}
1 +  \frac{2^{n}-1}{p^r} & \text{if}\; p \not= 2;
\\
2^{k-1} & \text{if}\; p = 2, r = 1;
\\
 2^{2k -2} + 2^{k-1}  & \text{if}\; p = 2, r = 2;
\\
2^{k-1}   +  2^{k2^{r-2} -r+1} + 2^{k \cdot 2^{r-1}-r}  +  \sum_{i=3}^{r-1}  2^{k2^{r-i}-r+i-1}
& \text{if}\; p = 2, r \ge 3.
\end{cases}
\]
\end{lemma}
\begin{proof}
Notice that $L \otimes_{\IZ}\IQ$ is isomorphic as $\IQ G$-module to $\IQ[\zeta]^k$ for $\zeta = \exp(2\pi i/m)$. 
This implies
\begin{eqnarray*}
n & = & k \cdot (p-1) \cdot p^{r-1};
\\
\sum_{l \ge 0} \, \rk_{\IZ}(\Lambda^l L)^G
& = & 
\dim_{\IQ}\left(\bigl(\Lambda^{\all} (\IQ[\zeta]^k)\bigr)^G\right).
\end{eqnarray*}
Now the  case $p \not= 2$ follows from Lemma~\ref{lem:Q[zeta]-in_A(G)_p_group_G}~\ref{lem:Q[zeta]-in_A(G)_p_group_G:all_p_not_2}.
Next we treat the case $p = 2$. Then the claim follows for $r =1$ directly and for
$r = 2$ and for $r \ge 3$ by the following calculations from 
Lemma~\ref{lem:Q[zeta]-in_A(G)_p_group_G}~\ref{lem:Q[zeta]-in_A(G)_p_group_G:all_p_is_2}.
Namely,  if $r = 2$, we get
\begin{eqnarray*}
\sum_{l \ge 0} \rk_{\IZ}(\Lambda^l L)^G)
& = & 
2^k    - 2^{k2^{2-2} -2 +1} + 2^{k \cdot 2^{2-1}-2}
\\
& = &
2^k    - 2^{k-1} + 2^{2k -2}
\\
& = &  2^{2k -2} + 2^{k-1}.
\end{eqnarray*} 
Suppose $r \ge 3$. Then 
\begin{eqnarray*}
\sum_{l \ge 0} \rk_{\IZ}(\Lambda^l L)^G)
& = & 
2^k    - 2^{k2^{r-2} -r +1} + 2^{k \cdot 2^{r-1}-r} 
+  \sum_{i=2}^{r-1} \bigl(2^{k2^{r - i}-r +i} - 2^{k2^{r-1-i}-r+i}\bigr)
\\
& = & 
2^k    - 2^{k2^{r-2} -r +1} + 2^{k \cdot 2^{r-1}-r} 
+  \sum_{i=2}^{r-1} 2^{k2^{r - i}-r +i} - \sum_{i=2}^{r-1}  2^{k2^{r-1-i}-r+i}
\\
& = & 
2^k    - 2^{k2^{r-2} -r +1} + 2^{k \cdot 2^{r-1}-r} 
+  \sum_{i=2}^{r-1} 2^{k2^{r - i}-r +i} - \sum_{i=3}^{r}  2^{k2^{r-i}-r+i-1}
\\
& = & 
2^k    - 2^{k2^{r-2} -r +1} + 2^{k \cdot 2^{r-1}-r} + 2^{k2^{r - 2}-r +2} - 2^{k-1}
\\
& & \quad +  \sum_{i=3}^{r-1} 2^{k2^{r - i}-r +i} - \sum_{i=3}^{r-1}  2^{k2^{r-i}-r+i-1}
\\
& = & 
2^{k-1}   +  2^{k2^{r-2} -r+1} + 2^{k \cdot 2^{r-1}-r}  +  \sum_{i=3}^{r-1}  2^{k2^{r-i}-r+i-1}.
\end{eqnarray*}
\end{proof}


\typeout{-------------------------------- Section 7: The general  case  --------------------------}

\section{The general case}
\label{sec:The_general_product_case}

In this section we consider the general case, where $m$ is not necessarily a prime power.

\begin{notation}\label{not:tensor_product_case}
We write
\[
m = |G| = p_1 ^{r_1} \cdot p_2 ^{r_2} \cdot \cdots \cdot p_s ^{r_s},
\]
for distinct primes $p_1, p_2, \ldots , p_s$ and integers $s, r_1, r_2, \ldots ,r_s$, all of which are greater or equal to $1$.
If $C \subseteq G$ is a subgroup, denote by $C[i]$ the $p_i$-Sylow subgroup.

Given  a subgroup $H \subseteq G[j]$,  we denote in the sequel  by $H'$ the subgroup of
$G = G[1] \times \cdots \times G[j-1] \times G[j] \times G[j+1] \times \cdots \times G[s]$
given by 
\[
H' := \{1\} \times \cdots \times \{1\} \times H \times \{1\} \cdots \times \{1\},
\]
and analogously  we define $H' \subseteq C$ for a subgroup $H$ of $C[i]$. (Obviously $H\cong H'$.) 
\end{notation}

\begin{lemma}\label{lem:H(C;L)}
  Let $C \subseteq G$ and $H \subseteq G$ be subgroups. Then  we have 
  \[
    H^1(C; L )^H \cong \{0\}
    \] 
   unless there exists $i \in \{1,2, \ldots , s\}$ such that
    $C[j] = H[j] = \{0\}$ holds for all $j \in \{1,2, \ldots, s\}$ with $j \not= i$.
\end{lemma}
\begin{proof}
Suppose that $H^1(C; L )^H \not= \{0\}$. Since $H^1(C; L )$ is annihilated by multiplication with
$m$ (see~\cite[Corollary~10.2 in~III.10 on page~84]{Brown(1982)}), there exists
$i \in \{1,2 ,\ldots ,s\}$ with
\[
\left(H^1(C; L )^H\right)_{(p_i)} \not = 0.
\]
We conclude for all $j \in \{1,2 \ldots, s\}$ because of $H[j]' \subseteq H$
\[
\left(H^1(C; L )^{H[j]'}\right)_{(p_i)} \not = 0.
\]
Assume in the sequel $j \not= i$. Then $p_i$ and $|H[j]'|$ are prime to one another and 
we conclude using~\cite[Theorem~10.3 in~III.10 on page~84]{Brown(1982)}
\begin{eqnarray*}
\left(H^1(C; L )^{H[j]'}\right)_{(p_i)} 
& \cong &
\left(H^1(C; L )_{(p_i)}\right)^{H[j]'}
\\
& \cong&
\left(H^1(C[i];L )^{C/C[i]}\right)^{H[j]'}.
\end{eqnarray*}
In particular 
\[
\left(H^1(C[i]; L )^{C/C[i]}\right)^{H[j]'} \not= \{0\}.
\]
This implies $H^1(C[i]; L )^{C[j]'} \not= 0$ and $H^1(C[i]; L)^{H[j]'}  \not= 0$. 
Since $|C[j]'|$ and $|C[i]|$ are prime to one another, we get
\[
H^1(C[i];L )^{C[j]'}  \cong H^1\bigl(C[i];L^{C[j]'}\bigr).
\]
This implies $H^1\bigl(C[i];L^{C[j]'}\bigr) \not= \{0\}$ and hence $L^{C[j]'} \not= \{0\}$.
Since the $G$-action on $L$ is free outside the origin, we conclude  $C[j] = 0$. 

Since $|H[j]|$ and $|C[i]|$ are prime to one another, we have
\[
H^1(C[i];L )^{H[j]'} \cong
H^1\bigl(C[i];L^{H[j]'}\bigr). 
\]
This implies $H^1\bigl(C[i];L^{H[j]'}\bigr) \not= \{0\}$ and hence $L^{H[j]'} \not= 0$.
Since the $G$-action on $L$ is free outside the origin, we conclude  $H[j] = 0$. 
\end{proof}

Let $C \subseteq G$ be a subgroup.  Let 
\[\bigl[H^1(C[i];L)\bigr] \in A(G[i])
\] 
be the class of the finite $G[i]$-set $H^1(C[i];L)$, and let
\[
\bigl[H^1(C;L)\bigr] \in A(G)
\]
 be the class of the finite $G$-set $H^1(C; L)$,
where the actions come from the $G$-action on $L$.
Denote by
\[
\ind_{G[i]}^G \colon A(G[i]) \to A(G)
\] 
 the homomorphism of abelian groups coming from induction with the inclusion of groups $G[i] \to G$.

\begin{lemma} \label{lem:H(C;L)_in_A(G)_prep}
Suppose that there exists $i \in \{1,2, \ldots ,s\}$
such that $C[j] = \{0\}$ holds for  all $j \in \{1,2, \ldots ,s\}$ with $j \not= i$.
Then we get in $A(G)$
\[
\ind_{G[i]}^G\left(\bigl[H^1(C[i];L)\bigr] - \bigl[G[i]/G[i]\bigr]\right) 
=  \frac{m}{p_i^{r_i}} \cdot \left(\bigl[H^1(C;L)\bigr] -[G/G]\right).
\]
\end{lemma}
\begin{proof}
We have the isomorphism of $G$-sets
\begin{multline*}
\ind_{G[i]}^G\left(H^1(C[i];L)\right)
\\
\cong  
G[1] \times \cdots \times G[i-1] \times H^1(C[i];L) \times G[i+1] \times \cdots \times G[r].
\end{multline*}
Consider a subgroup $H \subseteq G$. If $H[j] = \{0\}$ for each  $j \in \{1, \ldots ,s\}$ with $j \not= i$, then 
we get for every finite $G[i]$-set $S$
\begin{eqnarray*}
\ch^G_H\bigl(\ind_{G[i]}^G([S])\bigr)
& = &
\left|(\ind_{G[i]}^GS)^H\right| 
\\
& =  &
\left|G[1] \times \cdots \times G[i-1] \times S^{H[i]} \times G[i+1] \times \cdots \times G[r]\right|
\\
 & = & 
|G[1]| \cdot \cdots \cdot |G[i-1]| \cdot \bigl|S^{H[i]}\bigr| \cdot |G[i+1]| \cdot \cdots \cdot |G[r]|
\\
& = &
\frac{m}{p_i^{r_i}} \cdot \bigl|S^{H[i]}\bigr|,
\end{eqnarray*}
and hence because of $C = C[i]$ and $H = H[i]$
\begin{eqnarray}
\label{ch_case_A}
\lefteqn{\ch^G_H\left(\ind_{G[i]}^G\bigl(\bigl[H^1(C[i];L)\bigr]-\bigl[G[i]/G[i]\bigr]\bigr)\right)}
& &
\\
&  = &
\frac{m}{p_i^{r_i}} \cdot \left(\left|H^1(C[i];L)^{H[i]}\right|- \left|G[i]/G[i]^{H[i]}\right|\right)
\nonumber
\\
& = & 
\frac{m}{p_i^{r_i}} \cdot \left(\left|H^1(C[i];L)^{H[i]}\right|- 1\right)
\nonumber
\\
& = & \frac{m}{p_i^{r_i}} \cdot \left(\left|H^1(C;L)^H\right|- \left|G/G^H\right|\right)
\nonumber
\\
& = & 
\frac{m}{p_i^{r_i}} \cdot\ch^G_H \left(\bigl[H^1(C;L)\bigr] - \bigl[G/G\bigr] \right).
\nonumber
\end{eqnarray}

Suppose that there exists $j \in \{1, \ldots ,s\}$ with $j \not= i$ and $H[j] \not= \{0\}$.
Then we get for every finite $G[i]$-set $S$
\begin{eqnarray*}
\lefteqn{\ch^G_H\bigl(\ind_{G[i]}^G([S])\bigr)}
& & 
\\
& = &
\left|(\ind_{G[i]}^G S)^H\right| 
\\
& =  &
\left|G[1]^{H[1]} \times \cdots \times G[i-1]^{H[i-1]} \times S^{H[i]} \times G[i+1]^{H[i+1]} \times \cdots \times G[r]^{H[r]}\right|
\\
 & = & 
|\emptyset|
\\
& = &
0,
\end{eqnarray*}
and we compute using Lemma~\ref{lem:H(C;L)}
\begin{eqnarray}
\label{ch_case_B}
\lefteqn{\ch^G_H\left(\ind_{G[i]}^G\bigl(\bigl[H^1(C[i];L)\bigr]-\bigl[G[i]/G[i]\bigr]\bigr)\right)}
& &
\\
&  = &
0
\nonumber
\\
& = & 
\frac{m}{p_i^{r_i}} \cdot \left(\bigl|H^1(C;L)^H\bigr| - 1\right).
\nonumber
\\
& = & 
\frac{m}{p_i^{r_i}} \cdot \ch^G_H\bigl(\bigl[H^1(C;L)\bigr]- \bigl[G/G\bigr]\bigl).
\nonumber
\end{eqnarray}
Since the character map 
\[\ch^G \colon A(G) \to \prod_{H \subseteq G} \IZ, \quad [S] \mapsto |S^H|
\]
is injective, Lemma~\ref{lem:H(C;L)_in_A(G)_prep}  follows from~\eqref{ch_case_A} and~\eqref{ch_case_B}.
\end{proof}

\begin{theorem}[General case]\label{the:H(C;L)_in_A(G)}\
Then we get using Notation~\ref{not:tensor_product_case}:
\begin{enumerate}

\item \label{the:H(C;L)_in_A(G):k}
There exists an natural number $k$ uniquely determined by the property
\[
n = k \cdot \prod_{i=1}^s (p_i-1) \cdot p_i^{r_i-1};
\]

\item \label{the:H(C;L)_in_A(G):p-power} Suppose that there exists 
  $i \in \{1,2,  \ldots ,s\}$ such that $C[j] = \{0\}$ holds for all $j \in \{1,2, \ldots ,s\}$
  with $j \not= i$. Define the natural numbers $r_i$ and $c_i$ by 
 $|G[i]| =   p_i^{r_i}$ and $|C[i]| = p_i^{c_i}$. Let $k[i]$ be the integer uniquely
  determined by $n = (p_i -1) \cdot p_i^{r_i-1} \cdot k[i]$.  (Its existence
  follows from assertion~\ref{the:H(C;L)_in_A(G):k}.) Let $G[i]_l$ be the
  subgroup of $G[i]$ of order $p_i^l$.

Then we get in $A(G)$
\begin{multline*}
\bigl[H^1(C;L)\bigr] =  
[G/G]  +  \frac{\bigl(p_i^{k[i]} -1\bigr) \cdot p_i^{r_i}}{m} \cdot \bigl[G/G[i]'\bigr] 
\\ 
+  \sum_{l = 0}^{r_i-c_i-1} \frac{p_i^{k[i]p_i^{r_i-c_i-l}+c_i+l} - p_i^{k[i]p_i^{r_i-c_i-l-1}+c_i+l}}{m} \cdot \bigl[G/G[i]_{c_i + l}'\bigr];
\end{multline*}

\item \label{the:H(C;L)_in_A(G):not_p-power}
Suppose that there exists no $i \in \{1,2, \ldots ,s\}$
such that $C[j] = \{0\}$ holds for  all $j \in \{1,2, \ldots ,s\}$ with $j \not= i$. Then
we get in $A(G)$.
\[
\bigl[H^1(C;L)\bigr] = [G/G].
\]
\item \label{the:H(C;L)_in_A(G):maximal_subgroups}\
\begin{enumerate}

\item \label{the:H(C;L)_in_A(G):maximal_subgroups:G}
If $C = G$, we have
\[
|\calm(C)| = |\calc(C)| =\begin{cases}
p^k & s = 1;
\\
1 & s \ge 2;
\end{cases}
\]

\item \label{the:H(C;L)_in_A(G):maximal_subgroups:non-p-power}
If $C$ is different from $G$ and is not a  $p$-group, then 
\begin{eqnarray*}
|\calc(C)| & = & 1;
\\
|\calm(C)| & = & 0;
\end{eqnarray*}

\item \label{the:H(C;L)_in_A(G):maximal_subgroups:p-power_and_not_G}
If $C$ is different from $G$ and $|C| = p_i^j$ for some $i \in \{1,2, \ldots, s\}$ and 
$j \in \{1,2, \ldots,  r_i\}$, then
\[
|\calm(C)|
= \begin{cases}
\frac{(p_i^{k[i]}-1) \cdot p_i^{r_i}}{m} 
& 
j = r_i;
\\
\frac{p_i^{k[i]p_i^{r_i-j}+j} - p_i^{k[i]p_i^{r_i-j-1}+j}}{m}
&
j \le r_i -1.
\end{cases}
\]
\end{enumerate}

\end{enumerate}
\end{theorem}
\begin{proof}~\ref{the:H(C;L)_in_A(G):k} The rational $G$-representation $L
  \otimes_{\IZ} \IQ$ has by assumption the property that $(L \otimes_{\IZ}\IQ)^C  = \{0\}$ 
   for all subgroups $C \subseteq G$ with $C \not= \{0\}$. Hence the
  exists a natural number $k$ such that $L \otimes_{\IZ}\IQ$ is $\IQ
  G$-isomorphic to $\IQ[\exp(2\pi i/m)]^k$ 
  (see~\cite[Exercise~13.1 on  page~104]{Serre(1977)}).  Since $n = \dim_{\IQ}(L \otimes_{\IZ}\IQ)$ and
  $\dim_{\IQ}\bigl(\IQ[\exp(2\pi i/m)]\bigr) = \prod_{i=1}^s (p_i-1) \cdot
  p_i^{r_i-1}$, assertion~\ref{the:H(C;L)_in_A(G):k} follows.
  \\[1mm]~\ref{the:H(C;L)_in_A(G):p-power} From
  Lemma~\ref{lem:H(C;L)_in_A(G)_prep} we obtain the following equality in $A(G)$
\[
\bigl[H^1(C;L)\bigr] =  
[G/G]  -  \frac{p_i^{r_i}}{m} \cdot  \ind_{G[i]}^G\left(\bigl[G[i]/G[i]\bigr]\right)
+  \frac{p_i^{r_i}}{m} \cdot \ind_{G[i]}^G\left(\bigl[H^1(C[i];L)\bigr]\right).
\]
We get from Theorem~\ref{the:prime_power_case} applied to $G[i]$ the following equation in $A(G[i])$
\[
\bigl[H^1(C[i];L)\bigr] = 
p^{k[i]} \cdot \bigl[G[i]/G[i]\bigr] + 
\sum_{l = 0}^{r_i-c_i-1} \frac{p_i^{k[i]p_i^{r_i-c_i-l}} - p_i^{k[i]p_i^{r_i-c_i-1-l}}}{p_i^{r_i -c_i-l}} \cdot \bigl[G[i]/G_{c_j + l}\bigr].
\]
Since $\ind_{G[i]}^G\left(\bigl[G[i]/G_{c_j + l}\bigr]\right) 
= G/G[i]_{c_i + l}'$, assertion~\ref{the:H(C;L)_in_A(G):p-power} follows.
\\[1mm]~\ref{the:H(C;L)_in_A(G):not_p-power}
By Lemma~\ref{lem:H(C;L)} we have $H^1(C;L) =  \{0\}$ and hence $\bigl[H^1(C;L)\bigr] = [G/G]$.
\\[1mm]~\ref{the:H(C;L)_in_A(G):maximal_subgroups}
Claim~\ref{the:H(C;L)_in_A(G):maximal_subgroups:G} follows from 
Theorem~\ref{the:prime_power_case} if $G$ is a $p$-group.
If $G$ is not a $p$-group, claim~\ref{the:H(C;L)_in_A(G):maximal_subgroups:G} follows 
from Theorem~\ref{the:calm(C)}~\ref{the:calm(C):|calm(C)|_and_H:|calm(C)_and_homology} since 
$|H^1(G;L)| = 1$ holds by  assertion~\ref{the:H(C;L)_in_A(G):not_p-power}.
The claim~\ref{the:H(C;L)_in_A(G):maximal_subgroups:non-p-power} from
Theorem~\ref{the:calm(C)}~\ref{the:calm(C):calm(C)} and~\ref{the:calm(C):|calm(C)|_and_H}
since $|H^1(C;L)| = 1$ holds by  assertion~\ref{the:H(C;L)_in_A(G):not_p-power}. 

It remains to prove
claim~\ref{the:H(C;L)_in_A(G):maximal_subgroups:p-power_and_not_G}.
We begin with the case $j = r_i$. Then we get from Theorem~\ref{the:calm(C)}~\ref{the:calm(C):|calm(C)|_and_H}
and assertions~\ref{the:H(C;L)_in_A(G):p-power},~\ref{the:H(C;L)_in_A(G):maximal_subgroups:G}
and~\ref{the:H(C;L)_in_A(G):maximal_subgroups:non-p-power}
\[
\begin{array}{lclcl}
|\calc(C)| & = & 1 + \frac{(p_i^{k[i]} -1) \cdot p_i^{r_i}}{m} & & 
\\
|\calc(D)| & = & 1 & & \text{for}\; C \subsetneq D \subseteq G.
\end{array}
\]
Now apply Theorem~\ref{the:calm(C)}~\ref{the:calm(C):calm(C)}.  Finally we
consider the case $j \le r_i-1$. Let $D$ be an subgroup of $G$ with $C
\subsetneq D$.  If $D$ is not a $p$-group, we get $|\calc(D)| = 1$ from
Theorem~\ref{the:calm(C)}~\ref{the:calm(C):|calm(C)|_and_H} and
assertion~\ref{the:calm(C):|calm(C)|_and_H}. Hence the image of 
 $\calc(D) \to \calc(C)$ is contained in the image of $\calc(G) \to \calc(C)$. Let 
$C_l \subseteq G$ be the subgroup of order $p_i^{l}$ for $l = 0,1,2, \ldots, r_i$.
Then for every subgroup $D \subseteq G$ with $C \subsetneq D$ the image of
$\calc(D) \to \calc(C)$ is contained in the image of $\calc(C_{j+1}) \to
\calc(C)$. We conclude from Theorem~\ref{the:calm(C)}~\ref{the:calm(C):calm(C)}
and~\ref{the:calm(C):|calm(C)|_and_H} and
assertion~\ref{the:H(C;L)_in_A(G):p-power}
\begin{eqnarray*}
|\calm(C)| 
& = &
|\calc(C)| - |\calc(C_{j+1})|
\\
& = &
1 +  \frac{p_i^{k[i]} -1\bigr) \cdot p_i^{r_i}}{m} +  
\sum_{l = 0}^{r_i-j-1} \frac{p_i^{k[i]p_i^{r_i-j-l}+j+l} - p_i^{k[i]p_i^{r_i-j-l-1}+j+l}}{m}
\\
& & 
\quad  -1 - \frac{(p_i^{k[i]} -1\bigr) \cdot p_i^{r_i}}{m}  
- \sum_{l = 0}^{r_i-j-2} \frac{(p_i^{k[i]p_i^{r_i-j-1-l}+j+1+l} - p_i^{k[i]p_i^{r_i-j-1-l-1}+j+1+l}}{m}
\\
& = &
\sum_{l = 0}^{r_i-j-1} \frac{p_i^{k[i]p_i^{r_i-j-l}+j+l} - p_i^{k[i]p_i^{r_i-j-l-1}+j+l}}{m}
\\
& & 
\quad 
- \sum_{l = 1}^{r_i-j-1} \frac{p_i^{k[i]p_i^{r_i-j-l}+j+l} - p_i^{k[i]p_i^{r_i-j-l-1}+j+l}}{m}
\\
& = & 
\frac{p_i^{k[i]p_i^{r_i-j}+j} - p_i^{k[i]p_i^{r_i-j-1}+j}}{m}.
\end{eqnarray*}
\end{proof}

\begin{example}[$m$ square-free]\label{exa:M_square-free}
Suppose that $m$ is square-free, or, equivalently, $r_1 = r_2 = \cdots  = r_s = 1$.
Let $k$ be the natural number uniquely determined by
$n = k \cdot \prod_{j=1}^s (p_i -1)$. 
If $s = 1$, then every non-trivial finite subgroup of $\Gamma$ is cyclic of order $p$ and maximal among finite subgroups
and we conclude from Theorem~\ref{the:prime_power_case}
\[
|\calm| = p^k.
\]
If $s \ge 2$, then
\[
|\calm(C)| = 
\begin{cases}
1 
& 
C = G;
\\
0 & C \not = G \; \text{and}\; C \not= G[i] \;  \text{for every } \; i \in \{1,2, \ldots ,s\};
\\
\frac{p_i \cdot (p_i^{n/(p_i - 1)} -1)}{m}
& 
C = G[i] \;\text{for some} \; i \in \{1,2, \ldots ,s\}.
\end{cases}
\]
\end{example}

\begin{remark} It follows from the proof that all fractions appearing Theorem~\ref{the:H(C;L)_in_A(G)}
are indeed natural numbers. As an illustration we check this in Example~\ref{exa:M_square-free}.
We have to show that for every $j \in \{1,2, \ldots ,s\}$ with  $j \not= i$ the prime number $p_j$
divides $p_i^{k \cdot\prod_{j \in \{1,2, \ldots ,s\}, j \not= i} (p_j - 1)} -1$. Now the claim follows from
Fermat's little theorem, i.e., from  $p_i^{p_j - 1} \equiv 1 \mod p_j$.
\end{remark}

\begin{lemma} \label{lem:m_even_and_Lambda_odd}
Suppose that $m$ is even. Then
\[
\left(\Lambda^{\odd}\IQ[\zeta_m]\right)^G = 0.
\]
\end{lemma}
\begin{proof}
Since $G$ is even, it contains an element $g$ of order $2$. This elements acts by $-\id$ in $\IQ[\zeta_m]$.
Hence it acts by $-\id$ on $\Lambda^{2l+1}\IQ[\zeta_m]$ for every $l \ge 0$. This implies $\Lambda^{2l+1}\IQ[\zeta_m]^{\langle g \rangle} = \{0\}$ 
for every $l \ge 0$ and hence $\left(\Lambda^{\odd}\IQ[\zeta_m]\right)^G = 0$.
\end{proof}

\begin{remark}[Computing $s_i$ for even $m$]
\label{rem:computing_s_i_for_even_m}
Lemma~\ref{lem:m_even_and_Lambda_odd} implies for even $m$ that 
\begin{eqnarray*}
\sum_{l \in \IZ} \rk_{\IZ}\bigl((\Lambda^{2l+1} \IZ^n)^{\IZ/m}\bigr) & = & 0;
\\
\sum_{l \in \IZ} \rk_{\IZ}\bigl((\Lambda^{2l} \IZ^n)^{\IZ/m}\bigr)  & = & 
\sum_{l \in \IZ} (-1)^l \cdot \rk_{\IZ}\bigl((\Lambda^{l} \IZ^n)^{\IZ/m}\bigr).
\end{eqnarray*}
The alternating sum $\sum_{l \in \IZ} (-1)^l \cdot \rk_{\IZ}\bigl((\Lambda^{l} \IZ^n)^{\IZ/m}\bigr)$ will be computed
in terms of $\calm$ in Theorem~\ref{the:a_priori_estimates}~\ref{the:a_priori_estimates:G/backslashunderlineEG}
(see~Remark~\ref{rem:expressing_sums_involving_call(M)}).
Hence we will be able to compute the numbers
$s_i$ appearing in  Theorem~\ref{the:Topological_K-theory_for_Zr_rtimes_Z/n}~\ref{the:Topological_K-theory_for_Zr_rtimes_Z/n:explicite}
explicitly using Theorem~\ref{the:H(C;L)_in_A(G)}~\ref{the:H(C;L)_in_A(G):maximal_subgroups}
provided that $m$ is even.
\end{remark}


\typeout{-------------------------- Section 8: Equivariant Euler characteristics  --------------------}

\section{Equivariant Euler characteristics}
\label{sec:Equivariant_Euler_characteristics}

Recall that we are considering  the extension $1 \to \IZ^n \to \Gamma \xrightarrow{\pi} \IZ/m\to 1$ as it appears in
Theorem~\ref{the:Topological_K-theory_for_Zr_rtimes_Z/n}
and we will abbreviate $L = \IZ^n$ and $G = \IZ/m$.

Given a  finite $G$-$CW$-complex, define its \emph{G-Euler characteristic}
\begin{eqnarray}
\chi^G(X) & := &\sum_{c} (-1)^{d(c)} \cdot [G/G_c] \quad \in A(G)
\label{chiG(X)}
\end{eqnarray}
where $c$ runs over the equivariant cells $G/G_c \times D^{d(c)}$.

\begin{theorem}[Equivariant Euler characteristic of $L\backslash \underline{E}\Gamma$]\
\label{the:chiG(LbackslashunderlineEGamma}
\item \label{the:chiG(LbackslashunderlineEGamma:chiG_in_A(G)}
We get in the Burnside ring  $A(G)$
\[
\chi^{G}(L\backslash \underline{E}\Gamma)
=  a \cdot [G] + \sum_{(M) \in \calm} [G/\pi(M)],
\]
where the integer $a$  is given by $- \sum_{(M) \in \calm} \frac{1}{|M|}$.
\end{theorem}
\begin{proof}
Since the $G$-action on $L$ is free outside the origin,
we obtain  from Lemma~\ref{lem:maximal_maximal_finite_subgroups}
and~\cite[Corollary~2.11]{Lueck-Weiermann(2007)} a cellular
$\Gamma$-pushout
\[
\xycomsquareminus{\coprod_{(M) \in \calm} \Gamma \times _M EM}{i_0}{E\Gamma}
  {\coprod_{(M) \in \calm} \pr_M }{f}
  {\coprod_{(M) \in \calm} \Gamma/M}{i_1}{X}
\]
such that $X$ is a model for $\eub\Gamma$.
There exists a finite $\Gamma$-$CW$-model $Y$ for $\eub\Gamma$
by Lemma~\ref{lem:extension_splits}.
Choose a $\Gamma $-homotopy equivalence
$f \colon Y \to X$. Let $Z$ be defined by the $\Gamma$-pushout
\[
\xycomsquareminus{Y^{>\{1\}}}{}{Y}{f^{>\{1\}}}{}{X^{>\{1\}}}{}{Z}
\]
Then $Z$ is a finite $\Gamma$-$CW$-model for $\underline{E}\Gamma$ with
$Z^{>\{1\}} = X^{>\{1\}} = \coprod_{(M) \in \calm} \Gamma/M$. In the sequel we write
$\underline{E}\Gamma$ for $Z$. Then $L\backslash \underline{E}\Gamma$
is a finite $G$-$CW$-complex with 
\[
\bigl(L\backslash \underline{E}\Gamma\bigr)^{>\{1\}} = \coprod_{(M) \in \calm} G/\pi(M).
\]
Hence we get for some integer $a$ in $A(G)$
\[
\chi^G\bigl(L\backslash \underline{E}\Gamma\bigr) 
= 
a \cdot [G] + \sum_{(M) \in \calm} [G/\pi(M)].
\]
Let $\ch^G_{\{1\}} \colon A(G) \to \IZ$ be the map sending the class of a finite $G$-set to its cardinality.
It sends $\chi^G\bigl(L\backslash \underline{E}\Gamma\bigr)$ to the (non-equivariant)
Euler characteristic of $L\backslash \underline{E}\Gamma$ which is zero since 
$L\backslash \underline{E}\Gamma$ is homotopy equivalent to the $n$-torus.
Hence we get
\begin{eqnarray*}
0 
& = & 
\ch^G_{\{1\}}\left(\chi^G\bigl(L\backslash \underline{E}\Gamma\bigr) \right)
\\
& = & 
\ch^G_{\{1\}}\left(a \cdot [G] + \sum_{(M) \in \calm} [G/\pi(M)]\right)
\\
& = & 
a \cdot |G| + \sum_{(M) \in \calm} |G/\pi(M)|
\\
& = & 
|G| \cdot \left(a + \sum_{m \in \calm} \frac{1}{|M|}\right).
\end{eqnarray*}
This implies $a =  - \sum_{(M) \in \calm} \frac{1}{|M|}$.
\end{proof}

\begin{theorem}[Rational permutation modules]
\label{the:rational_permutation_modules}\
\begin{enumerate}
\item \label{the:rational_permutation_modules:iso}
Sending a finite $G$-set to the associated $\IQ G$-permutation module defines an isomorphism of rings
\[
\perm \colon A(G) \to R_{\IQ}(G);
\]
\item \label{the:rational_permutation_modules:iso:chi_and_h_general}
If $X$ is a finite $G$-$CW$-complex, then
\[
\perm\bigl(\chi^G(X)\bigr)
=  \sum_{i \ge 0} (-1)^i \cdot [H_i(X;\IQ)] 
= [K_0(X) \otimes_{\IZ} \IQ] - [K_1(X) \otimes_{\IZ} \IQ].
\]
We have
\[
\perm\bigl(\chi^G(L\backslash \eub\Gamma)\bigr)
= \sum_{i \ge 0} (-1)^i \cdot [\Lambda^i L \otimes_{\IZ} \IQ].
\]
\end{enumerate}
\end{theorem}
\begin{proof}~\ref{the:rational_permutation_modules:iso}
The map $\perm$ is surjective  by~\cite[Exercise~13.1~(c) on page~105]{Serre(1977)}.
Since the source and target of $\perm$ are finitely generated free abelian groups of the same rank,
$\perm$ is bijective.
\\[1mm]~\ref{the:rational_permutation_modules:iso:chi_and_h_general}
We compute in $R_{\IQ}(G)$, where $C_*(X)$ is the cellular $\IZ$-chain complex of $X$
\begin{multline*}
\perm\bigl(\chi^G(X)\bigr)  
= \sum_{i \ge 0} (-1)^i \cdot [C_i(X)\otimes_{\IZ} \IQ] 
\\
= \sum_{i \ge 0} (-1)^i \cdot \bigl[H_i(C_*(X)\otimes_{\IZ} \IQ)\bigr] 
= \sum_{i \ge 0} (-1)^i \cdot [H_i(X;\IQ)].
\end{multline*}
The homological Chern character~\cite{Dold(1962)} yields an isomorphism of $\IQ G$-modules
\begin{eqnarray}
\bigoplus_{j \in \IZ} H_{i + 2j}(X;\IQ) & \xrightarrow{\cong} & K_i(X) \otimes_{\IZ} \IQ
\label{homological_Chern-Character}
\end{eqnarray}

The cup product induces natural isomorphisms
\[
\Lambda_{\IQ}^i H^1(L\backslash \underline{E}\Gamma,\IQ)
\xrightarrow{\cong} H^i(L\backslash \underline{E}\Gamma,\IQ),
\]
since $L\backslash \underline{E}\Gamma$ is homotopy equivalence to the $n$-torus.
There are  natural isomorphisms $L\otimes_{\IZ} \IQ\xrightarrow{\cong} H_1(L\backslash \underline{E}\Gamma;\IQ)$,
$\Lambda_{\IQ}^i \bigl((L \otimes_{\IZ} \IQ)^*\bigr)) \xrightarrow{\cong} (\Lambda^i L \otimes_{\IZ} \IQ)^*$
and $H_i(L\backslash \underline{E}\Gamma;\IQ)^* \xrightarrow{\cong}  H^i(L\backslash \underline{E}\Gamma;\IQ)$.
For every finitely generated $\IQ G$-module $M$ there is a canonical $\IQ G$-isomorphism $M \xrightarrow{\cong} (M^*)^*$.
This implies that the $\IQ G$-modules $H_i(L\backslash \underline{E}\Gamma;\IQ)$
and $\Lambda^i L \otimes_{\IZ} \IQ$ are isomorphic. 
\end{proof}

\begin{lemma} \label{lem:quot} Let $X$ be a finite $G$-$CW$-complex. The map of
  abelian groups
  \[
  \quot \colon A(G) \to \IZ, \quad [S] \mapsto |G\backslash S|
  \]
  sends the $G$-Euler characteristic
  $\chi^G(X)$ to the Euler characteristic $\chi(G\backslash X)$ of the 
  quotient. We have 
  \[
   \chi(G\backslash X) = \dim_{\IQ}\bigl(K_0(G\backslash X)
  \otimes_{\IZ} \IQ\bigr) - \dim_{\IQ}\bigl(K_1(G\backslash X) \otimes_{\IZ}
  \IQ\bigr).\]
\end{lemma}
\begin{proof} We get $\quot(\chi^G(X)) = \chi(G\backslash X)$ by inspecting the
  definitions of the Euler characteristics in terms of counting (equivariant)
  cells.  The second claim follows from the homological Chern
  character~\eqref{homological_Chern-Character}.
\end{proof}

\begin{lemma} \label{lem:kg} Let $X$ be a finite $G$-$CW$-complex. The map of
  abelian groups
  \[
  \kg \colon A(G) \to \IZ, \quad [G/H] \mapsto |H|
  \]
  satisfies 
  \[
   \kg\bigl(\chi^G(X)\bigr) = \dim_{\IQ}\bigl(K_0^G(X)
  \otimes_{\IZ} \IQ\bigr) - \dim_{\IQ}\bigl(K_1^G(X) \otimes_{\IZ} \IQ\bigr).
  \]
\end{lemma}
\begin{proof}
  The expression
  \[
  \chi_{K_*^G}(X) := \dim_{\IQ}\bigl(K_0^G(X) \otimes_{\IZ}
  \IQ\bigr) - \dim_{\IQ}\bigl(K_1^G(X) \otimes_{\IZ} \IQ\bigr)
 \] 
  depends only on
  the $G$-homotopy type of $X$ and is additive under $G$-pushouts of finite
  $G$-$CW$-complexes.  The latter follows from the Mayer-Vietoris sequence associated
  to such a $G$-pushout and Bott periodicity.  Hence $\chi_{K_*^G}(X)$ can be
  computed by counting equivariant cells, where the contribution of an
  equivariant cell of the type $G/H \times D^d$ is
\begin{eqnarray*}
\chi_{K_*^G}(G/H \times D^d) 
& = & 
(-1)^d \cdot \chi_{K_*^G}(G/H)
\\
& = & 
(-1)^d \cdot \bigl(\dim_{\IQ}(K_0^G(G/H)\otimes_{\IZ} \IQ) - \dim_{\IQ}(K_1^G(G/H)\otimes_{\IZ} \IQ)\bigr)
\\
& = & 
(-1)^d \cdot \bigl(\dim_{\IQ}(K_0^H(\pt)\otimes_{\IZ} \IQ) - \dim_{\IQ}(K_1^H(\pt)\otimes_{\IZ} \IQ)\bigr)
 \\
& = & 
(-1)^d \cdot \bigl(\dim_{\IQ}(R_{\IC}(H)\otimes_{\IZ} \IQ) - \dim_{\IQ}(\{0\}\otimes_{\IZ} \IQ)\bigr)
 \\
& = & 
(-1)^d \cdot |H|.
\end{eqnarray*}
\end{proof}

\begin{theorem}[A priori estimates]\label{the:a_priori_estimates}\
\begin{enumerate}
\item \label{the:a_priori_estimates:R_Q(G)}
We obtain the following identity in $R_{\IQ}(G)$
\[
\left(- \sum_{(M) \in \calm} \frac{1}{|M|}\right) \cdot  [\IQ G] + \sum_{(M) \in \calm} \IQ[G/\pi(M)] 
=  \sum_{l \ge 0} (-1)^l \cdot [\Lambda^l L \otimes_{\IZ} \IQ];
\]
\item \label{the:a_priori_estimates:G/backslashunderlineEG}
We have the following identity of integers
\begin{eqnarray*}
\sum_{(M) \in \calm} \frac{|M|-1}{|M|} = \chi(\underline{B}\Gamma) 
& = &
\dim_{\IQ}\bigl(K_0(\underline{B}\Gamma) \otimes_{\IZ} \IQ\bigr)
- \dim_{\IQ}\bigl(K_1(\underline{B}\Gamma) \otimes_{\IZ}\IQ\bigr)
\\
& = & 
\sum_{l \in \IZ} (-1)^l \cdot \rk_{\IZ}\bigl((\Lambda^{l} L)^{G}\bigr);
\end{eqnarray*}

\item \label{the:a_priori_estimates:K(Cast)}
We have the following identity of integers
\begin{eqnarray*}
\sum_{(M) \in \calm} \frac{|M|^2 -1}{|M|} 
& = & 
\dim_{\IQ}\bigl(K_0^G(L\backslash \underline{E}\Gamma))\otimes_{\IZ}   \IQ\bigr) 
- \dim_{\IQ}\bigl(K_1^G(L\backslash \underline{E}\Gamma))\otimes_{\IZ}   \IQ\bigr) 
\\
& = & 
\dim_{\IQ}\bigl(K_0(C^*_r(\Gamma))\otimes_{\IZ}   \IQ\bigr) 
- \dim_{\IQ}\bigl(K_1(C^*_r(\Gamma))\otimes_{\IZ}   \IQ\bigr).
\end{eqnarray*}
\end{enumerate}
\end{theorem}
\begin{proof}~\ref{the:a_priori_estimates:R_Q(G)}
This follows from Theorem~\ref{the:chiG(LbackslashunderlineEGamma} and
Theorem~\ref{the:rational_permutation_modules}
\\[1mm]~\ref{the:a_priori_estimates:G/backslashunderlineEG}  
This follows from Theorem~\ref{the:chiG(LbackslashunderlineEGamma}, Theorem~\ref{the:rational_permutation_modules}
and Lemma~\ref{lem:quot} as soon as we have shown
\begin{eqnarray}
\quad \quad \dim_{\IQ}\bigl(K_0(\underline{B}\Gamma) \otimes_{\IZ} \IQ\bigr)
- \dim_{\IQ}\bigl(K_1(\underline{B}\Gamma) \otimes_{\IZ}\IQ\bigr)
& = &
\sum_{l \in \IZ} (-1)^l \cdot \rk_{\IZ}\bigl((\Lambda^{l} L)^{G}\bigr).
\label{K_0-k_1_is_sum_alt_LammbdaL}
\end{eqnarray}
Since $G$ acts properly on $L\backslash \eub{\Gamma}$, there is an isomorphism of $\IQ$-modules
\[
K_i(\bub{\Gamma}) \otimes_{\IZ} \IQ \cong \bigl(K_i(L\backslash \eub{\Gamma}) \otimes_{\IZ} \IQ\bigr)^G 
\]
Now~\eqref{K_0-k_1_is_sum_alt_LammbdaL} follows from
Theorem~\ref{the:rational_permutation_modules}~\ref{the:rational_permutation_modules:iso:chi_and_h_general}.
\\[1mm]~\ref{the:a_priori_estimates:K(Cast)}  The first equation follows 
from Theorem~\ref{the:chiG(LbackslashunderlineEGamma} and
Lemma~\ref{lem:kg}. The second equation is a direct consequence of the Baum-Connes Conjecture,
which is known to be true for $\Gamma$ (see~\cite{Higson-Kasparov(2001)}), 
and from the isomorphism $K_i^{\Gamma}(\underline{E}\Gamma) 
\xrightarrow{\cong} K_i^G(L\backslash \underline{E}\Gamma)$ coming from 
the fact that $L$ acts freely on $\underline{E}\Gamma$.
\end{proof}

\begin{remark}\label{rem:expressing_sums_involving_call(M)}
Recall that we can write $\calm$ as the disjoint union
\[\calm = \coprod_{\{1\} \subsetneq C \subseteq G} \calm(C).
\]
This implies
\begin{eqnarray*}
\sum_{(M) \in \calm} \frac{1}{|M|}
& = & 
\sum_{\{1\} \subsetneq C \subseteq G} \frac{|\calm(C)|}{|C|};
\\
\sum_{(M) \in \calm} \IQ[G/\pi(M)]
& = & 
\sum_{\{1\} \subsetneq C \subseteq G} |\calm(C)| \cdot [\IQ[G/C]];
\\
\sum_{(M) \in \calm} \frac{|M| -1}{|M|} 
& = & 
\sum_{\{1\} \subsetneq C \subseteq G} |\calm(C)| \cdot \frac{|C| -1}{|C|};
\\
 \sum_{(M) \in \calm} \frac{|M|^2 -1}{|M|} 
& = & 
\sum_{\{1\} \subsetneq C \subseteq G} |\calm(C)| \cdot \frac{|C|^2 -1}{|C|}.
\end{eqnarray*}
Hence one can compute the sums of the right sides of the equations above if
one knows all the numbers $|\calm(C)|$. These have been computed explicitly
in Theorem~\ref{the:H(C;L)_in_A(G)}~\ref{the:H(C;L)_in_A(G):maximal_subgroups}.
\end{remark}


\typeout{-------------------- References -------------------------------}


\begin{thebibliography}{10}

\bibitem{Adem-Duman-Gomez(2010)}
A.~Adem, A.~N. Duman, and J.~M. Gomez.
\newblock Cohomology of toroidal orbifold quotients.
\newblock preprint, arXiv:1003.0435v3 [math.AT], 2010.

\bibitem{Adem-Ge-Pan-Petrosyan(2008)}
A.~Adem, J.~Ge, J.~Pan, and N.~Petrosyan.
\newblock Compatible actions and cohomology of crystallographic groups.
\newblock {\em J. Algebra}, 320(1):341--353, 2008.

\bibitem{Baum-Connes-Higson(1994)}
P.~Baum, A.~Connes, and N.~Higson.
\newblock Classifying space for proper actions and ${K}$-theory of group
  ${C}\sp \ast$-algebras.
\newblock In {\em $C\sp \ast$-algebras: 1943--1993 (San Antonio, TX, 1993)},
  pages 240--291. Amer. Math. Soc., Providence, RI, 1994.

\bibitem{Brown(1982)}
K.~S. Brown.
\newblock {\em Cohomology of groups}, volume~87 of {\em Graduate Texts in
  Mathematics}.
\newblock Springer-Verlag, New York, 1982.

\bibitem{Connolly-Kozniewski(1990)}
F.~X. Connolly and T.~Ko{\'z}niewski.
\newblock Rigidity and crystallographic groups. {I}.
\newblock {\em Invent. Math.}, 99(1):25--48, 1990.

\bibitem{Cuntz-Li(2009integers)}
J.~Cuntz and X.~Li.
\newblock {$C^*$}-algebras associated with integral domains and crossed
  products by actions on adele spaces.
\newblock Preprint, arXiv:0906.4903v1 [math.OA], to appear in the Journal of
  Non-Commuttaive Geometry, 2009.

\bibitem{Davis-Lueck(2010)}
J.~Davis and W.~L{\"u}ck.
\newblock The topological ${K}$-theory of certain crystallographic groups.
\newblock Preprint, arXiv:1004.2660v1, to appear in Journal of Non-Commutative
  Geometry, 2010.

\bibitem{Dold(1962)}
A.~Dold.
\newblock Relations between ordinary and extraordinary homology.
\newblock Colloq. alg. topology, Aarhus 1962, 2-9, 1962.

\bibitem{Echterhoff-Lueck-Phillips-Walters(2010)}
S.~Echterhoff, W.~L{\"u}ck, N.~C. Phillips, and S.~Walters.
\newblock The structure of crossed products of irrational rotation algebras by
  finite subgroups of {${\rm SL}_2(\Bbb Z)$}.
\newblock {\em J. Reine Angew. Math.}, 639:173--221, 2010.

\bibitem{Higson-Kasparov(2001)}
N.~Higson and G.~Kasparov.
\newblock ${E}$-theory and ${K}{K}$-theory for groups which act properly and
  isometrically on {H}ilbert space.
\newblock {\em Invent. Math.}, 144(1):23--74, 2001.

\bibitem{Langer-Lueck(2011_homology)}
M.~Langer and W.~L\"uck.
\newblock On the group cohomology of the semi-direct product {$\mathbb{Z}^n
  \rtimes_{\rho} \mathbb{Z}/m$} and a conjecture of
  {A}dem-{G}e-{P}an-{P}etrosyan.
\newblock Preprint, arXiv:1105.4772v1 [math.AT], 2011.

\bibitem{Li-Lueck(2011)}
X.~Li and W.~L\"uck.
\newblock {$K$}-theory for ring {$C^*$}-algebras -- the case of number fields
  with higher roots of unity.
\newblock in preparation, 2011.

\bibitem{Lueck(2002b)}
W.~L{\"u}ck.
\newblock Chern characters for proper equivariant homology theories and
  applications to ${K}$- and ${L}$-theory.
\newblock {\em J. Reine Angew. Math.}, 543:193--234, 2002.

\bibitem{Lueck(2005c)}
W.~L{\"u}ck.
\newblock Equivariant cohomological {C}hern characters.
\newblock {\em Internat. J. Algebra Comput.}, 15(5-6):1025--1052, 2005.

\bibitem{Lueck(2005s)}
W.~L{\"u}ck.
\newblock Survey on classifying spaces for families of subgroups.
\newblock In {\em Infinite groups: geometric, combinatorial and dynamical
  aspects}, volume 248 of {\em Progr. Math.}, pages 269--322. Birkh\"auser,
  Basel, 2005.

\bibitem{Lueck(2007)}
W.~L{\"u}ck.
\newblock Rational computations of the topological {$K$}-theory of classifying
  spaces of discrete groups.
\newblock {\em J. Reine Angew. Math.}, 611:163--187, 2007.

\bibitem{Lueck-Oliver(2001b)}
W.~L{\"u}ck and B.~Oliver.
\newblock Chern characters for the equivariant ${K}$-theory of proper
  ${G}$-{C}{W}-complexes.
\newblock In {\em Cohomological methods in homotopy theory (Bellaterra, 1998)},
  pages 217--247. Birkh\"auser, Basel, 2001.

\bibitem{Lueck-Reich(2005)}
W.~L{\"u}ck and H.~Reich.
\newblock The {B}aum-{C}onnes and the {F}arrell-{J}ones conjectures in {$K$}-
  and {$L$}-theory.
\newblock In {\em Handbook of $K$-theory. Vol. 1, 2}, pages 703--842. Springer,
  Berlin, 2005.

\bibitem{Lueck-Stamm(2000)}
W.~L{\"u}ck and R.~Stamm.
\newblock Computations of ${K}$- and ${L}$-theory of cocompact planar groups.
\newblock {\em $K$-Theory}, 21(3):249--292, 2000.

\bibitem{Lueck-Weiermann(2007)}
W.~L\"uck and M.~Weiermann.
\newblock On the classifying space of the family of virtually cyclic subgroups.
\newblock Preprintreihe SFB 478 --- Geometrische Strukturen in der Mathematik,
  Heft 453, M\"unster, arXiv:math.AT/0702646v2, to appear in the Proceedings in
  honour of Farrell and Jones in Pure and Applied Mathematic Quarterly, 2007.

\bibitem{Neukirch(1999)}
J.~Neukirch.
\newblock {\em Algebraic number theory}.
\newblock Springer-Verlag, Berlin, 1999.
\newblock Translated from the 1992 German original and with a note by Norbert
  Schappacher, With a foreword by G. Harder.

\bibitem{Serre(1977)}
J.-P. Serre.
\newblock {\em Linear representations of finite groups}.
\newblock Springer-Verlag, New York, 1977.
\newblock Translated from the second French edition by Leonard L. Scott,
  Graduate Texts in Mathematics, Vol. 42.

\bibitem{Switzer(1975)}
R.~M. Switzer.
\newblock {\em Algebraic topology---homotopy and homology}.
\newblock Springer-Verlag, New York, 1975.
\newblock Die Grundlehren der mathematischen Wissenschaften, Band 212.

\end{thebibliography}

\end{document}